\documentclass[twoside, 11pt]{amsart}

\usepackage[english]{babel}

\usepackage[T1]{fontenc}
\usepackage[lf]{Baskervaldx} 
\usepackage[bigdelims,vvarbb]{newtxmath} 
\usepackage[cal=boondoxo]{mathalfa} 

\usepackage{latexsym}
\usepackage{tensor}
\usepackage{mathrsfs}
\usepackage{times,amsmath,amsthm}
\usepackage{graphicx}
\usepackage{a4wide}
\usepackage[all]{xy}
\usepackage{paralist}
\usepackage{subfigure}
\usepackage{multicol}
\usepackage{tikz}
\usepackage{tikz-cd}
\usetikzlibrary{decorations.pathmorphing,arrows,arrows.meta,calc,shapes,shapes.geometric,decorations.pathreplacing, fit, positioning, matrix}
\tikzcdset{arrow style=tikz, diagrams={>=stealth}}
\usepackage{comment}
\tikzset{
  optree/.style={scale=.5,thick,grow'=up,level distance=10mm,inner sep=1pt},
  comp/.style={draw=none,circle,fill,line width=0,inner sep=0pt},
  dot/.style={draw,circle,fill,inner sep=0pt,minimum width=3pt},
  circ/.style={draw,circle,inner sep=1pt,minimum width=4mm},
  emptycirc/.style={draw,circle,inner sep=1pt,minimum width=2mm},
  root/.style={level distance=10mm,inner sep=1pt},
  leaf/.style={draw=none,circle,fill,line width=0,inner sep=0pt},
  nodot/.style={draw,circle,inner sep=1pt},
}

\usepackage{color}
\definecolor{Chocolat}{rgb}{0.36, 0.2, 0.09}
\definecolor{BleuTresFonce}{rgb}{0.215, 0.215, 0.36}

\usepackage[colorlinks,final,hyperindex, ]{hyperref}
\hypersetup{citecolor=BleuTresFonce, linkcolor=Chocolat, urlcolor  = BleuTresFonce}
\usepackage[noabbrev,capitalize]{cleveref}

\usepackage[normalem]{ulem}

\usepackage{todonotes}

\let\oldtocsection=\tocsection
\let\oldtocsubsection=\tocsubsection

\renewcommand{\tocsection}[2]{\hspace{0em}\vspace{0.1em}\rule{0pt}{14pt}\oldtocsection{#1}{#2}\bf}
\renewcommand{\tocsubsection}[2]{\hspace{2em}\oldtocsubsection{#1}{#2}}

\def\Liegra{\mathrm{Lie}\textrm{-}\mathrm{gra}}
\def\Liemgra{\mathrm{Lie}\textrm{-}\mathrm{mgra}}
\def\Liencgra{\mathrm{Lie}\textrm{-}\mathrm{ncgra}}
\def\Com{\mathrm{Com}}
\def\c{\mathrm{c}}
\def\ln{\mathrm{log}}
\def\k{\mathbb{k}}

\def\F{\mathrm{F}}
\def\d{\mathrm{d}}
\def\P{\mathcal{P}}
\def\C{\mathcal{C}}
\def\oC{\overline{\mathcal{C}}}
\def\1{\mathbb{1}}
\def\I{\mathcal{I}}
\def\id{\mathrm{id}}
\def\End{\mathrm{End}}
\def\eend{\mathrm{end}}
\def\gra{\mathrm{g}}
\newcommand{\Sy}{\mathbb{S}}
\def\dcGra{\mathsf{dsGra}}
\def\dmGra{\mathsf{dmGra}}
\def\dsncGra{\mathsf{dsncGra}}
\def\ldcGra{2\textsf{-}\mathsf{dsGra}}
\def\ldcncGra{2\textsf{-}\mathsf{dsncGra}}
\def\3ldcncGra{3\textsf{-}\mathsf{dsncGra}}
\def\multildcGra#1{#1\textsf{-}\mathsf{dsGra}}
\def\kdcGra{\k\mathsf{dsGra}}
\def\btGra{{\Join}\textsf{-}{\dcGra}_{\Sy}}
\def\Lieadm{\mathrm{Lie}\textrm{-}\mathrm{adm}}
\def\Prelie{\mathrm{pre}\textrm{-}\mathrm{Lie}}
\def\Lie{\mathrm{Lie}}
\def\ibt{\underset{\scriptscriptstyle (1,1)}{\boxtimes}}
\newcommand{\ZZ}{\mathbb{Z}}
\newcommand{\NN}{\mathbb{N}}
\newcommand{\ac}{\scriptstyle \text{\rm !`}}
\def\G{\mathfrak{G}}
\def\g{\mathfrak{g}}
\def\a{\mathfrak{a}}
\def\Hom{\mathrm{Hom}}
\newcommand{\BCH}{\mathrm{BCH}}
\newcommand{\Sbimod}{\Sy\mbox{-}\mathsf{bimod}}
\newcommand{\dcGraph}{\mathsf{dcGra}}
\def\trunc{\mathrm{barb}}
\def\barb{\mathrm{barb}}

\def\cc{\circledcirc}
\def\ad{\mathrm{ad}}
\def\MC{\mathrm{MC}}
\def\pap#1#2#3{#1\stackrel{#2}{\vcenter{\hbox{\text{\scalebox{1.2}{$\Join$}}}}}#3}
\def\ba{\bar{\alpha}}
\def\QQ{\mathbb{Q}}
\def\RR{\mathbb{R}}

\def\Aut{\mathrm{Aut}}

\numberwithin{equation}{section}
\theoremstyle{plain}
\newtheorem{proposition}[equation]{Proposition}
\newtheorem{theorem}[equation]{Theorem}

\newtheorem{lemma}[equation]{Lemma}

\theoremstyle{definition}
\newtheorem{definition}[equation]{Definition}

\newtheorem{remark}[equation]{\sc Remark}
\newtheorem{example}[equation]{\sc Example}
\newtheorem{cexample}[equation]{\sc Counterexample}
\newtheorem{examples}[equation]{\sc Examples}



\title{Integration theory of Lie-graph algebras}

\date{\today}

\author{Ricardo Campos}
\address{Universit\'e de Toulouse, Institut de Math\'ematiques de Toulouse, CNRS, UMR 5219,  Toulouse, France}
\email{ricardo.campos@math.univ-toulouse.fr}

\author{Bruno Vallette}
\address{Universit\'e Sorbonne Paris Nord, Laboratoire de G\'eom\'etrie, Analyse et Applications, CNRS, UMR 7539, Villetaneuse, France}
\email{vallette@math.univ-paris13.fr}

\thanks{
\newline
The authors are supported by ANR-20-CE40-0016 HighAGT}

\begin{document}

\maketitle

\begin{abstract}
We develop the Lie theory of Lie-admissible algebras whose product is enriched with higher operations modeled on directed graphs with a view to apply it to  the deformation theories controlled by this kind of Lie algebras. 
We produce effective formulas for their exponential map, their gauge group structure and the action on Maurer--Cartan elements. 
The main motivation and range of applications lies in the deformation theory of types of bialgebras which is done in a sequel article. 
This work extends the case of pre-Lie algebra structures which appear in the deformation theory of operadic algebras. 
\end{abstract}

\tableofcontents

\section*{Introduction}

\paragraph*{\bf Integration theory of {properadic deformation}  Lie algebras} 
Lie's third theorem tells us how to integrate the infinitesimal symmetries structured as a Lie algebra into global symmetries composing a group. This integration process plays a key role in deformation theories that are controlled by differential graded Lie algebras where this group, called the \emph{gauge group}, encodes the possible equivalences between the Maurer--Cartan elements. These data form the so-called \emph{Deligne groupoid}.

\medskip

The integration of differential graded Lie algebras arising from deformation complexes of properadic structures has remained an open problem for several years, despite the growing interest in 
{algebraic structures having operations with several inputs and outputs \cite{MerkulovVallette09I, MerkulovVallette09II, CFL15, KS18, HLV19, LV22, KTV25}.}
This stands in stark contrast to the operadic setting, where an additional \emph{pre-Lie structure} on the deformation complex allows one to produce \emph{effective} formulas for the Deligne groupoid of the governing differential graded Lie algebra \cite{DotsenkoShadrinVallette16, DotsenkoShadrinVallette22}.

\medskip

In this paper we address this gap {in two steps. First, we introduce a new type of algebraic structure, called 
\emph{Lie-graph algebras}, which is parametrised by directed graphs and which induces a Lie bracket. Then, we settle the complete integration theory for this new class of algebras of Lie-type.
Unlike the associative and pre-Lie cases, where the {entire} structure can be {recovered} by iterating a single generating operation, the operad governing Lie-graph algebras is not finitely generated (\cref{prop:nofinitepresentation}): one must take \emph{all} operations into account simultaneously.  
This lack of finite presentability makes both the theoretical analysis and the explicit computation of their Deligne groupoids substantially more challenging.

\medskip

Our approach exploits the combinatorics of the underlying graph structure to produce explicit formulas for the gauge group of a Lie-graph algebra and for its gauge action on the Maurer--Cartan {elements}.  
This framework applies in particular to the deformation complexes of gebras over properads, for which we obtain a complete and effective description of the associated Deligne groupoid, see \cite[Section~1]{CV25II}.

\medskip

\paragraph*{\bf Context and historical overview} A classical heuristic in deformation theory, dating back to Deligne, Drinfeld, Gerstenhaber, Quillen, and others \cite{Drinfeld14}, states that deformation problems over a field characteristic zero are governed by differential graded Lie algebras. This heuristic has been formalised via the equivalence {of $\infty$-categories} between {differential graded} Lie algebras and formal moduli problems by Lurie and Pridham \cite{Lurie11, pridham2010unifying}, see also the recent operadic proof providing point-set models by Le Grignou--Roca i Lucio \cite{LGRC23}.   

\medskip

Under this correspondence, given a {differential graded} Lie algebra $(\g, d , [\,,])$, the {set} of objects to be deformed is given by the set of Maurer--Cartan elements of $\g$
\[ {\MC(\g)\coloneq \left\{  \alpha \in \g_{-1}\, | \, d\alpha + \tfrac12 [\alpha, \alpha]=0\right\}~, }\] while equivalent deformations are detected by the action of the gauge group $\Gamma$ on $\MC(\g)$, which arises as 
the integration of the {degree 0 Lie sub-algebra} {in}to a Lie group by means of exponentiation. The data of {the gauge group} $\G$ (morphisms) together with its action on {the set of Maurer--Cartan elements} $\MC(\g)$ (objects)
{forms} the \emph{Deligne groupoid} of $\g$.\\

While one can describe the Deligne groupoid in theoretical terms, in practice, for an arbitrary {differential graded Lie algebra} $\g$, one cannot provide very effective formulas for either the gauge group or its action since both rely on the Baker--Campbell--Hausdorff formula \cite{BF12}.
One case where the Deligne groupoid can be made fully explicit is in the situation where the Lie bracket on $\g$ actually arises as the {skew}-symmetrisation of an associative product $\star$, {that is} $[x,y] = x\star y - y\star x$. In this situation, the gauge group is actually obtained by the classical exponentiation of the elements of $\g$, 
\[\exp(x) = 1 + x + \tfrac{1}{2!}x\star x+\tfrac{1}{3!}x\star x\star x +\dots \in \G=\exp(\g_0)~.\]

\medskip

In fact,  it {happens frequently in the context of deformations of algebraic structures that the Lie bracket is obtained by skew-symmetrisation}  of a {binary} product $\star$, {but} which is not necessarily associative.
In this case, the Maurer--Cartan equation takes up the particularly nice form
\[d\alpha+\alpha\star \alpha = 0~.\] 
 For instance, the deformation complex of {homotopy} $\P$-algebra structures, {for a Koszul operad $\P$}, on a fixed chain complex, {carries} the structure of a pre-Lie algebra, see \cite{DotsenkoShadrinVallette16}. 
The exponential analogy turns out to be fruitful: in \emph{loc. cit.}, the second author, together with Dotsenko and Shadrin, {used} the concept of \emph{pre-Lie exponentials} \cite{AgrachevGamkrelidze80}, which allowed them to address in an {effective} way the integration {theory} of pre-Lie algebras which has since then been applied {successfully} to 
rational homotopy theory~\cite{campos2019lie}, 
algebraic topology and algebraic geometry \cite{Emprin24}, 
numerical analysis \cite{BLBM23}, and renormalisation theory \cite{BHZ19, BD24}.

\medskip

\paragraph*{\bf Lie-graph algebras and their exponential groups}

In the properadic context, one can identify {a binary product} 
whose {skew}-symmetrisation yields the Lie bracket of interest. {However, this product is just Lie-admissible and not pre-Lie in general.}
It is therefore natural to ask how much further can the exponentiation techniques be pushed.
One first key observation is that, contrary to the pre-Lie case, all the possible iterations of this Lie-admissible product are not enough to produce a meaningful exponential map.
So the first main insight of this paper is to identify additional structural operations and to encode them in a new type of algebraic structure:}
{we introduce the notion of a \emph{Lie-graph algebra},
which is made up of operations parametrised by directed simple graphs.}
{This new type of algebras} is governed by an operad that we denote by  $\Liegra$, of which the operad $\Lie$ (and even $\Prelie$) is a quotient. The crucial result from Sections \ref{subsec:LiegraConv} and \ref{subsec:complete} is that \emph{the deformation complexes of properadic structures carry a canonical Lie-graph algebra structure}.

\medskip

{This range of applications motivates the development of an effective integration theory of Lie-graph algebras.
This forms the first main result of this paper: an explicit description of the gauge group of Lie algebras arising
from Lie-graph algebras via \emph{Lie-graph exponential and logarithm maps}.}

\medskip

\textbf{Theorem \ref{thm:ExpIsoMorph}.}
\emph{
	For any 
	complete Lie-graph algebra $\g_0$, there are isomorphism of {complete} topological groups between the gauge group and the {so-called} deformation gauge group:
	\[\exp \ : \Gamma=\big( \g_0,  \BCH, 0\big)
	\stackrel{\cong}{\leftrightarrows}
	\left( \1 + \g_0, \circledcirc, \1\right)=\G\ : \ \ln
	\ ,\]
where both the exponential map $\exp$ and the associative product $\circledcirc$ are explicitly given as sums of {directed (respectively $2$-leveled) graphs}.
}
\[\exp(x)\ =\ \1 \  + \
	\vcenter{\hbox{\begin{tikzpicture}[scale=0.4]
				\draw[thick]	(0,0)--(1,0)--(1,1)--(0,1)--cycle;
				\draw (0.5,0.5) node {{$x$}} ; 
	\end{tikzpicture}}}
	\ + \
	\scalebox{1.1}{$\frac{1}{2}$}\ 
	\vcenter{\hbox{\begin{tikzpicture}[scale=0.4]
				\draw[thick]	(0,0)--(1,0)--(1,1)--(0,1)--cycle;
				\draw[thick]	(0,2)--(1,2)--(1,3)--(0,3)--cycle;
				\draw[thick]	(0.5,1)--(0.5,2);
				\draw (0.5,0.5) node {{$x$}} ; 
				\draw (0.5,2.5) node {{$x$}} ; 
	\end{tikzpicture}}}
	\ + \
	\scalebox{1.1}{$\frac{1}{6}$}\ 
	\vcenter{\hbox{\begin{tikzpicture}[scale=0.4]
				\draw[thick]	(0,0)--(1,0)--(1,1)--(0,1)--cycle;
				\draw[thick]	(2,0)--(3,0)--(3,1)--(2,1)--cycle;
				\draw[thick]	(1,2)--(2,2)--(2,3)--(1,3)--cycle;
				\draw[thick]	(0.5,1)--(1.33,2);
				\draw[thick]	(2.5,1)--(1.66,2);
				\draw (1.5,2.5) node {{$x$}} ; 
				\draw (0.5,0.5) node {{$x$}} ; 
				\draw (2.5,0.5) node {{$x$}} ; 	
	\end{tikzpicture}}}
	\ + \ 
	\scalebox{1.1}{$\frac{1}{6}$}\ 
	\vcenter{\hbox{\begin{tikzpicture}[scale=0.4]
				\draw[thick]	(0,2)--(1,2)--(1,3)--(0,3)--cycle;
				\draw[thick]	(2,2)--(3,2)--(3,3)--(2,3)--cycle;
				\draw[thick]	(1,0)--(2,0)--(2,1)--(1,1)--cycle;
				\draw[thick]	(0.5,2)--(1.33,1);
				\draw[thick]	(2.5,2)--(1.66,1);
				\draw (1.5,0.5) node {{$x$}} ; 
				\draw (0.5,2.5) node {{$x$}} ; 
				\draw (2.5,2.5) node {{$x$}} ; 	
	\end{tikzpicture}}}
	\ + \ 
	\scalebox{1.1}{$\frac{1}{6}$}\ 
	\vcenter{\hbox{\begin{tikzpicture}[scale=0.4]
				\draw[thick]	(0,0)--(1,0)--(1,1)--(0,1)--cycle;
				\draw[thick]	(0,2)--(1,2)--(1,3)--(0,3)--cycle;
				\draw[thick]	(0,4)--(1,4)--(1,5)--(0,5)--cycle;
				\draw[thick]	(0.5,1)--(0.5,2);
				\draw[thick]	(0.5,3)--(0.5,4);	
				\draw (0.5,0.5) node {{$x$}} ; 
				\draw (0.5,2.5) node {{$x$}} ; 
				\draw (0.5,4.5) node {{$x$}} ; 
	\end{tikzpicture}}}
	\ + \ 
	\scalebox{1.1}{$\frac{1}{6}$}\ 
	\vcenter{\hbox{\begin{tikzpicture}[scale=0.4]
				\draw[thick]	(0,0)--(1,0)--(1,1)--(0,1)--cycle;
				\draw[thick]	(0,2)--(1,2)--(1,3)--(0,3)--cycle;
				\draw[thick]	(0,4)--(1,4)--(1,5)--(0,5)--cycle;
				\draw[thick]	(0.33,1)--(0.33,2);
				\draw[thick]	(0.33,3)--(0.33,4);
				\draw[thick]	(0.66,1) to[out=60,in=270] (1.5,2.5) to[out=90,in=300] (0.66,4);	
				\draw (0.5,0.5) node {{$x$}} ; 
				\draw (0.5,2.5) node {{$x$}} ; 
				\draw (0.5,4.5) node {{$x$}} ; 
	\end{tikzpicture}}}
	\ + \
	\scalebox{1.1}{$\frac{1}{8}$}\ 
	\vcenter{\hbox{\begin{tikzpicture}[scale=0.4]
				\draw[thick]	(0,0)--(1,0)--(1,1)--(0,1)--cycle;
				\draw[thick]	(2,0)--(3,0)--(3,1)--(2,1)--cycle;
				\draw[thick]	(1,2)--(2,2)--(2,3)--(1,3)--cycle;
				\draw[thick]	(0,-2)--(1,-2)--(1,-1)--(0,-1)--cycle;	
				\draw[thick]	(0.5,1)--(1.33,2);
				\draw[thick]	(2.5,1)--(1.66,2);
				\draw[thick]	(0.5,0)--(0.5,-1);	
				\draw (0.5,-1.5) node {{$x$}} ; 
				\draw (1.5,2.5) node {{$x$}} ; 
				\draw (0.5,0.5) node {{$x$}} ; 
				\draw (2.5,0.5) node {{$x$}} ; 	
	\end{tikzpicture}}}
	\ + \ 
	\scalebox{1.1}{$\frac{1}{24}$}\ 
	\vcenter{\hbox{\begin{tikzpicture}[scale=0.4]
				\draw[thick]	(0.66,2)--(2.33,1);	
				\draw[draw=white,double=black,double distance=2*\pgflinewidth,thick]	(0.66,1)--(2.33,2);		
				\draw[thick]	(0,0)--(1,0)--(1,1)--(0,1)--cycle;
				\draw[thick]	(2,0)--(3,0)--(3,1)--(2,1)--cycle;
				\draw[thick]	(0,2)--(1,2)--(1,3)--(0,3)--cycle;
				\draw[thick]	(2,2)--(3,2)--(3,3)--(2,3)--cycle;	
				\draw[thick]	(0.33,1)--(0.33,2);
				\draw[thick]	(2.66,1)--(2.66,2);
				\draw (0.5,2.5) node {{$x$}} ; 
				\draw (2.5,2.5) node {{$x$}} ; 	
				\draw (0.5,0.5) node {{$x$}} ; 
				\draw (2.5,0.5) node {{$x$}} ; 	
	\end{tikzpicture}}}
	\ + \ \cdots\]
	\[(\1+x)\circledcirc (\1+y)\ =\ \1 \  + \
	\vcenter{\hbox{\begin{tikzpicture}[scale=0.4]
				\draw[thick]	(0,0)--(1,0)--(1,1)--(0,1)--cycle;
				\draw (0.5,0.5) node {{$x$}} ; 
	\end{tikzpicture}}}
	\ + \
	\vcenter{\hbox{\begin{tikzpicture}[scale=0.4]
				\draw[thick]	(0,0)--(1,0)--(1,1)--(0,1)--cycle;
				\draw (0.5,0.5) node {{$y$}} ; 
	\end{tikzpicture}}}
	\ + \
	\vcenter{\hbox{\begin{tikzpicture}[scale=0.4]
				\draw[thick]	(0,0)--(1,0)--(1,1)--(0,1)--cycle;
				\draw[thick]	(0,2)--(1,2)--(1,3)--(0,3)--cycle;
				\draw[thick]	(0.5,1)--(0.5,2);
				\draw (0.5,0.5) node {{$x$}} ; 
				\draw (0.5,2.5) node {{$y$}} ; 
	\end{tikzpicture}}}
	\ + \
	\scalebox{1.1}{$\frac{1}{2}$}\ 
	\vcenter{\hbox{\begin{tikzpicture}[scale=0.4]
				\draw[thick]	(0,0)--(1,0)--(1,1)--(0,1)--cycle;
				\draw[thick]	(2,0)--(3,0)--(3,1)--(2,1)--cycle;
				\draw[thick]	(1,2)--(2,2)--(2,3)--(1,3)--cycle;
				\draw[thick]	(0.5,1)--(1.33,2);
				\draw[thick]	(2.5,1)--(1.66,2);
				\draw (1.5,2.5) node {{$y$}} ; 
				\draw (0.5,0.5) node {{$x$}} ; 
				\draw (2.5,0.5) node {{$x$}} ; 	
	\end{tikzpicture}}}
	\ + \ 
	\scalebox{1.1}{$\frac{1}{2}$}\ 
	\vcenter{\hbox{\begin{tikzpicture}[scale=0.4]
				\draw[thick]	(0,2)--(1,2)--(1,3)--(0,3)--cycle;
				\draw[thick]	(2,2)--(3,2)--(3,3)--(2,3)--cycle;
				\draw[thick]	(1,0)--(2,0)--(2,1)--(1,1)--cycle;
				\draw[thick]	(0.5,2)--(1.33,1);
				\draw[thick]	(2.5,2)--(1.66,1);
				\draw (1.5,0.5) node {{$x$}} ; 
				\draw (0.5,2.5) node {{$y$}} ; 
				\draw (2.5,2.5) node {{$y$}} ; 	
	\end{tikzpicture}}}
	\ + \ 
	\scalebox{1.1}{$\frac{1}{4}$}\ 
	\vcenter{\hbox{\begin{tikzpicture}[scale=0.4]
				\draw[thick]	(0.66,2)--(2.33,1);	
				\draw[draw=white,double=black,double distance=2*\pgflinewidth,thick]	(0.66,1)--(2.33,2);		
				\draw[thick]	(0,0)--(1,0)--(1,1)--(0,1)--cycle;
				\draw[thick]	(2,0)--(3,0)--(3,1)--(2,1)--cycle;
				\draw[thick]	(0,2)--(1,2)--(1,3)--(0,3)--cycle;
				\draw[thick]	(2,2)--(3,2)--(3,3)--(2,3)--cycle;	
				\draw[thick]	(0.33,1)--(0.33,2);
				\draw[thick]	(2.66,1)--(2.66,2);
				\draw (0.5,2.5) node {{$y$}} ; 
				\draw (2.5,2.5) node {{$y$}} ; 	
				\draw (0.5,0.5) node {{$x$}} ; 
				\draw (2.5,0.5) node {{$x$}} ; 	
	\end{tikzpicture}}}
	\ + \ \cdots\]

\medskip

While such a result might look similar in form to analogous classical one for associative and even  pre-Lie algebras, the 
{conceptual approach leading to these} 
 maps is very different. In the classical associative and pre-Lie cases, the exponential maps were obtained by summing over repeated iterations of the generating binary operation $\star$ divided by some factor. On the other hand, in the Lie-graph setting, we {have to} sum over \emph{all possible graphical operations}, divided by an appropriate combinatorial factor, see Definition \ref{def:exp}. {There is} a way to interpret this {discrepancy:} while the associative and pre-Lie operads are generated by their respective binary operations, that is not the case for the operad $\Liegra$, {which is} not even finitely generated (Proposition~\ref{prop:nofinitepresentation}). Retrospectively, one can get the associative and pre-Lie exponential maps by the same sum over all the elements of the respective operads.
 
 \medskip
 
 {Pursuing  this direction, we are able to provide the literature with an effective formula for the action of the (deformation) gauge group on Maurer--Cartan elements. This represents the second main result of the present paper.}

\medskip

\textbf{Theorem \ref{thm:Action}.}
\emph{
	For any complete Lie-graph algebra $\g$, the action of the (deformation) gauge group on the set of Maurer--Cartan elements $\MC(\g)$ {is given by} the explicit formula
	\[\exp(\lambda)\cdot \alpha\coloneq \lambda.\alpha =
	\pap{\exp(\lambda)}{\alpha}{\exp(-\lambda)} \ ,\]
	where 
		\begin{align*}
			\pap{(\1+x)}{\alpha}{(\1+y)}\ = \ &
			\vcenter{\hbox{\begin{tikzpicture}[scale=0.4]
						\draw[thick]	(0,0)--(1,0)--(1,1)--(0,1)--cycle;
						\draw (0.5,0.5) node {{$\alpha$}} ; 
			\end{tikzpicture}}}
			\ + \
			\vcenter{\hbox{\begin{tikzpicture}[scale=0.4]
						\draw[thick]	(0,0)--(1,0)--(1,1)--(0,1)--cycle;
						\draw[thick]	(0,2)--(1,2)--(1,3)--(0,3)--cycle;
						\draw[thick]	(0.5,1)--(0.5,2);
						\draw (0.5,0.5) node {{$x$}} ; 
						\draw (0.5,2.5) node {{$\alpha$}} ; 
			\end{tikzpicture}}}
			\ + \
			\vcenter{\hbox{\begin{tikzpicture}[scale=0.4]
						\draw[thick]	(0,0)--(1,0)--(1,1)--(0,1)--cycle;
						\draw[thick]	(0,2)--(1,2)--(1,3)--(0,3)--cycle;
						\draw[thick]	(0.5,1)--(0.5,2);
						\draw (0.5,0.5) node {{$\alpha$}} ; 
						\draw (0.5,2.5) node {{$y$}} ; 
			\end{tikzpicture}}}
			\ + \
			\scalebox{1.1}{$\frac{1}{2}$}\ 
			\vcenter{\hbox{\begin{tikzpicture}[scale=0.4]
						\draw[thick]	(0,0)--(1,0)--(1,1)--(0,1)--cycle;
						\draw[thick]	(2,0)--(3,0)--(3,1)--(2,1)--cycle;
						\draw[thick]	(1,2)--(2,2)--(2,3)--(1,3)--cycle;
						\draw[thick]	(0.5,1)--(1.33,2);
						\draw[thick]	(2.5,1)--(1.66,2);
						\draw (1.5,2.5) node {{$\alpha$}} ; 
						\draw (0.5,0.5) node {{$x$}} ; 
						\draw (2.5,0.5) node {{$x$}} ; 	
			\end{tikzpicture}}}
			\ + \ 
			\scalebox{1.1}{$\frac{1}{2}$}\ 
			\vcenter{\hbox{\begin{tikzpicture}[scale=0.4]
						\draw[thick]	(0,2)--(1,2)--(1,3)--(0,3)--cycle;
						\draw[thick]	(2,2)--(3,2)--(3,3)--(2,3)--cycle;
						\draw[thick]	(1,0)--(2,0)--(2,1)--(1,1)--cycle;
						\draw[thick]	(0.5,2)--(1.33,1);
						\draw[thick]	(2.5,2)--(1.66,1);
						\draw (1.5,0.5) node {{$\alpha$}} ; 
						\draw (0.5,2.5) node {{$y$}} ; 
						\draw (2.5,2.5) node {{$y$}} ; 	
			\end{tikzpicture}}}
			\ + \ \\&
			\vcenter{\hbox{\begin{tikzpicture}[scale=0.4]
						\draw[thick]	(0,0)--(1,0)--(1,1)--(0,1)--cycle;
						\draw[thick]	(0,2)--(1,2)--(1,3)--(0,3)--cycle;
						\draw[thick]	(0,4)--(1,4)--(1,5)--(0,5)--cycle;
						\draw[thick]	(0.5,1)--(0.5,2);
						\draw[thick]	(0.5,3)--(0.5,4);	
						\draw (0.5,0.5) node {{$x$}} ; 
						\draw (0.5,2.5) node {{$\alpha$}} ; 
						\draw (0.5,4.5) node {{$y$}} ; 
			\end{tikzpicture}}}
			\ + \  
			\vcenter{\hbox{\begin{tikzpicture}[scale=0.4]
						\draw[thick]	(0,0)--(1,0)--(1,1)--(0,1)--cycle;
						\draw[thick]	(0,2)--(1,2)--(1,3)--(0,3)--cycle;
						\draw[thick]	(0,4)--(1,4)--(1,5)--(0,5)--cycle;
						\draw[thick]	(0.33,3)--(0.33,4);
						\draw[thick]	(0.66,1) to[out=60,in=270] (1.5,2.5) to[out=90,in=300] (0.66,4);	
						\draw (0.5,0.5) node {{$x$}} ; 
						\draw (0.5,2.5) node {{$\alpha$}} ; 
						\draw (0.5,4.5) node {{$y$}} ; 
			\end{tikzpicture}}}
			\ + \
			\vcenter{\hbox{\begin{tikzpicture}[scale=0.4]
						\draw[thick]	(0,0)--(1,0)--(1,1)--(0,1)--cycle;
						\draw[thick]	(0,2)--(1,2)--(1,3)--(0,3)--cycle;
						\draw[thick]	(0,4)--(1,4)--(1,5)--(0,5)--cycle;
						\draw[thick]	(0.33,1)--(0.33,2);
						\draw[thick]	(0.66,1) to[out=60,in=270] (1.5,2.5) to[out=90,in=300] (0.66,4);	
						\draw (0.5,0.5) node {{$x$}} ; 
						\draw (0.5,2.5) node {{$\alpha$}} ; 
						\draw (0.5,4.5) node {{$y$}} ; 
			\end{tikzpicture}}}
			\ + \  
			\vcenter{\hbox{\begin{tikzpicture}[scale=0.4]
						\draw[thick]	(0,0)--(1,0)--(1,1)--(0,1)--cycle;
						\draw[thick]	(0,2)--(1,2)--(1,3)--(0,3)--cycle;
						\draw[thick]	(0,4)--(1,4)--(1,5)--(0,5)--cycle;
						\draw[thick]	(0.33,1)--(0.33,2);
						\draw[thick]	(0.33,3)--(0.33,4);
						\draw[thick]	(0.66,1) to[out=60,in=270] (1.5,2.5) to[out=90,in=300] (0.66,4);	
						\draw (0.5,0.5) node {{$x$}} ; 
						\draw (0.5,2.5) node {{$\alpha$}} ; 
						\draw (0.5,4.5) node {{$y$}} ; 
			\end{tikzpicture}}}
			\ + \ 
			\scalebox{1.1}{$\frac{1}{2}$}\ 
			\vcenter{\hbox{\begin{tikzpicture}[scale=0.4]
						\draw[thick]	(0,0)--(1,0)--(1,1)--(0,1)--cycle;
						\draw[thick]	(2,0)--(3,0)--(3,1)--(2,1)--cycle;
						\draw[thick]	(1,2)--(2,2)--(2,3)--(1,3)--cycle;
						\draw[thick]	(1,4)--(2,4)--(2,5)--(1,5)--cycle;	
						\draw[thick]	(0.5,1)--(1.33,2);
						\draw[thick]	(2.5,1)--(1.66,2);
						\draw[thick]	(1.5,3)--(1.5,4);		
						\draw (1.5,2.5) node {{$\alpha$}} ; 
						\draw (0.5,0.5) node {{$x$}} ; 
						\draw (2.5,0.5) node {{$x$}} ; 	
						\draw (1.5,4.5) node {{$y$}} ; 	
			\end{tikzpicture}}}
			\ + \ 
			\vcenter{\hbox{\begin{tikzpicture}[scale=0.4]
						\draw[thick] (0.66,1)--(1.33,2);
						\draw[thick] (2.33,1)--(1.66,2);
						\draw[thick] (2.66,1) to (2.66,2.5) to[out=90,in=300]  (1.66,4);
						\draw[thick]	(0,0)--(1,0)--(1,1)--(0,1)--cycle;
						\draw[thick]	(2,0)--(3,0)--(3,1)--(2,1)--cycle;
						\draw[thick]	(1,2)--(2,2)--(2,3)--(1,3)--cycle;
						\draw[thick]	(1,4)--(2,4)--(2,5)--(1,5)--cycle;	
						\draw[thick]	(1.5,3)--(1.5,4);	
						\draw (1.5,4.5) node {{$y$}} ; 
						\draw (1.5,2.5) node {{$\alpha$}} ; 
						\draw (0.5,0.5) node {{$x$}} ; 
						\draw (2.5,0.5) node {{$x$}} ; 	
			\end{tikzpicture}}}
			\ + \ \cdots
		\end{align*}
{is the sum of \emph{bow-tie} graphs, that is $3$-leveled directed graphs with only one middle vertex labeled by the Maurer--Cartan element $\alpha$}~.
}

\medskip		

 Once restricted to ladders and to rooted trees, this formula recovers respectively the associative and the pre-Lie cases but in a more conceptual form. The methods used in the various proofs here are worth mentioning since they are based on the classical real and finite dimensional Lie theory, unlike the previous works on this domain. In order to settle these new universal formulas, we consider the free Lie-graph algebra over the field of rational numbers, naturally provided by the operadic calculus. We embed it into the one with real coefficients that we truncate under its canonical filtration in order to obtain a finite dimensional real Lie algebra.
Finally, in order to prove the last aforementioned theorem, we give a new conceptual treatment of the \emph{differential trick}, which is a method to internalise the differential map into a given algebraic structure. 

\medskip

{
\paragraph*{\bf Applications} 
In a sequel article \cite{CV25II}, we apply this new universal integration theory to the 
differential graded Lie algebras coming from properadic deformation theory. 
We describe there the Deligne groupoid. We give a gauge theoretical explanation of the homotopy transfer theorem, emphasizing its perturbative nature. And, we generalise the Koszul hierarchy and the twisting procedure from the operadic case to the properadic case. 
Theses results apply to a wide range of recently studied properadic structures appearing in algebra, geometry, topology, and mathematical physics.

\medskip

\paragraph*{\bf Conventions} 
For a uniform and light presentation, we work over a ground field $\k$ of characteristic $0$ though many results hold under weaker hypotheses. 
The underlying category is that of differential $\ZZ$-graded $\k$-vector spaces, using the homological degree convention, for which differentials have degree $-1$.  
The symmetric groups are denoted by $\Sy_n$.
We use the conventions of \cite{MerkulovVallette09I, MerkulovVallette09II, HLV19} for  properads: for instance, we draw graphs with inputs at the top and outputs at the bottom. All the graphs present here will be directed by a global flow going from top to bottom.

\medskip

\paragraph*{\bf Acknowledgements} We would like to express our appreciation to 
Ga\"etan Borot, Fr\'ed\'eric Chapoton, Vladimir Dotsenko, Coline Emprin, Johan Leray, and Thomas Willwacher 
for interesting and helpful discussions. The second author is grateful to the IH\'ES for the long term invitation and the excellent working conditions. 

\section{Algebraic structures on convolution algebras}

In the first subsection, we start by recalling the main elements of the properadic calculus. 
After that, we introduce a new algebraic structure, dubbed \emph{Lie-graph algebra} for its connection with Lie algebras, which encodes all the natural operations acting on the totalisation of the underlying space of a properad. 
In the last subsection, we introduce compatible complete topologies which will allow us to develop the integration theory of 
Lie-graph algebras in the next section. 

\subsection{Properadic calculus} 
Recall that an $\Sy$-bimodule is a collection $\{\mathcal{M}(m,n)\}_{m,n\in \NN}$ of right $\Sy_m^{\mathrm{op}}\times \Sy_n$-modules; they are meant to encode operations with $n$ inputs and $m$ outputs. 
We consider the monoidal category 
$\big(\Sbimod,\allowbreak \boxtimes, \allowbreak \I\big)$ of $\Sy$-bimodules  equipped with the composition product based on 2-leveled connected graphs directed by a global flow, see for instance \cite[Section~1]{HLV19} for more details. 

\begin{figure*}[h]
	\begin{tikzpicture}[scale=0.8]
		\draw[thick]	(0.8,2)--(2.4,1);	
		\draw[thick]	(0.6,2)--(2.2,1);	
		\draw[draw=white,double=black,double distance=2*\pgflinewidth,thick]	(0.66,1)--(2.33,2);		
		\draw[thick]	(0,0)--(1,0)--(1,1)--(0,1)--cycle;
		\draw[thick]	(2,0)--(3,0)--(3,1)--(2,1)--cycle;
		\draw[thick]	(0,2)--(1,2)--(1,3)--(0,3)--cycle;
		\draw[thick]	(2,2)--(3,2)--(3,3)--(2,3)--cycle;	
		\draw[thick]	(0.25,1)--(0.25,2);
		\draw[thick]	(0.4,1)--(0.4,2);
		\draw[thick]	(2.66,1)--(2.66,2);
		\draw[thick]	(0.33,3)--(0.33,3.5);	
		\draw[thick]	(0.66,3)--(0.66,3.5);		
		\draw[thick]	(2.5,3)--(2.5,3.5);		
		\draw[thick]	(2.25,0)--(2.25,-0.5);				
		\draw[thick]	(0.5,0)--(0.5,-0.5);				
		\draw[thick]	(2.75,0)--(2.75,-0.5);						
		\draw[thick]	(2.5,0)--(2.5,-0.5);							
		\draw (0.5,2.5) node {{$1$}} ; 
		\draw (2.5,2.5) node {{$2$}} ; 	
		\draw (0.5,0.5) node {{$3$}} ; 
		\draw (2.5,0.5) node {{$4$}} ; 	
		\draw (0.33,3.5) node[above] {{\tiny$1$}} ; 		
		\draw (0.66,3.5) node[above] {{\tiny$3$}} ; 			
		\draw (2.5,3.5) node[above] {{\tiny$2$}} ; 		
		\draw (0.5,-0.5) node[below] {{\tiny$3$}} ; 			
		\draw (2.25,-0.5) node[below] {{\tiny$1$}} ; 				
		\draw (2.5,-0.5) node[below] {{\tiny$2$}} ; 				
		\draw (2.75,-0.5) node[below] {{\tiny$4$}} ; 		
		\draw [->] (4,3.5) -- (4,-0.5); 					
	\end{tikzpicture}
	\caption{Examples of 2-leveled  connected graph directed from top to bottom.}
	\label{Fig:2LevDCGraph}
\end{figure*}
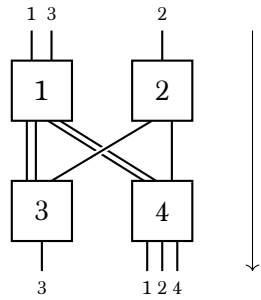

\begin{definition}[Properad and coproperad]
	A   \emph{properad} (resp. \emph{coproperad}) is a monoid (resp. \emph{comonoid}) in the monoidal category 
	$\big(\Sbimod,\allowbreak \boxtimes, \allowbreak \I\big)$.
\end{definition}

\begin{examples}\leavevmode
	\begin{description}
		\item[$\diamond$ \ \sc Endomorphism properad]
		The  \emph{endomorphism properad} of a chain complex $A$ is made up of all its multilinear maps:
		$$\End_A\coloneq\left\{\Hom\left(A^{\otimes n}, A^{{\otimes} m}\right)\right\}_{n, m\in \NN}\ .$$
		\item[$\diamond$ \ \sc Convolution properad] 
		For any coproperad $\mathcal{C}$ and any properad $\mathcal{P}$, the mapping 
		$\Sy$-bimodule 
		\[\{\Hom(\mathcal{C}(m,n), \mathcal{P}(m,n)\}_{m,n\in \NN}\] is endowed with a canonical natural properad structure called the \emph{convolution properad}, see \cite[Section~2.1]{MerkulovVallette09I}.
	\end{description}
\end{examples}

Let us now recall various properties and constructions of properads and coproperads, see \cite[Section~2]{HLV19} for more details. 
We denote  the  set of \emph{connected graphs directed by a global flow} by $\dcGraph$~. By convention, we  exclude the trivial  graph  $|$ from living in $\dcGraph$. 
For any $\Sy$-bimodule $\mathcal M$, the \emph{graph module} $\gra(\mathcal{M})$ is the $\Sy$-module obtained by labelling each vertex of $\gra$ by an element of $\mathcal{M}$ of corresponding arity. 
By convention, we set $|(\mathcal{M})\coloneq \I$. 
The notion of a properad is equivalently defined as an algebra over the  monad of graphs 
$$\mathcal{G}(\mathcal{M}) \coloneq \bigoplus_{\gra \in \dcGraph\cup \{|\}} \gra(\mathcal{M}) \ , $$
whose structure product is given by the substitution of graphs.
For any $k \geqslant 1$, we denote by $\dcGraph^{(k)}$ the set of connected directed graphs with $k$ vertices. 
We consider the \emph{infinitesimal composition product} 
\[\mathcal{M} \ibt \mathcal{M}\coloneq \bigoplus_{\gra \in \dcGraph^{(2)}} \gra(\mathcal{M}) \]
made up of labelled graphs with 2 vertices and an arbitrary number of edges connecting them. 
The \emph{infinitesimal composition map} 
\[
\begin{tikzcd}
	\gamma_{(1,1)} \ : \ \mathcal{P}\ibt \mathcal{P} \ar[r, hook]  &
	\mathcal{P} \boxtimes \mathcal{P} \ar[r,"\gamma"]  &
	\mathcal{P}
\end{tikzcd}
\]
of a properad $(\mathcal{P}, \gamma, \eta)$ amounts to composing only two operations at a time.

\medskip

One way to get a single chain complex from an $\Sy$-bimodule is to consider the total sum of the invariant parts of its components: 
\[\bigoplus_{m,n\in \NN} \mathcal{P}(m,n)^{\Sy_m^{\mathrm{op}}\times \Sy_n}~.\]
Recall that a binary product $\star$ is called \emph{Lie-admissible product} when its commutator 
$[\ ,\, ] \coloneq \ \star\ -\ \star ^{(12)}$ satisfies the Jacobi identity. 

\begin{lemma}[{\cite[Proposition~6]{MerkulovVallette09I}}]\label{lem:TotLieAdm}
	The infinitesimal composition map of a dg properad
	$(\mathcal{P}, \allowbreak \d, \allowbreak \gamma, \allowbreak \eta)$ 
	induces a dg Lie-admissible algebra structure 
	$\left(
	\oplus_{m,n\in \NN} \mathcal{P}(m,n)^{\Sy_m^{\mathrm{op}}\times \Sy_n}, 
	\d, \star
	\right)
	$
	on
	the total sum of the invariant parts of its components.
\end{lemma}

Unfortunately, the construction of the Lie-admissible product $\star$ is not well defined on the total \emph{product} of the components of a properad, since it might produce infinite sums of elements. For this to hold, we need to restrict ourselves to proper subclasses of properads, like the convolution properads  $\Hom(\C, \P)$, for instance. 

\medskip

The properties of coproperads are not exactly dual to the properties of properads. 
For instance, coalgebras over the  graphs comonad 
$$\mathcal{G}^c(\mathcal{M}) \coloneq \bigoplus_{\gra \in \mathsf{dcGra}} \gra(\mathcal{M}) \ , $$
whose structure coproduct is given by the partitions of graphs, 
do not encompass all coproperads: they 
induce coaugmented coproperad structures that are called \emph{conilpotent}, see \cite[Proposition~2.18]{HLV19}. 
Given a coaugmented coproperad $\mathcal C$, we denote by $\oC$ the cokernel of the coaugmentation map and we denote by
\[\widetilde{\Delta}\ :\ \oC \to \mathcal{G}^c\left(\oC\right) \]
the comonadic decomposition map of conilpotent coproperads.
The \emph{infinitesimal decomposition map} of a coaugmented  coproperad $(\mathcal{C}, \Delta)$ is given by 
\[
\begin{tikzcd}
	\Delta_{(1,1)} \ : \   \oC 
	\ar[r, "\Delta"]  &
	\mathcal{C} \boxtimes \mathcal{C} \ar[r,->>] &
	\oC \ibt \oC\ ,
\end{tikzcd}
\]
see \cite[Definition~2.20]{HLV19} for more details. 

\begin{definition}[Infinitesimal coproperad]
	An \emph{infinitesimal coproperad} is an $\Sy$-bimodule $\oC$ equipped with a morphism of $\Sy$-bimodules
	\[\Delta_{(1,1)} \colon\oC \to \oC \ibt \oC~,\] 
	which satisfies the properties of the restriction of a comonadic product, see  \cite[Section~2]{MerkulovVallette09I}.
\end{definition}

Any coaugmented coproperad $\C$ carries a canonical infinitesimal coproperad structure on its coaugmentation coideal $\oC$ by composing its 2-levels decomposition map with the counit everywhere expect for one place above and one place below. 
Any comonadic coproperad $\C$ carries a canonical infinitesimal coproperad structure by restriction of the decomposition to directed graphs with two vertices. Any infinitesimal coproperad $\oC$ can be coaugmented by adding freely a counit $\C \coloneq \oC \oplus \I$. 
In the other way round, an infinitesimal coproperad might fail producing a coproperad or a comonadic coproperad by iterating its infinitesimal decomposition maps since this process might produce infinite sums of elements. When the image of any element under all the possible iterations of the infinitesimal decomposition map produces a finite sum, one gets a notion equivalent to a conilpotent coproperad. To summarise, under the conilpotent condition, the three notions of coproperad, comonadic coproperad, and infinitesimal coproperad are equivalent.

\medskip

The main object of study of this work will be the following algebraic structure. 
\begin{proposition}[Convolution algebra {\cite[Proposition~11]{MerkulovVallette09I}}]\label{prop:ConLiead}
	For any infinitesimal dg coproperad $\oC$ and any dg module $A$, 
	the data 
	\[
	\g_{\mathcal{C}, A}:=\left(\Hom_{\Sy}\left(\oC, \End_A\right), \partial, \star
	\right)\ ,\]
	where 
	\[\Hom_{\Sy}\left(\oC, \End_A\right)\coloneq \prod_{n,m\geqslant 0} 
	\Hom_{\Sy_m\times\Sy_n^{\mathrm{op}}}\left(\oC(m,n), \Hom\left(A^{{\otimes} n}, A^{{\otimes} m}\right)\right)\]
	and 
	\[f\star g \ : \ \oC \xrightarrow{\Delta_{(1,1)}} \oC \ibt  \oC \xrightarrow{f\ibt  g} \End_A \ibt  \End_A \xrightarrow{\gamma_{(1,1)}} \End_A\ 
	\]
	forms a Lie-admissible algebra, called the \emph{convolution algebra}.   
\end{proposition}

When the infinitesimal coproperad comes from a coaugmented coproperad, the convolution Lie-admissible algebra
 is contained into the bigger dg module  
\[
\Hom_{\Sy}\left(\oC, \End_A\right)
\subset 
\Hom_{\Sy}\left(\C, \End_A\right)
\ ,\]
which admits the following two distinguished elements $\1$ of degree $0$ and $\delta$ of degree $-1$ defined respectively by 
\[
\begin{tikzcd}[column sep=normal, row sep=tiny]
	\1 \ : \ \C \ar[r,"\varepsilon"]  & \I \ar[r,"\id_A"] & \End_A  &\text{and} &
	\delta \ : \ \C \ar[r,"\varepsilon"]  & \I \ar[r,"\partial_A"] & \End_A \\
	\\
	\ \ \ \ \ \ \id \ar[rr, mapsto] && \id_A\ , &&
	\ \ \ \  \ \ \id \ar[rr, mapsto] && \partial_A\ , 
\end{tikzcd} 
\]
where $\varepsilon$ stands for the counit of the coproperad $\C$ and where $\partial_A$ stands for the differential of $\End_A$~. 

\begin{definition}[Convolution monoid]\label{def:ConMono}
	The 
	\emph{convolution monoid} 
	associated to any dg coproperad $\C$
	and any dg module $A$, is defined on the degree $0$ elements of $\Hom_{\Sy}\left(\C, \End_A\right)$ by
	\[
	{\a_{\mathcal{C}, A}:=\big(\Hom_{\Sy}\left(\C, \End_A\right)_0, \circledcirc, \1\big)}\ ,\]
	equipped with  the following associative and unital product
	\[
	f\circledcirc g \ : \ \C \xrightarrow{\Delta} 
	\C \boxtimes \C \xrightarrow{f\boxtimes g}  
	\End_A\boxtimes \End_A \xrightarrow{\gamma}
	\End_A \  ,
	\]
	where $\Delta$ and $\gamma$ stand respectively for the decomposition  map of the coproperad $\C$ and the  composition map of  the endomorphism operad $\End_A$\ . 
\end{definition}

\begin{definition}[Maurer--Cartan element]
	A \emph{Maurer--Cartan element} is a degree $-1$ element $\bar\alpha$ of a dg Lie-admissible algebra $(\g, \d, \star)$ satisfying the equation 
	\[\d \bar\alpha+\bar\alpha \star \bar\alpha=0\ .\]
	We denote the set of Maurer--Cartan elements by $\MC(\g)$~. 
\end{definition}

\begin{proposition}[{\cite[Proposition~17]{MerkulovVallette09I}}]\label{prop:MCInfGebra}
	For any infinitesimal coaugmented dg coproperad $\oC$
	and any  dg module $A$, the 
	Maurer--Cartan elements of the convolution algebra $\g_{\mathcal{C}, A}$
	are  in natural one-to-one correspondence with $\Omega \C$-gebra structures on $A$.
\end{proposition}

\begin{example}
	The main range of applications for this result is when one considers the Koszul dual coproperad $\C=\P^{\ac}$~. In this case, Maurer--Cartan elements of $\g_{\P^{\ac}, A}$ coincide with homotopy $\P$-gebra structures on $A$~.
\end{example}

\subsection{Lie-graph algebras}\label{subsec:DsGRAoperad}
The convolution algebra actually carries a more refined algebraic structure governed by a new operad that we introduce now. 

\begin{definition}[Directed simple graph] 
	We call \emph{directed simple graph} a connected graph $\gra$, directed by a global flow from top to bottom, with at least one vertex, with no input, no output and at most one edge between two vertices. 
	The number of vertices of a graph $\gra$ is denoted by $|\gra|$ and they are  labeled bijectively by $1,\ldots, |\gra|$. The set of directed simple graphs is denoted by $\dcGra$; examples are given in \cref{Fig:SimpleGraph}. 
\end{definition}
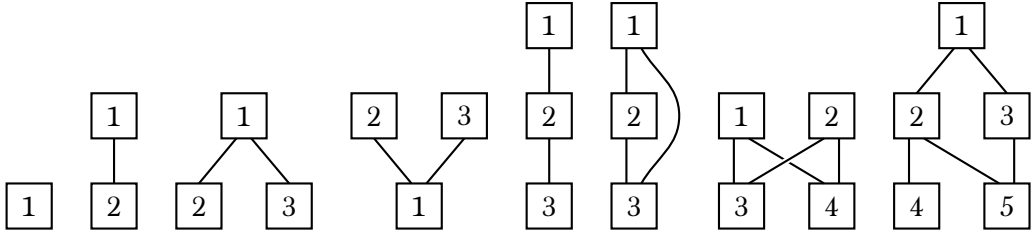
\begin{figure*}[h!]
	\begin{tikzpicture}[scale=0.6]
		\draw[thick]	(0,0)--(1,0)--(1,1)--(0,1)--cycle;
		\draw (0.5,0.5) node {{$1$}} ; 
	\end{tikzpicture} 
	\quad 
	\begin{tikzpicture}[scale=0.6]
		\draw[thick]	(0,0)--(1,0)--(1,1)--(0,1)--cycle;
		\draw[thick]	(0,2)--(1,2)--(1,3)--(0,3)--cycle;
		\draw[thick]	(0.5,1)--(0.5,2);
		\draw (0.5,0.5) node {{$2$}} ; 
		\draw (0.5,2.5) node {{$1$}} ; 
	\end{tikzpicture}
	\quad 
	\begin{tikzpicture}[scale=0.6]
		\draw[thick]	(0,0)--(1,0)--(1,1)--(0,1)--cycle;
		\draw[thick]	(2,0)--(3,0)--(3,1)--(2,1)--cycle;
		\draw[thick]	(1,2)--(2,2)--(2,3)--(1,3)--cycle;
		\draw[thick]	(0.5,1)--(1.33,2);
		\draw[thick]	(2.5,1)--(1.66,2);
		\draw (1.5,2.5) node {{$1$}} ; 
		\draw (0.5,0.5) node {{$2$}} ; 
		\draw (2.5,0.5) node {{$3$}} ; 	
	\end{tikzpicture}
	\quad 
	\begin{tikzpicture}[scale=0.6]
		\draw[thick]	(0,2)--(1,2)--(1,3)--(0,3)--cycle;
		\draw[thick]	(2,2)--(3,2)--(3,3)--(2,3)--cycle;
		\draw[thick]	(1,0)--(2,0)--(2,1)--(1,1)--cycle;
		\draw[thick]	(0.5,2)--(1.33,1);
		\draw[thick]	(2.5,2)--(1.66,1);
		\draw (1.5,0.5) node {{$1$}} ; 
		\draw (0.5,2.5) node {{$2$}} ; 
		\draw (2.5,2.5) node {{$3$}} ; 	
	\end{tikzpicture}
	\quad 
	\begin{tikzpicture}[scale=0.6]
		\draw[thick]	(0,0)--(1,0)--(1,1)--(0,1)--cycle;
		\draw[thick]	(0,2)--(1,2)--(1,3)--(0,3)--cycle;
		\draw[thick]	(0,4)--(1,4)--(1,5)--(0,5)--cycle;
		\draw[thick]	(0.5,1)--(0.5,2);
		\draw[thick]	(0.5,3)--(0.5,4);	
		\draw (0.5,0.5) node {{$3$}} ; 
		\draw (0.5,2.5) node {{$2$}} ; 
		\draw (0.5,4.5) node {{$1$}} ; 
	\end{tikzpicture}
	\quad 
	\begin{tikzpicture}[scale=0.6]
		\draw[thick]	(0,0)--(1,0)--(1,1)--(0,1)--cycle;
		\draw[thick]	(0,2)--(1,2)--(1,3)--(0,3)--cycle;
		\draw[thick]	(0,4)--(1,4)--(1,5)--(0,5)--cycle;
		\draw[thick]	(0.33,1)--(0.33,2);
		\draw[thick]	(0.33,3)--(0.33,4);
		\draw[thick]	(0.66,1) to[out=60,in=270] (1.5,2.5) to[out=90,in=300] (0.66,4);	
		\draw (0.5,0.5) node {{$3$}} ; 
		\draw (0.5,2.5) node {{$2$}} ; 
		\draw (0.5,4.5) node {{$1$}} ; 
	\end{tikzpicture}
	\quad
	\begin{tikzpicture}[scale=0.6]
		\draw[thick]	(0.66,2)--(2.33,1);	
		\draw[draw=white,double=black,double distance=2*\pgflinewidth,thick]	(0.66,1)--(2.33,2);		
		\draw[thick]	(0,0)--(1,0)--(1,1)--(0,1)--cycle;
		\draw[thick]	(2,0)--(3,0)--(3,1)--(2,1)--cycle;
		\draw[thick]	(0,2)--(1,2)--(1,3)--(0,3)--cycle;
		\draw[thick]	(2,2)--(3,2)--(3,3)--(2,3)--cycle;	
		\draw[thick]	(0.33,1)--(0.33,2);
		\draw[thick]	(2.66,1)--(2.66,2);
		\draw (0.5,2.5) node {{$1$}} ; 
		\draw (2.5,2.5) node {{$2$}} ; 	
		\draw (0.5,0.5) node {{$3$}} ; 
		\draw (2.5,0.5) node {{$4$}} ; 	
	\end{tikzpicture}
	\quad 
	\begin{tikzpicture}[scale=0.6]
		\draw[thick]	(0.66,2)--(2.33,1);	
		\draw[thick]	(0,0)--(1,0)--(1,1)--(0,1)--cycle;
		\draw[thick]	(2,0)--(3,0)--(3,1)--(2,1)--cycle;
		\draw[thick]	(0,2)--(1,2)--(1,3)--(0,3)--cycle;
		\draw[thick]	(2,2)--(3,2)--(3,3)--(2,3)--cycle;	
		\draw[thick]	(1,4)--(2,4)--(2,5)--(1,5)--cycle;
		\draw[thick]	(0.33,1)--(0.33,2);
		\draw[thick]	(2.66,1)--(2.66,2);
		\draw[thick]	(0.5,3)--(1.33,4);
		\draw[thick]	(2.5,3)--(1.66,4);	
		\draw (1.5,4.5) node {{$1$}} ; 
		\draw (0.5,2.5) node {{$2$}} ; 
		\draw (2.5,2.5) node {{$3$}} ; 	
		\draw (0.5,0.5) node {{$4$}} ; 
		\draw (2.5,0.5) node {{$5$}} ; 	
	\end{tikzpicture}
	\caption{The first  directed simple graphs.}
	\label{Fig:SimpleGraph}
\end{figure*}

The canonical action of the symmetric groups on the labels of the vertices of directed simple graphs endow the set $\dcGra$ 
with a set-theoretical $\Sy$-module structure. We denote by $\kdcGra$  the $\k$-linear $\Sy$-module spanned by $\dcGra$~. We endow it with the partial composition products $\gra_1 \circ_i \gra_2$ defined by the sum of graphs obtained by 
\begin{enumerate}
	\item inserting the graph $\gra_2$ at the vertex $i$ of the graph $\gra_1$, 
	\item for any edge joining a vertex $j$ to the vertex $i$ in $\g_1$, considering the sum of all the possible 
	ways to join with edges the vertex $j$ to at least one vertex of $\gra_2$,
	\item keeping the labels $1, \ldots, i-1$ of the graph $\gra_1$, 
	\item shifting the labels of $\gra_2$ by $i-1$, 
	\item and shifting the remaining vertices of $\gra_1$ by $|\gra_2|-1$\ . 
\end{enumerate}
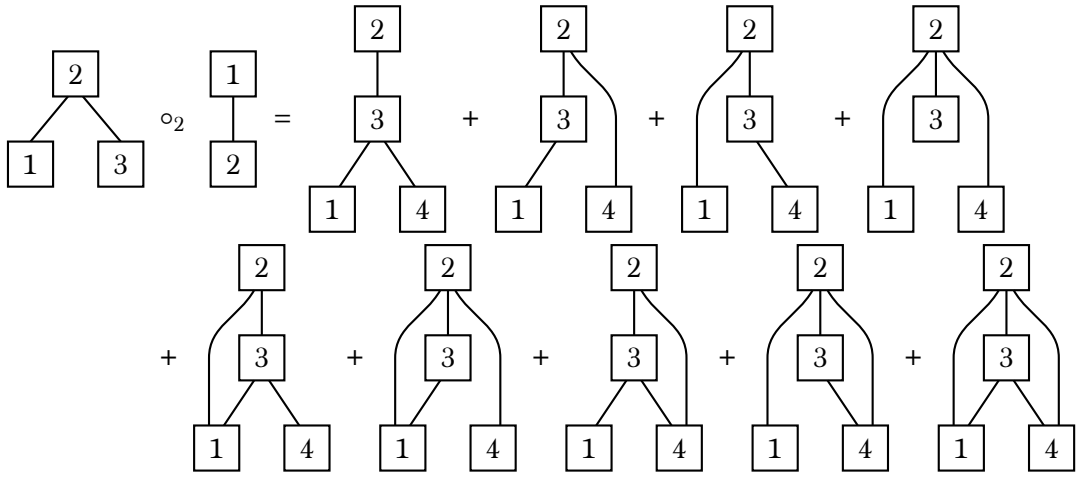
\begin{figure*}[h!]
	\begin{align*}
		\vcenter{\hbox{\begin{tikzpicture}[scale=0.6]
					\draw[thick]	(0,0)--(1,0)--(1,1)--(0,1)--cycle;
					\draw[thick]	(2,0)--(3,0)--(3,1)--(2,1)--cycle;
					\draw[thick]	(1,2)--(2,2)--(2,3)--(1,3)--cycle;
					\draw[thick]	(0.5,1)--(1.33,2);
					\draw[thick]	(2.5,1)--(1.66,2);
					\draw (1.5,2.5) node {{$2$}} ; 
					\draw (0.5,0.5) node {{$1$}} ; 
					\draw (2.5,0.5) node {{$3$}} ; 	
		\end{tikzpicture}}}
		& \ \circ_2\ 
		\vcenter{\hbox{
				\begin{tikzpicture}[scale=0.6]
					\draw[thick]	(0,0)--(1,0)--(1,1)--(0,1)--cycle;
					\draw[thick]	(0,2)--(1,2)--(1,3)--(0,3)--cycle;
					\draw[thick]	(0.5,1)--(0.5,2);
					\draw (0.5,0.5) node {{$2$}} ; 
					\draw (0.5,2.5) node {{$1$}} ; 
				\end{tikzpicture}
		}} = \ 
		\vcenter{\hbox{\begin{tikzpicture}[scale=0.6]
					\draw[thick] (0.66,1)--(1.33,2);
					\draw[thick] (2.33,1)--(1.66,2);
					\draw[thick]	(0,0)--(1,0)--(1,1)--(0,1)--cycle;
					\draw[thick]	(2,0)--(3,0)--(3,1)--(2,1)--cycle;
					\draw[thick]	(1,2)--(2,2)--(2,3)--(1,3)--cycle;
					\draw[thick]	(1,4)--(2,4)--(2,5)--(1,5)--cycle;	
					\draw[thick]	(1.5,3)--(1.5,4);	
					\draw (1.5,4.5) node {{$2$}} ; 
					\draw (1.5,2.5) node {{$3$}} ; 
					\draw (0.5,0.5) node {{$1$}} ; 
					\draw (2.5,0.5) node {{$4$}} ; 	
		\end{tikzpicture}}}
		\ + \ \vcenter{\hbox{\begin{tikzpicture}[scale=0.6]
					\draw[thick] (0.66,1)--(1.33,2);
					\draw[thick] (2.66,1) to (2.66,2.5) to[out=90,in=300]  (1.66,4);
					\draw[thick]	(0,0)--(1,0)--(1,1)--(0,1)--cycle;
					\draw[thick]	(2,0)--(3,0)--(3,1)--(2,1)--cycle;
					\draw[thick]	(1,2)--(2,2)--(2,3)--(1,3)--cycle;
					\draw[thick]	(1,4)--(2,4)--(2,5)--(1,5)--cycle;	
					\draw[thick]	(1.5,3)--(1.5,4);	
					\draw (1.5,4.5) node {{$2$}} ; 
					\draw (1.5,2.5) node {{$3$}} ; 
					\draw (0.5,0.5) node {{$1$}} ; 
					\draw (2.5,0.5) node {{$4$}} ; 	
		\end{tikzpicture}}}
		\ + \ \vcenter{\hbox{\begin{tikzpicture}[scale=0.6]
					\draw[thick] (0.33,1) to (0.33, 2.5) to[out=90,in=240] (1.33,4);
					\draw[thick] (2.33,1)--(1.66,2);
					\draw[thick]	(0,0)--(1,0)--(1,1)--(0,1)--cycle;
					\draw[thick]	(2,0)--(3,0)--(3,1)--(2,1)--cycle;
					\draw[thick]	(1,2)--(2,2)--(2,3)--(1,3)--cycle;
					\draw[thick]	(1,4)--(2,4)--(2,5)--(1,5)--cycle;	
					\draw[thick]	(1.5,3)--(1.5,4);	
					\draw (1.5,4.5) node {{$2$}} ; 
					\draw (1.5,2.5) node {{$3$}} ; 
					\draw (0.5,0.5) node {{$1$}} ; 
					\draw (2.5,0.5) node {{$4$}} ; 	
		\end{tikzpicture}}}
		\ + \ \vcenter{\hbox{\begin{tikzpicture}[scale=0.6]
					\draw[thick] (0.33,1) to (0.33, 2.5) to[out=90,in=240] (1.33,4);
					\draw[thick] (2.66,1) to (2.66,2.5) to[out=90,in=300]  (1.66,4);
					\draw[thick]	(0,0)--(1,0)--(1,1)--(0,1)--cycle;
					\draw[thick]	(2,0)--(3,0)--(3,1)--(2,1)--cycle;
					\draw[thick]	(1,2)--(2,2)--(2,3)--(1,3)--cycle;
					\draw[thick]	(1,4)--(2,4)--(2,5)--(1,5)--cycle;	
					\draw[thick]	(1.5,3)--(1.5,4);	
					\draw (1.5,4.5) node {{$2$}} ; 
					\draw (1.5,2.5) node {{$3$}} ; 
					\draw (0.5,0.5) node {{$1$}} ; 
					\draw (2.5,0.5) node {{$4$}} ; 	
		\end{tikzpicture}}} \\ &
		\ + \ \vcenter{\hbox{\begin{tikzpicture}[scale=0.6]
					\draw[thick] (0.66,1)--(1.33,2);
					\draw[thick] (0.33,1) to (0.33, 2.5) to[out=90,in=240] (1.33,4);
					\draw[thick] (2.33,1)--(1.66,2);
					\draw[thick]	(0,0)--(1,0)--(1,1)--(0,1)--cycle;
					\draw[thick]	(2,0)--(3,0)--(3,1)--(2,1)--cycle;
					\draw[thick]	(1,2)--(2,2)--(2,3)--(1,3)--cycle;
					\draw[thick]	(1,4)--(2,4)--(2,5)--(1,5)--cycle;	
					\draw[thick]	(1.5,3)--(1.5,4);	
					\draw (1.5,4.5) node {{$2$}} ; 
					\draw (1.5,2.5) node {{$3$}} ; 
					\draw (0.5,0.5) node {{$1$}} ; 
					\draw (2.5,0.5) node {{$4$}} ; 	
		\end{tikzpicture}}}
		\ + \ \vcenter{\hbox{\begin{tikzpicture}[scale=0.6]
					\draw[thick] (0.66,1)--(1.33,2);
					\draw[thick] (0.33,1) to (0.33, 2.5) to[out=90,in=240] (1.33,4);
					\draw[thick] (2.66,1) to (2.66,2.5) to[out=90,in=300]  (1.66,4);
					\draw[thick]	(0,0)--(1,0)--(1,1)--(0,1)--cycle;
					\draw[thick]	(2,0)--(3,0)--(3,1)--(2,1)--cycle;
					\draw[thick]	(1,2)--(2,2)--(2,3)--(1,3)--cycle;
					\draw[thick]	(1,4)--(2,4)--(2,5)--(1,5)--cycle;	
					\draw[thick]	(1.5,3)--(1.5,4);	
					\draw (1.5,4.5) node {{$2$}} ; 
					\draw (1.5,2.5) node {{$3$}} ; 
					\draw (0.5,0.5) node {{$1$}} ; 
					\draw (2.5,0.5) node {{$4$}} ; 	
		\end{tikzpicture}}}
		\ + \ \vcenter{\hbox{\begin{tikzpicture}[scale=0.6]
					\draw[thick] (0.66,1)--(1.33,2);
					\draw[thick] (2.33,1)--(1.66,2);
					\draw[thick] (2.66,1) to (2.66,2.5) to[out=90,in=300]  (1.66,4);
					\draw[thick]	(0,0)--(1,0)--(1,1)--(0,1)--cycle;
					\draw[thick]	(2,0)--(3,0)--(3,1)--(2,1)--cycle;
					\draw[thick]	(1,2)--(2,2)--(2,3)--(1,3)--cycle;
					\draw[thick]	(1,4)--(2,4)--(2,5)--(1,5)--cycle;	
					\draw[thick]	(1.5,3)--(1.5,4);	
					\draw (1.5,4.5) node {{$2$}} ; 
					\draw (1.5,2.5) node {{$3$}} ; 
					\draw (0.5,0.5) node {{$1$}} ; 
					\draw (2.5,0.5) node {{$4$}} ; 	
		\end{tikzpicture}}}
		\ + \ \vcenter{\hbox{\begin{tikzpicture}[scale=0.6]
					\draw[thick] (0.33,1) to (0.33, 2.5) to[out=90,in=240] (1.33,4);
					\draw[thick] (2.33,1)--(1.66,2);
					\draw[thick] (2.66,1) to (2.66,2.5) to[out=90,in=300]  (1.66,4);
					\draw[thick]	(0,0)--(1,0)--(1,1)--(0,1)--cycle;
					\draw[thick]	(2,0)--(3,0)--(3,1)--(2,1)--cycle;
					\draw[thick]	(1,2)--(2,2)--(2,3)--(1,3)--cycle;
					\draw[thick]	(1,4)--(2,4)--(2,5)--(1,5)--cycle;	
					\draw[thick]	(1.5,3)--(1.5,4);	
					\draw (1.5,4.5) node {{$2$}} ; 
					\draw (1.5,2.5) node {{$3$}} ; 
					\draw (0.5,0.5) node {{$1$}} ; 
					\draw (2.5,0.5) node {{$4$}} ; 	
		\end{tikzpicture}}}
		\ + \ \vcenter{\hbox{\begin{tikzpicture}[scale=0.6]
					\draw[thick] (0.66,1)--(1.33,2);
					\draw[thick] (0.33,1) to (0.33, 2.5) to[out=90,in=240] (1.33,4);
					\draw[thick] (2.33,1)--(1.66,2);
					\draw[thick] (2.66,1) to (2.66,2.5) to[out=90,in=300]  (1.66,4);
					\draw[thick]	(0,0)--(1,0)--(1,1)--(0,1)--cycle;
					\draw[thick]	(2,0)--(3,0)--(3,1)--(2,1)--cycle;
					\draw[thick]	(1,2)--(2,2)--(2,3)--(1,3)--cycle;
					\draw[thick]	(1,4)--(2,4)--(2,5)--(1,5)--cycle;	
					\draw[thick]	(1.5,3)--(1.5,4);	
					\draw (1.5,4.5) node {{$2$}} ; 
					\draw (1.5,2.5) node {{$3$}} ; 
					\draw (0.5,0.5) node {{$1$}} ; 
					\draw (2.5,0.5) node {{$4$}} ; 	
		\end{tikzpicture}}}
	\end{align*}
	\caption{Example of a partial composition in the operad $\Liegra$.}
	\label{Fig:PartCompo}
\end{figure*}

\begin{lemma}\label{lem:Lie-gra is an operad}
	The data $(\kdcGra, \{\circ_i\}, \vcenter{\hbox{\begin{tikzpicture}[scale=0.3]
				\draw[thick]	(0,0)--(1,0)--(1,1)--(0,1)--cycle;
				\draw (0.5,0.5) node {{\tiny$1$}} ; 
	\end{tikzpicture}}} 
	)$ forms an algebraic operad. 
\end{lemma}

\begin{proof}
	The element 
	$\vcenter{\hbox{\begin{tikzpicture}[scale=0.3]
				\draw[thick]	(0,0)--(1,0)--(1,1)--(0,1)--cycle;
				\draw (0.5,0.5) node {{\tiny$1$}} ; 
	\end{tikzpicture}}}$
	is clearly a unit for the partial composition products. The parallel and sequential axioms are straightforward to check. The various compatibilities with the action of the symmetric groups are also automatic. 
\end{proof}

\begin{definition}[Lie-graph algebra]
	We call \emph{Lie-graph algebra} an algebra over the operad 
	\[\Liegra\coloneq (\kdcGra, \{\circ_i\}, \vcenter{\hbox{\begin{tikzpicture}[scale=0.3]
				\draw[thick]	(0,0)--(1,0)--(1,1)--(0,1)--cycle;
				\draw (0.5,0.5) node {{\tiny$1$}} ; 
	\end{tikzpicture}}} 
	)~.\]
\end{definition}

The terminology chosen here comes from the relationship with the following various Lie-type algebraic structures. 
As we will see in \cref{prop:TotdsGra} and \cref{prop:conv algebras are dsGra}, this operad encodes  all the natural  operations acting on the totalisation of properad and thus on the deformation complex of morphisms of properads, which is necessary to adequately express the gauge group. 
Contrary to operads such as $\Lie$, $\mathrm{Ass}$, $\Prelie$, which are finitely presented and easily described in terms of generators and relations, no such description is available for $\Liegra$. Informally, this means that we cannot bypass operad theory to even define $\Liegra$-algebras by a finite set of structural operations satisfying some relations. This can be formalised by the proposition below, whose proof we postpone to \cref{AppA}. 

\begin{proposition}\label{prop:nofinitepresentation}
The operad $\Liegra$ is not finitely generated. 
\end{proposition}

First, the binary operation of $\Liegra$ given by the graph with two vertices satisfies the Lie-admissible relation, i.e. we have a morphism of operads 
\[ \begin{array}{ccc}
	\Lieadm & \to &\Liegra\\
	\star&\mapsto&  
	\vcenter{\hbox{
			\begin{tikzpicture}[scale=0.4]
				\draw[thick]	(0,0)--(1,0)--(1,1)--(0,1)--cycle;
				\draw[thick]	(0,2)--(1,2)--(1,3)--(0,3)--cycle;
				\draw[thick]	(0.5,1)--(0.5,2);
				\draw (0.5,0.5) node {{\scalebox{0.8}{$2$}}} ; 
				\draw (0.5,2.5) node {{\scalebox{0.8}{$1$}}} ; 
	\end{tikzpicture}}}\ \ .
\end{array}\]
It is worth noticing that this  morphism of operads  \emph{fails to be surjective} since for instance 
\[\dim \Lieadm(3)=11 \neq \dim \Liegra(3)=18~.\] This means that the data of a Lie-graph algebra structure is significantly richer than that of a Lie-admissible product. 

\medskip

The operation of skew-symmetrising a Lie-admissible product produces a Lie bracket; this amounts to a morphism of operads 
$\Lie \rightarrowtail  \Lieadm$~, which is injective after \cite{MarklRemm06}. So, any Lie-graph algebra structure contains a Lie-admissible product which yields a Lie bracket under  skew-symmetrisation; that induces the following composite of operadic morphisms:
\[\Lie \rightarrowtail \Lieadm \to \Liegra \ .\]

The notion of a Lie-graph algebra generalises the main type of algebraic structure governing the operadic deformation theory: pre-Lie algebras. Let us recall that a pre-Lie product is a binary operation whose associator is right symmetric, see \cite[Section~13.4]{LodayVallette12}. The associated binary quadratic operad admits a graphical interpretation \cite{ChapotonLivernet01}: it is isomorphic to the operad of rooted trees 
\[\Prelie\cong \mathrm{RT}~,\] 
where the operadic composition map is similar to the one of the operad $\Liegra$ (insertion at a vertex and then grafting), except that one does not allow the number of edges to increase, see \cref{Fig:PartCompoPreLie}. 
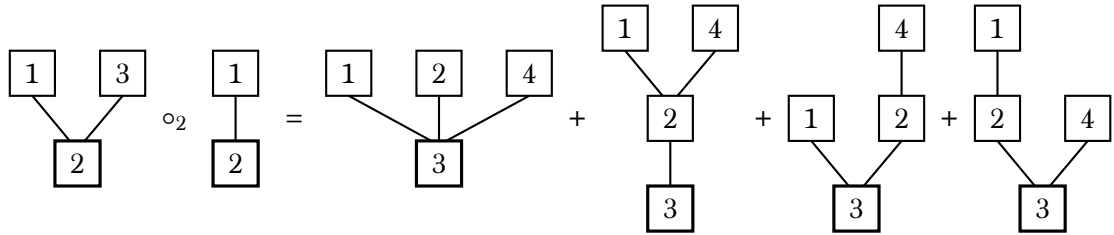
\begin{figure*}[h!]
	\begin{align*}
		\vcenter{\hbox{\begin{tikzpicture}[scale=0.6]
					\draw[thick]	(0,2)--(1,2)--(1,3)--(0,3)--cycle;
					\draw[thick]	(2,2)--(3,2)--(3,3)--(2,3)--cycle;
					\draw[very thick]	(1,0)--(2,0)--(2,1)--(1,1)--cycle;
					\draw[thick]	(0.5,2)--(1.33,1);
					\draw[thick]	(2.5,2)--(1.66,1);
					\draw (1.5,0.5) node {{$2$}} ; 
					\draw (0.5,2.5) node {{$1$}} ; 
					\draw (2.5,2.5) node {{$3$}} ; 	
		\end{tikzpicture}}}
		& \ \circ_2\ 
		\vcenter{\hbox{
				\begin{tikzpicture}[scale=0.6]
					\draw[very thick]	(0,0)--(1,0)--(1,1)--(0,1)--cycle;
					\draw[thick]	(0,2)--(1,2)--(1,3)--(0,3)--cycle;
					\draw[thick]	(0.5,1)--(0.5,2);
					\draw (0.5,0.5) node {{$2$}} ; 
					\draw (0.5,2.5) node {{$1$}} ; 
				\end{tikzpicture}
		}}\ = \ \,
		\vcenter{\hbox{\begin{tikzpicture}[scale=0.6]
					\draw[thick]	(-1,2)--(0,2)--(0,3)--(-1,3)--cycle;
					\draw[thick]	(3,2)--(4,2)--(4,3)--(3,3)--cycle;
					\draw[very thick]	(1,0)--(2,0)--(2,1)--(1,1)--cycle;
					\draw[thick]	(1,2)--(2,2)--(2,3)--(1,3)--cycle;	
					\draw[thick]	(-0.5,2)--(1.33,1);
					\draw[thick]	(3.5,2)--(1.66,1);
					\draw[thick]	(1.5,2)--(1.5,1);	
					\draw (-0.5,2.5) node {{$1$}} ; 
					\draw (1.5,0.5) node {{$3$}} ; 
					\draw (1.5,2.5) node {{$2$}} ; 
					\draw (3.5,2.5) node {{$4$}} ; 	
		\end{tikzpicture}}}
		\ +\ 
		\vcenter{\hbox{\begin{tikzpicture}[scale=0.6]
					\draw[thick]	(0,2)--(1,2)--(1,3)--(0,3)--cycle;
					\draw[thick]	(2,2)--(3,2)--(3,3)--(2,3)--cycle;
					\draw[thick]	(1,0)--(2,0)--(2,1)--(1,1)--cycle;
					\draw[very thick]	(1,-2)--(2,-2)--(2,-1)--(1,-1)--cycle;	
					\draw[thick]	(0.5,2)--(1.33,1);
					\draw[thick]	(2.5,2)--(1.66,1);
					\draw[thick]	(1.5,-1)--(1.5,0);
					\draw (1.5,-1.5) node {{$3$}} ; 
					\draw (1.5,0.5) node {{$2$}} ; 
					\draw (0.5,2.5) node {{$1$}} ; 
					\draw (2.5,2.5) node {{$4$}} ; 	
		\end{tikzpicture}}}
		\ + \ 
		\vcenter{\hbox{\begin{tikzpicture}[scale=0.6]
					\draw[thick]	(0,2)--(1,2)--(1,3)--(0,3)--cycle;
					\draw[thick]	(2,2)--(3,2)--(3,3)--(2,3)--cycle;
					\draw[thick]	(2,4)--(3,4)--(3,5)--(2,5)--cycle;	
					\draw[very thick]	(1,0)--(2,0)--(2,1)--(1,1)--cycle;
					\draw[thick]	(0.5,2)--(1.33,1);
					\draw[thick]	(2.5,2)--(1.66,1);
					\draw[thick]	(2.5,4)--(2.5,3);	
					\draw (1.5,0.5) node {{$3$}} ; 
					\draw (0.5,2.5) node {{$1$}} ; 
					\draw (2.5,2.5) node {{$2$}} ; 	
					\draw (2.5,4.5) node {{$4$}} ; 		
		\end{tikzpicture}}}
		\ + \ 
		\vcenter{\hbox{\begin{tikzpicture}[scale=0.6]
					\draw[thick]	(0,2)--(1,2)--(1,3)--(0,3)--cycle;
					\draw[thick]	(2,2)--(3,2)--(3,3)--(2,3)--cycle;
					\draw[thick]	(0,4)--(1,4)--(1,5)--(0,5)--cycle;		
					\draw[very thick]	(1,0)--(2,0)--(2,1)--(1,1)--cycle;
					\draw[thick]	(0.5,2)--(1.33,1);
					\draw[thick]	(2.5,2)--(1.66,1);
					\draw[thick]	(0.5,4)--(0.5,3);	
					\draw (1.5,0.5) node {{$3$}} ; 
					\draw (0.5,2.5) node {{$2$}} ; 
					\draw (2.5,2.5) node {{$4$}} ; 	
					\draw (0.5,4.5) node {{$1$}} ; 		
		\end{tikzpicture}}}
	\end{align*}
	
	\caption{Example of a partial composition in the operad $\mathrm{RT}$.}
	\label{Fig:PartCompoPreLie}
\end{figure*}

In the same range of ideas, the operad $\mathrm{Ass}$ governing associative algebras can also to described in a graphical way as follows. 
We consider the subset of rooted trees made up of \emph{ladders}, that is rooted trees such that every vertex has at most one input. 
We endow this $\Sy$-set with the partial composition products $\circ_i$ where one first inserts the second ladder at the place of the $i$-th vertex of the first ladder and where  one then grafts the root vertex and the top vertex of the second ladder in order to produce a final ladder, see \cref{Fig:PartCompoLad}. 
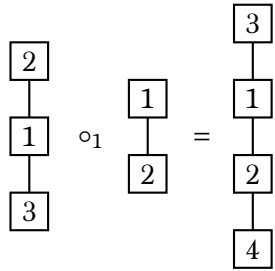
\begin{figure*}[h!]
	\begin{align*}
		\vcenter{\hbox{\begin{tikzpicture}[scale=0.5]
					\draw[thick]	(0,0)--(1,0)--(1,1)--(0,1)--cycle;
					\draw[thick]	(0,2)--(1,2)--(1,3)--(0,3)--cycle;
					\draw[thick]	(0,-2)--(1,-2)--(1,-1)--(0,-1)--cycle;
					\draw[thick]	(0.5,1)--(0.5,2);
					\draw[thick]	(0.5,-1)--(0.5,0);
					\draw (0.5,0.5) node {{$1$}} ; 
					\draw (0.5,2.5) node {{$2$}} ; 
					\draw (0.5,-1.5) node {{$3$}} ; 
		\end{tikzpicture}}}\ \,
		& \ \circ_1\ 
		\vcenter{\hbox{
				\begin{tikzpicture}[scale=0.5]
					\draw[thick]	(0,0)--(1,0)--(1,1)--(0,1)--cycle;
					\draw[thick]	(0,2)--(1,2)--(1,3)--(0,3)--cycle;
					\draw[thick]	(0.5,1)--(0.5,2);
					\draw (0.5,0.5) node {{$2$}} ; 
					\draw (0.5,2.5) node {{$1$}} ; 
				\end{tikzpicture}
		}}\ = \ \,
		\vcenter{\hbox{\begin{tikzpicture}[scale=0.5]
					\draw[thick]	(0,0)--(1,0)--(1,1)--(0,1)--cycle;
					\draw[thick]	(0,2)--(1,2)--(1,3)--(0,3)--cycle;
					\draw[thick]	(0,-2)--(1,-2)--(1,-1)--(0,-1)--cycle;
					\draw[thick]	(0,-4)--(1,-4)--(1,-3)--(0,-3)--cycle;	
					\draw[thick]	(0.5,1)--(0.5,2);
					\draw[thick]	(0.5,-1)--(0.5,0);
					\draw[thick]	(0.5,-3)--(0.5,-2);
					\draw (0.5,2.5) node {{$3$}} ; 
					\draw (0.5,0.5) node {{$1$}} ; 
					\draw (0.5,-1.5) node {{$2$}} ; 
					\draw (0.5,-3.5) node {{$4$}} ; 	
		\end{tikzpicture}}}
	\end{align*}
	\caption{Example of a partial composition of ladders.}
	\label{Fig:PartCompoLad}
\end{figure*}
Denoting by $\mathrm{Lad}$ the $\k$-linear operad spanned by the set of ladders, one gets the canonical isomorphism of operads
\[ \mathrm{Ass}\cong \mathrm{Lad}~.\]

The projection of rooted trees onto ladders induces a surjection  of operads 
\[
\Prelie\cong \mathrm{RT}
\twoheadrightarrow
\mathrm{Lad}\cong \mathrm{Ass}~,\]
whose pullback coincides with the fact that an associative product gives a pre-Lie product. A similar relation holds between the operads $\Liegra$ and $\mathrm{RT}$. Since all the edges of a directed simple graph are oriented, they induce a partial order on the set of vertices, where we say that a vertex at the source of an edge is greater than the vertex lying at the target of that edge. 

\begin{proposition}
	The projection of directed simple graphs onto the ones of genus 0 with only one minimum vertex defines an epimorphism of operads
	\[\Liegra \twoheadrightarrow \mathrm{RT}\cong \Prelie\ .\]
\end{proposition}

\begin{proof}
	First we notice that the subset of genus 0 directed simple graphs with only one minimum vertex is in one-to-one correspondence with the set of rooted trees. In one way, rooted trees are directed simple graphs: for such trees, one orients the edges from the leaves to the root. 
	In the other way around, one can prove, for instance by induction on the number of vertices, that every vertex of a genus 0 directed simple  graph with only one minimum vertex has at most one output edge. 
	Then one can see that the projection of directed simple graphs onto rooted trees preserves the respective operadic partial composition products: its kills the extra terms appearing in the partial composition product of the operad $\Liegra$ and not in the one of the operad $\mathrm{RT}$~. 
\end{proof}

It is worth noticing that the canonical inclusions of graphs in the other way round, that is rooted trees into directed simple graphs and ladders into rooted trees, \emph{do not respect} the operadic structures and so they fail to define morphisms of operads. In the end, we get the following relations between these various algebraic structures:
\[\Lie \rightarrowtail \Lieadm \to \Liegra \twoheadrightarrow \Prelie \twoheadrightarrow \mathrm{Ass} \ ,\]
where $\Liegra$ is the biggest operad. 

\begin{remark}\label{rem:other graph operads}
	Despite an apparent similar flavour, the operad $\Liegra$ {is quite different} from the graphs operads considered in the study of the little discs operad, the configuration spaces of points, and the Grothendieck--Teichm\"uller Lie algebra, see  \cite{Kontsevich97, LambrechtsVolic14, Willwacher15, Idrissi19, CW16} and the survey \cite[Chapter~6]{DotsenkoShadrinVallette22}: they all share the property that their composition preserves the number of edges. 
\end{remark}

More generally, we consider the set $\dsncGra$ of \emph{directed simple non-necessarily connected graphs}, see 
	\cref{Fig:SimpleNCGraph} for examples. We endow them with the partial composition products $\gra_1 \circ_i \gra_2$ defined similarly by the sum of graphs obtained by first inserting the graph $\gra_2$ at the vertex $i$ of the graph $\gra_1$, 
	then joining any vertex $j$ of $\gra_1$, attached previously to $i$, in all possible ways to vertices of $\gra_2$ with at least one edge each time,
	and finally relabeling the vertices accordingly.

	\begin{figure*}[h!]
		\begin{tikzpicture}[scale=0.6]
			\draw[thick]	(0,0)--(1,0)--(1,1)--(0,1)--cycle;
			\draw[thick]	(2,0)--(3,0)--(3,1)--(2,1)--cycle;	
			\draw[thick]	(4,0)--(5,0)--(5,1)--(4,1)--cycle;		
			\draw (0.5,0.5) node {{$1$}} ; 
			\draw (2.5,0.5) node {{$2$}} ; 
			\draw (4.5,0.5) node {{$3$}} ; 		
		\end{tikzpicture} 
		\qquad \qquad \qquad
		\begin{tikzpicture}[scale=0.6]
			\draw[thick]	(-2,0)--(-1,0)--(-1,1)--(-2,1)--cycle;
			\draw[thick]	(-2,2)--(-1,2)--(-1,3)--(-2,3)--cycle;	
			\draw[thick]	(-2,4)--(-1,4)--(-1,5)--(-2,5)--cycle;		
			\draw[thick]	(0,0)--(1,0)--(1,1)--(0,1)--cycle;
			\draw[thick]	(2,0)--(3,0)--(3,1)--(2,1)--cycle;
			\draw[thick]	(1,2)--(2,2)--(2,3)--(1,3)--cycle;
			\draw[thick]	(0.5,1)--(1.33,2);
			\draw[thick]	(2.5,1)--(1.66,2);
			\draw[thick]	(-1.5,1)--(-1.5,2);
			\draw[thick]	(-1.5,3)--(-1.5,4);	
			\draw (-1.5,0.5) node {{$4$}} ; 
			\draw (-1.5,2.5) node {{$5$}} ; 	
			\draw (-1.5,4.5) node {{$1$}} ; 		
			\draw (1.5,2.5) node {{$3$}} ; 
			\draw (0.5,0.5) node {{$2$}} ; 
			\draw (2.5,0.5) node {{$6$}} ; 	
		\end{tikzpicture}
		\caption{Examples of  directed simple non-necessarily connected graphs.}
		\label{Fig:SimpleNCGraph}
	\end{figure*}
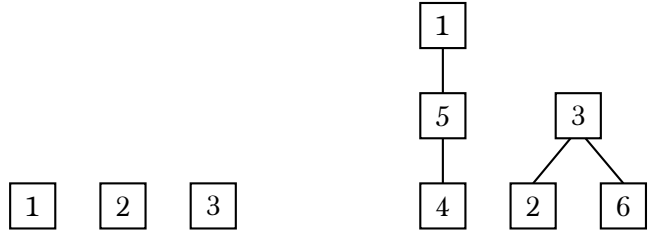

\begin{proposition}\label{prop:LiencgraDistLaw}
		The data $(\k\dsncGra, \{\circ_i\}, \vcenter{\hbox{\begin{tikzpicture}[scale=0.3]
					\draw[thick]	(0,0)--(1,0)--(1,1)--(0,1)--cycle;
					\draw (0.5,0.5) node {{\tiny$1$}} ; 
		\end{tikzpicture}}} 
		)$ defines an algebraic operad. The identification of $\Com(n)$ with the span of \quad 
		\begin{tikzpicture}[scale=0.4,baseline=.1cm]
			\draw[thick]	(0,0)--(1,0)--(1,1)--(0,1)--cycle;
			\draw[thick]	(3,0)--(4,0)--(4,1)--(3,1)--cycle;		
			\draw (0.5,0.5) node {{$1$}} ; 
			\draw (2.1,0.5) node {{$\dots$}} ; 
			\draw (3.5,0.5) node {{$n$}} ; 		
		\end{tikzpicture} 
		$\in \k\dsncGra(n)$
		induces a distributive law 
		$ \colon \Liegra\circ\mathrm{Com} \to \mathrm{Com}\circ \Liegra$
		between the operads 
		$\mathrm{Com}$ and $\Liegra$~, which is isomorphic to $(\k\dsncGra, \{\circ_i\}, \vcenter{\hbox{\begin{tikzpicture}[scale=0.3]
					\draw[thick]	(0,0)--(1,0)--(1,1)--(0,1)--cycle;
					\draw (0.5,0.5) node {{\tiny$1$}} ; 
		\end{tikzpicture}}} 
		)$~. 
\end{proposition}

	\begin{proof}
		Like in the proof of \cref{lem:Lie-gra is an operad}, the unit, parallel, sequential, and equivariance axioms are straightforward to prove. 
		
		\medskip
		
		Let us make fully explicit the distributive law. To any directed simple graph $\gra(\pi_1, \ldots, \pi_k)$ with $k$ vertices labelled bijectively by the components of a partition 
		$\pi_1\sqcup \ldots \sqcup \pi_k$
		of $\{1,\ldots, n\}$, we associate the sum of directed simple non-necessarily connected graphs obtained by 
		replacing all the numbers of the components of the partition by a vertex labeled by that number, and by replacing the edges joining two vertices in the graph $\gra$ by the sum of way to join these new vertices with at least one edge each time: 
		\[
		\vcenter
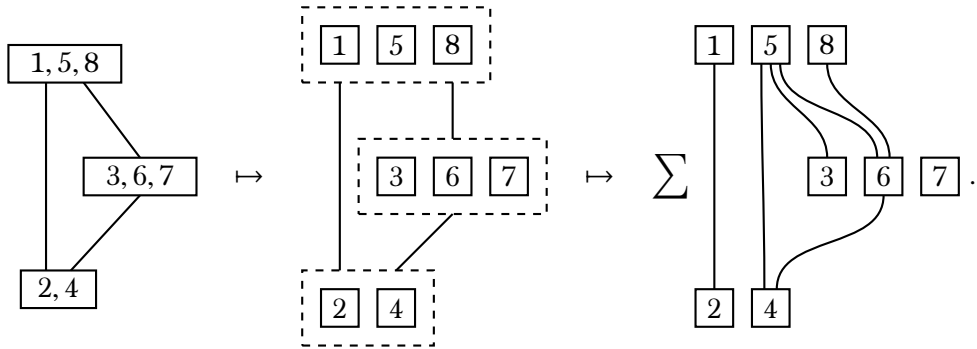
{\hbox{\begin{tikzpicture}[scale=0.5]
					\coordinate (A) at (0.83, 1);
					\draw ($(A)+(0.5,0.5)$) node {{$2,4$}} ; 
					\draw[thick] (0.33,1)--(2.33,1)--(2.33,2)--(0.33,2)--cycle;
					\coordinate (A) at (3.5, 4.5);
					\draw ($(A)$) node {{$3,6,7$}} ; 
					\draw[thick] (2,4)--(5,4)--(5,5)--(2,5)--cycle;
					\coordinate (A) at (1, 7);
					\draw ($(A)+(0.5,0.5)$) node {{$1,5,8$}} ; 	
					\draw[thick] (0,7)--(3,7)--(3,8)--(0,8)--cycle;	
					\draw[thick] (1,2)--(1,7);
					\draw[thick] (1.66,2)--(3.5,4);	
					\draw[thick] (2,7)--(3.5,5);	
		\end{tikzpicture}}}
		\quad \mapsto\quad 
		\vcenter{\hbox{\begin{tikzpicture}[scale=0.5]
					\coordinate (A) at (0.5, 0.5);
					\draw[thick]	($(A)+(0, 0)$)--($(A)+(0, 1)$)--($(A)+(1,1)$)--($(A)+(1, 0)$)--cycle;
					\draw ($(A)+(0.5,0.5)$) node {{$2$}} ; 
					\coordinate (A) at (2, 0.5);
					\draw[thick]	($(A)+(0, 0)$)--($(A)+(0, 1)$)--($(A)+(1,1)$)--($(A)+(1, 0)$)--cycle;
					\draw[thick, dashed] (0,0)--(3.5,0)--(3.5,2)--(0,2)--cycle;
					\draw ($(A)+(0.5,0.5)$) node {{$4$}} ; 
					\coordinate (A) at (2, 4);
					\draw[thick]	($(A)+(0, 0)$)--($(A)+(0, 1)$)--($(A)+(1,1)$)--($(A)+(1, 0)$)--cycle;
					\draw ($(A)+(0.5,0.5)$) node {{$3$}} ; 
					\coordinate (A) at (3.5, 4);
					\draw[thick]	($(A)+(0, 0)$)--($(A)+(0, 1)$)--($(A)+(1,1)$)--($(A)+(1, 0)$)--cycle;
					\draw ($(A)+(0.5,0.5)$) node {{$6$}} ; 	
					\coordinate (A) at (5, 4);
					\draw[thick]	($(A)+(0, 0)$)--($(A)+(0, 1)$)--($(A)+(1,1)$)--($(A)+(1, 0)$)--cycle;
					\draw ($(A)+(0.5,0.5)$) node {{$7$}} ; 
					\draw[thick, dashed] (1.5,3.5)--(6.5,3.5)--(6.5,5.5)--(1.5,5.5)--cycle;
					\draw[thick, dashed] (0,7)--(5,7)--(5,9)--(0,9)--cycle;
					\coordinate (A) at (0.5, 7.5);
					\draw[thick]	($(A)+(0, 0)$)--($(A)+(0, 1)$)--($(A)+(1,1)$)--($(A)+(1, 0)$)--cycle;
					\draw ($(A)+(0.5,0.5)$) node {{$1$}} ; 	
					\coordinate (A) at (2, 7.5);
					\draw[thick]	($(A)+(0, 0)$)--($(A)+(0, 1)$)--($(A)+(1,1)$)--($(A)+(1, 0)$)--cycle;
					\draw ($(A)+(0.5,0.5)$) node {{$5$}} ; 	
					\coordinate (A) at (3.5, 7.5);
					\draw[thick]	($(A)+(0, 0)$)--($(A)+(0, 1)$)--($(A)+(1,1)$)--($(A)+(1, 0)$)--cycle;	
					\draw ($(A)+(0.5,0.5)$) node {{$8$}} ; 	
					\draw[thick] (1,2)--(1,7);
					\draw[thick] (2.5,2)--(4,3.5);	
					\draw[thick] (4,7)--(4,5.5);	
		\end{tikzpicture}}}
		\quad\mapsto \quad \sum 
		\vcenter{\hbox{\begin{tikzpicture}[scale=0.5]
					\coordinate (A) at (0.5, 0.5);
					\draw[thick]	($(A)+(0, 0)$)--($(A)+(0, 1)$)--($(A)+(1,1)$)--($(A)+(1, 0)$)--cycle;
					\draw ($(A)+(0.5,0.5)$) node {{$2$}} ; 
					\coordinate (A) at (2, 0.5);
					\draw[thick]	($(A)+(0, 0)$)--($(A)+(0, 1)$)--($(A)+(1,1)$)--($(A)+(1, 0)$)--cycle;
					\draw ($(A)+(0.5,0.5)$) node {{$4$}} ; 
					\coordinate (A) at (3.5, 4);
					\draw[thick]	($(A)+(0, 0)$)--($(A)+(0, 1)$)--($(A)+(1,1)$)--($(A)+(1, 0)$)--cycle;
					\draw ($(A)+(0.5,0.5)$) node {{$3$}} ; 
					\coordinate (A) at (5, 4);
					\draw[thick]	($(A)+(0, 0)$)--($(A)+(0, 1)$)--($(A)+(1,1)$)--($(A)+(1, 0)$)--cycle;
					\draw ($(A)+(0.5,0.5)$) node {{$6$}} ; 	
					\coordinate (A) at (6.5, 4);
					\draw[thick]	($(A)+(0, 0)$)--($(A)+(0, 1)$)--($(A)+(1,1)$)--($(A)+(1, 0)$)--cycle;
					\draw ($(A)+(0.5,0.5)$) node {{$7$}} ; 
					\coordinate (A) at (0.5, 7.5);
					\draw[thick]	($(A)+(0, 0)$)--($(A)+(0, 1)$)--($(A)+(1,1)$)--($(A)+(1, 0)$)--cycle;
					\draw ($(A)+(0.5,0.5)$) node {{$1$}} ; 	
					\coordinate (A) at (2, 7.5);
					\draw[thick]	($(A)+(0, 0)$)--($(A)+(0, 1)$)--($(A)+(1,1)$)--($(A)+(1, 0)$)--cycle;
					\draw ($(A)+(0.5,0.5)$) node {{$5$}} ; 	
					\coordinate (A) at (3.5, 7.5);
					\draw[thick]	($(A)+(0, 0)$)--($(A)+(0, 1)$)--($(A)+(1,1)$)--($(A)+(1, 0)$)--cycle;	
					\draw ($(A)+(0.5,0.5)$) node {{$8$}} ; 	
					\draw[thick] (1,1.5)--(1,7.5);
					\draw[thick] (2.33,1.5) to[out=90,in=270] (2.25,7.5);	
					\draw[thick] (2.66,1.5) to[out=90,in=270] (5.5,4);		
					\draw[thick] (4,5) to[out=90,in=270] (2.5,7.5);
					\draw[thick] (5.33,5) to[out=90,in=270] (2.75,7.5);	
					\draw[thick] (5.66,5) to[out=90,in=270] (4,7.5);		
		\end{tikzpicture}}}~.
		\]
		
		The axioms $\mathrm{(I)}$, $\mathrm{(II)}$, $\mathrm{(i)}$, and, $\mathrm{(ii)}$, see \cite[Section~8.6.1]{LodayVallette12}, defining the notion of a distributive law are straightforward to check. Finally, it is easy to see that these two operads are isomorphic since the composite product $\mathrm{Com} \circ \Liegra$ is given by disjoint unions of direct simple graphs and since the respective two composition maps are the same. 
	\end{proof}

\begin{remark}
Notice that, since the operad  $\Liegra$ is not defined by generators and relations, this distributive law is not induced by a rewriting rule, which is rather unusual.  
\end{remark}

\begin{definition}[Lie-ncgraph algebra]\label{def:Poisgra}
		We call \emph{Lie-ncgraph algebra} an algebra over the operad 
		\[\Liencgra\coloneq (\k\dsncGra, \{\circ_i\}, \vcenter{\hbox{\begin{tikzpicture}[scale=0.3]
					\draw[thick]	(0,0)--(1,0)--(1,1)--(0,1)--cycle;
					\draw (0.5,0.5) node {{\tiny$1$}} ; 
		\end{tikzpicture}}} 
		)~.\]
\end{definition}

The inclusion $\dcGra\rightarrowtail \dsncGra$ defines an embedding of operads $\Liegra \rightarrowtail \Liencgra$~.

\begin{remark}		
On the other hand, the natural projection $\dsncGra \twoheadrightarrow \dcGra$ does not induce a morphism of operads, as the potential creation of edges means that inserting disconnected graphs can produce connected graphs.  		
For the same reason, the commutative product and the Lie bracket do not satisfy the Leibniz relation and therefore the operad 
$\Liencgra$ does not contain the Poisson operad. This is in contrast with the graph operads mentioned in \cref{rem:other graph operads}. \end{remark}

\subsection{Lie-graph algebra structure on properadic convolution algebra}\label{subsec:LiegraConv}
Conversely, one might wonder whether it is possible to lift Lie-admissible structures to Lie-graph algebra structures. 
In the cases of the sum of the components of a properad and of properadic  convolution algebras, we solve this question as follows. 
We consider the \emph{barber} map 
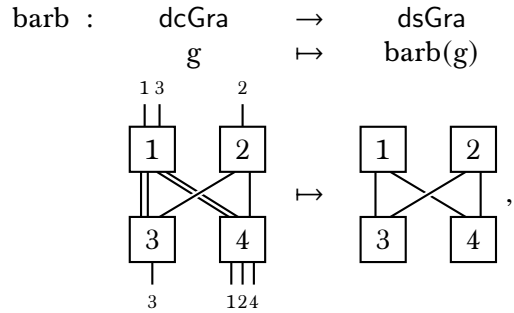
\[ \begin{array}{lccc}
	\barb
	\ : & \mathsf{dcGra} & \to &\dcGra\\
	&\gra&\mapsto& \trunc(\gra) \\
	& 
	\vcenter{\hbox{
			\begin{tikzpicture}[scale=0.6]
				\draw[thick]	(0.8,2)--(2.4,1);	
				\draw[thick]	(0.6,2)--(2.2,1);	
				\draw[draw=white,double=black,double distance=2*\pgflinewidth,thick]	(0.66,1)--(2.33,2);		
				\draw[thick]	(0,0)--(1,0)--(1,1)--(0,1)--cycle;
				\draw[thick]	(2,0)--(3,0)--(3,1)--(2,1)--cycle;
				\draw[thick]	(0,2)--(1,2)--(1,3)--(0,3)--cycle;
				\draw[thick]	(2,2)--(3,2)--(3,3)--(2,3)--cycle;	
				\draw[thick]	(0.25,1)--(0.25,2);
				\draw[thick]	(0.4,1)--(0.4,2);
				\draw[thick]	(2.66,1)--(2.66,2);
				\draw[thick]	(0.33,3)--(0.33,3.5);	
				\draw[thick]	(0.66,3)--(0.66,3.5);		
				\draw[thick]	(2.5,3)--(2.5,3.5);		
				\draw[thick]	(2.25,0)--(2.25,-0.5);				
				\draw[thick]	(0.5,0)--(0.5,-0.5);				
				\draw[thick]	(2.75,0)--(2.75,-0.5);						
				\draw[thick]	(2.5,0)--(2.5,-0.5);							
				\draw (0.5,2.5) node {{$1$}} ; 
				\draw (2.5,2.5) node {{$2$}} ; 	
				\draw (0.5,0.5) node {{$3$}} ; 
				\draw (2.5,0.5) node {{$4$}} ; 	
				\draw (0.33,3.5) node[above] {{\tiny$1$}} ; 		
				\draw (0.66,3.5) node[above] {{\tiny$3$}} ; 			
				\draw (2.5,3.5) node[above] {{\tiny$2$}} ; 		
				\draw (0.5,-0.5) node[below] {{\tiny$3$}} ; 			
				\draw (2.25,-0.5) node[below] {{\tiny$1$}} ; 				
				\draw (2.5,-0.5) node[below] {{\tiny$2$}} ; 				
				\draw (2.75,-0.5) node[below] {{\tiny$4$}} ; 						
	\end{tikzpicture}}} &\mapsto 
	&
	\vcenter{\hbox{
			\begin{tikzpicture}[scale=0.6]
				\draw[thick]	(0.66,2)--(2.33,1);	
				\draw[draw=white,double=black,double distance=2*\pgflinewidth,thick]	(0.66,1)--(2.33,2);		
				\draw[thick]	(0,0)--(1,0)--(1,1)--(0,1)--cycle;
				\draw[thick]	(2,0)--(3,0)--(3,1)--(2,1)--cycle;
				\draw[thick]	(0,2)--(1,2)--(1,3)--(0,3)--cycle;
				\draw[thick]	(2,2)--(3,2)--(3,3)--(2,3)--cycle;	
				\draw[thick]	(0.33,1)--(0.33,2);
				\draw[thick]	(2.66,1)--(2.66,2);
				\draw (0.5,2.5) node {{$1$}} ; 
				\draw (2.5,2.5) node {{$2$}} ; 	
				\draw (0.5,0.5) node {{$3$}} ; 
				\draw (2.5,0.5) node {{$4$}} ; 	
	\end{tikzpicture}}}\ ,
\end{array} \]
which amounts to cutting out all the global inputs and outputs and keeping just one edge between pairs of vertices that have at least one.

\begin{proposition}\label{prop:TotdsGra}
	The total sum $\oplus_{m,n\in \NN} \mathcal{P}(m,n)^{\Sy_m^{\mathrm{op}}\times \Sy_n}$ of the invariant parts of the components of any  dg properad $(\P, \d, \gamma, \eta)$ carries a natural dg Lie-graph algebra structure 
	\[
	\gra(\mu_1, \ldots, \mu_k)\coloneq 
	\sum_{\widetilde{\gra}\in \mathsf{dcGra} \atop \trunc(\widetilde{\gra})=\gra}
	\gamma\big(
	\widetilde{\gra}\left(\mu_1, \ldots, \mu_k\right)\big)~,
	\]
	which extends the Lie-admissible product of \cref{lem:TotLieAdm}. 
\end{proposition}

\begin{proof}
	Let us first notice that the set of connected directed graphs $\widetilde{\gra}$ satisfying $\trunc(\widetilde{\gra})=\gra$ is infinite but since each $\mu_i$ is a finite linear combination, the output of 
	$\gra(\mu_1, \ldots, \mu_k)$ is actually a finite sum, so it is well defined.  
	Since 
	$\barb^{-1}(\vcenter{\hbox{\begin{tikzpicture}[scale=0.3]
				\draw[thick]	(0,0)--(1,0)--(1,1)--(0,1)--cycle;
				\draw (0.5,0.5) node {{\tiny$1$}} ; 
	\end{tikzpicture}}})=\{c_{m,n}\}_{m,n\in \NN}$, where $c_{m,n}$ stands for the 1-vertex graph with $n$ inputs and $m$-outputs, 
	the 1-vertex graph 
	$\vcenter{\hbox{\begin{tikzpicture}[scale=0.3]
				\draw[thick]	(0,0)--(1,0)--(1,1)--(0,1)--cycle;
				\draw (0.5,0.5) node {{\tiny$1$}} ; 
	\end{tikzpicture}}}$ 
	acts as the identity:
	\[\vcenter{\hbox{\begin{tikzpicture}[scale=0.3]
				\draw[thick]	(0,0)--(1,0)--(1,1)--(0,1)--cycle;
				\draw (0.5,0.5) node {{\tiny$1$}} ; 
	\end{tikzpicture}}}(\mu)=\sum_{m,n\in \NN} \mu(m,n)=\mu~.\]
	For any pair $\gra_1, \gra_2\in \dcGra$ of directed simple graphs and for elements $\mu_1, \ldots , \mu_{s+t-1}\in \allowbreak \oplus_{m,n\in \NN} \allowbreak \mathcal{P}(m,n)^{\Sy_m^{\mathrm{op}}\times \Sy_n}$, where $s=|\gra_1|$ and $t=|\gra_2|$, we have to show that 
	\[
	\gra_1(\mu_1, \ldots, \mu_{i-1}, \gra_2(\mu_i,\ldots, \mu_{i+s-1}), \mu_{i+s}, \ldots, \mu_{s+t-1})=(\gra_1 \circ_i \gra_2)(\mu_1, \ldots, \mu_{s+t-1}) \ .
	\]
	The left-hand side is obtained by first applying the monadic properad structure on $\mu_i,\ldots, \mu_{i+s-1}$ along all the directed connected graphs $\widetilde{\gra}_2$ of barber type $\gra_2$ and then applying the same procedure to the result  
	and to $\mu_1, \ldots, \mu_{i-1}, \mu_i,\ldots, \mu_{i+s-1}$ along all the 
	directed connected graphs $\widetilde{\gra}_1$ of barber type $\gra_1$~. 
	The set of directed connected graphs that appear here are the ones of barber type $\gra_1\circ_i \gra_2$ by the definition of the partial composition products of the operad $\Liegra$~. 
	The axiom defining the algebra structure of the properad $\P$ over the  graph monad $\G$ concludes this part of the proof. 
	
	The induced Lie-admissible product is given by the action of the 2-vertices graph 
	$\vcenter{\hbox{
			\begin{tikzpicture}[scale=0.3]
				\draw[thick]	(0,0)--(1,0)--(1,1)--(0,1)--cycle;
				\draw[thick]	(0,2)--(1,2)--(1,3)--(0,3)--cycle;
				\draw[thick]	(0.5,1)--(0.5,2);
				\draw (0.5,0.5) node {{\scalebox{0.8}{$2$}}} ; 
				\draw (0.5,2.5) node {{\scalebox{0.8}{$1$}}} ; 
	\end{tikzpicture}}}$~, which is equal to the product $\star$ defined in \cref{lem:TotLieAdm}. 
\end{proof}

\begin{theorem}\label{prop:conv algebras are dsGra}
	The convolution algebra $\g_{\mathcal{C}, A}$ associated to an infinitesimal dg coproperad $\oC$ and a dg module $A$ carries the following natural dg Lie-graph algebra structure 
	\[
	\gra(f_1, \ldots, f_k)\ : \ \oC 
\to
	\mathcal{G}^c\big(\oC\big) \twoheadrightarrow
	\bigoplus_{\widetilde{\gra}\in \mathsf{dcGra}^{(k)} \atop \trunc(\widetilde{\gra})=\gra}
	\widetilde{\gra}\big(\oC\big)
	\xrightarrow{\widetilde{\gra}(f_1, \ldots, f_k)}
	\mathcal{G}(\End_A)
	\xrightarrow{\gamma} \End_A\ ,
	\]
	where $k=|\gra|$, which extends the Lie-admissible product of \cref{prop:ConLiead}. 
\end{theorem}

The map $\gamma$ stands for the monadic composition map of the properad $\End_A$ and 
the maps $\oC \to \widetilde{\gra}\big(\oC\big)$ are suitable iterations of the infinitesimal decomposition maps of $\oC$.
The notation $\widetilde{\gra}(f_1, \ldots, f_k)$ means that we apply $f_i : \oC \to \End_A$ to the label of the vertex $i$ of $\widetilde{\gra}$, for any $i=1,\ldots, k$~.

\begin{proof}
The computations are similar to that of \cref{prop:TotdsGra}.
Given any directed simple graph $\gra$, all the ways to iterate the infinitesimal decomposition maps of $\oC$ which produce graphs of type $\widetilde{\gra}$, such that $\barb(\widetilde{\gra})=\gra$, are equal by the axioms of an infinitesimal coproperad, and they produce finite sums, once applied to elements 
 $c\in \mathcal{C}(m,n)$.
 Therefore the action $\gra(f_1, \ldots, f_k)$ of the Lie-graph operations is well defined. 
\end{proof}

\begin{remark}\label{rem:minmax}
	If we were developing the same theory with associative algebras, that is with properads concentrated in arity $(1,1)$, we would first get  the structure of an associative algebra on the convolution algebra between an associative coalgebra and an a coassociative coalgebra. 
	This would be enough to write down the Maurer--Cartan equation. Then, the potentially bigger algebraic structure would here be that of an algebra encoded by the ladders operad. But in this case, both are isomorphic $\mathrm{Ass}\cong \mathrm{Lad}$, and one would not get more structure. 
	
	If we were working with operads, that is with properads concentrated in arities $(1,n)$, for $n \in \NN$, we would get
	first the operad of pre-Lie algebras and then the operad of rooted trees. But, again, they turn out to be isomorphic 
	$\Prelie \cong \mathrm{RT}$~, by \cite{ChapotonLivernet01}. So we would not get any additional structure in this case either. 
	
	\medskip
	
	In the present case of properads, there is significant gap between the minimal structure needed to write the Maurer--Cartan equation and the maximal structure made up of all the operations acting on the deformation complex: recall from \cref{subsec:DsGRAoperad} that the morphism of operads $\Lieadm \to \Liegra$ fails to be surjective. Thus the Lie-graph algebra structure on $\g_{\mathcal{C}, A}$ \emph{cannot} be obtained by iterating its Lie-admissible product. This key remark conceptually explains the discrepancy between the properadic calculus developed below and the operadic calculus developed in \cite{DotsenkoShadrinVallette16}.
	
	\medskip
	
	A similar phenomenon can be observed on the level of modular operads, see \cite{DSVV20}, where the operad of connected graphs $\mathrm{CGra}$ is the operad of all the natural operations acting on the totalisation of a modular operad while the operad $\Delta\Lie$ encodes the minimal structure required to write down the Maurer--Cartan equation. 
The totalisation of any dg prop, admits a Lie-admissible structure, via the usual vertical composition formula, which extends to a canonical $\Liencgra$-algebra structure. The first classical example is the $\Liencgra$-structure is given by the Weyl algebra $S(V\oplus V^*)$, see \cite[Section~1]{BV24} for instance, which is actually the totalisation of the endomorphism prop of the chain complex $V$. We refer the reader to \cref{Fig:AlgStrDefTh} for survey of these various cases. 
	
	\begin{figure*}[h!]
		\[\begin{array}{|c|c|c|c|}
			\cline{2-3}
			\multicolumn{1}{r|}{}  & \rule{0pt}{10pt} \textsc{minimal structure} &\textsc{maximal structure}  \\
			\cline{1-3}
			\textsc{algebras}
			&\rule{0pt}{10pt}  \mathrm{Ass}
			& \mathrm{Ass}\cong\mathrm{Lad}
			\\
			\cline{1-3}
			\textsc{operads} &\rule{0pt}{10pt} \Prelie
			&  \Prelie\cong \mathrm{RT}
			\\
			\cline{1-3}
			\textsc{modular operads} 
			&\rule{0pt}{10pt} \Delta\Lie
			& \mathrm{CGra} 
			\\
			\cline{1-3}
			\textsc{properads} 
			&\rule{0pt}{10pt} \Lieadm
			& \Liegra
			\\
			\cline{1-3}
			\textsc{props} 
			&\rule{0pt}{10pt} \Lieadm
			& \Liencgra
			\\
			\cline{1-3}
		\end{array}\]
		
		\caption{The various algebraic structures present in the operadic deformation theory.}
		\label{Fig:AlgStrDefTh}
	\end{figure*}
\end{remark}

\subsection{The complete setting}\label{subsec:complete}
Like the classical integration theory of Lie algebras, the integration theory of Lie-graph algebras, that we will develop in the next section~\ref{sec:IntDsGraAlg}, will require to consider infinite series of elements. In order to control their behavior, we need complete topologies on the underlying chain complexes. 

\begin{definition}[Complete differential graded module]
	A \emph{complete differential graded (dg) module} $(A, \F, \d)$ is a chain complex in the category of complete $\k$-modules. In other words, each degreewise component is equipped with a descending filtration 
	$$A_n=\F_0 A_n \supset \F_1 A_n \supset \F_2 A_n \supset \cdots \supset \F_k A_n \supset \F_{k+1}A_n \supset \cdots$$
	made up of  submodules  satisfying $\d(\F_k A_n)\subset \F_k A_{n-1}$ and 
	which is complete. That is, the canonical map 
	\[A\xrightarrow{\cong}\widehat{A}\coloneq \varprojlim_k A/\F_kA\] 
	is an isomorphism.
\end{definition}

We refer the reader for instance to \cite[Section~2]{DotsenkoShadrinVallette22} for a detailed exposition of the properties of complete differential graded modules. 
Let us quickly recall the main ones used throughout this text. 
The internal Hom structure is given by 
$$
\F_k \Hom(A, B)\coloneq
\big\{
f :A \to B\ | \ f(\F_l A)\subset \F_{k+l} B\ , \ \forall n\in \NN
\big\}\ .
$$
We denote with small letters the subspace of maps which preserve the respective filtrations: 
\[\hom(A,B)\coloneq \F_0 \Hom(A,B)\ .\] 
The induced filtration on  the tensor product of two filtered modules is given by 
\[{\F}_k (A\otimes B)\coloneq\sum_{j+l=k} \mathrm{Im} \big(\F_j A \otimes \mathrm{G}_l B\to A\otimes B\big)\]
and their complete tensor product is defined by 
$$ A\widehat{\otimes} B\coloneq\widehat{A\otimes B}\ . $$

\begin{lemma}[{\cite[Chapter~2, Lemma~2.19]{DotsenkoShadrinVallette22}}]\label{lem:smc}
	The category of complete dg modules forms a bicomplete closed symmetric  monoidal  category, whose monoidal product preserves colimits. 
\end{lemma}

This allows one to perform operadic and properadic calculus in this setting, see \cite[Chapter~2]{DotsenkoShadrinVallette22} and \cite[Sections~1-2]{HLV19}. For instance a \emph{complete} dg Lie-graph algebra is a complete dg module 
$(A, \F, \d)$ equipped with operations 
\[\gra \ : \ A^{\otimes |\gra|} \to A~, \]
for any directed simple graph $\gra\in \dcGra$, satisfying the axioms of a dg Lie-graph algebra and preserving the filtration, that is 
\[\gra \ : \ \F_{i_1} A {\otimes} \cdots {\otimes} \F_{i_k} A \to \F_{i_1+\cdots+i_k} A~. \]

\medskip

Actually, the convolution algebras associated to conilpotent coproperads always admit a compatible complete filtration as follows. 
We denote by 
\[\mathcal{G}^c(\mathcal{M})^{(k)} \coloneq \bigoplus_{\gra \in \mathsf{dcGra}^{(k)}} \gra(\mathcal{M}) \ ,\]
the summand of the graphs comonad spanned by directed connected graphs with $k$ vertices.

\begin{definition}[Coradical filtration]
The \emph{coradical filtration} of the coaugmentation coideal $\oC$ of a  conilpotent coproperad is defined by 
\[\mathscr{R}_k \oC \coloneq \left\{c\in \oC \ |\ \widetilde{\Delta}(c) \in \bigoplus_{l=1}^{k}\mathcal{G}^c\Big(\oC\Big)^{(l)}\right\}
~. \]
\end{definition}
By definition, this filtration is increasing and exhaustive:
\[
0=\mathscr{R}_{0} \oC\subset 
\mathscr{R}_1 \oC\subset \mathscr{R}_2 \oC\subset \cdots \subset \mathscr{R}_k \oC\subset \mathscr{R}_{k+1} \oC\subset \cdots 
\quad \& \quad 
\bigcup_{k \geqslant 1} \mathscr{R}_k \oC =\oC
~. 
\]
The first non-trivial part of the coradical filtration is the space of primitive elements:
\[\mathscr{R}_1 \oC=\left\{c\in \oC \ |\ \widetilde{\Delta}(c)=c \right\}~.\]

\begin{definition}[Canonical filtration]
	The \emph{canonical filtration} of the convolution algebra $\g_{\C, A}$ associated to a conilpotent dg coproperad $\C$ and a dg module $A$ is defined by 
	\[\F_k \g_{\C, A} \coloneq \left\{f : \oC \to \End_A\ ; \ f|_{\mathscr{R}_{k-1} \oC}= 0
	\right\}~. \]
\end{definition}

\begin{proposition}\label{prop:CompleteDcGra}
	For any conilpotent dg coproperad $\C$ and any dg module $A$, the convolution algebra
	\[\left(\g_{\mathcal{C}, A}, \F, \partial, \{\gra\}_{\gra \in \dcGra}\right)~,\]
	equipped with the canonical filtration,
	forms a complete dg Lie-graph algebra. 
\end{proposition}

\begin{proof}
	First, the canonical filtration is obviously decreasing
	\[
	\g_{\mathcal{C}, A}=\F_0 \g_{\mathcal{C}, A}= 
	\F_1 \g_{\mathcal{C}, A} \supset \F_2 \g_{\mathcal{C}, A} \supset \cdots \supset \F_k \g_{\mathcal{C}, A}\subset 
	\F_{k+1} \g_{\mathcal{C}, A}\supset \cdots 
	\]
	and, since the coradical filtration is exhaustive, the canonical filtration is complete. 
	Then, since the underlying differential $d_{\oC}$ of $\oC$ is a coderivation with respect to the graphs comonadic coalgebra structure, it preserves the coradical filtration. This implies that the differential 
	\[\partial f = \partial_A \circ f - (-1)^{|f|} f \circ d_{\oC}\] 
	of the convolution algebra 
	preserves the canonical filtration. 
	Finally, let us show that the operations of the Lie-graph algebra structure preserve the canonical filtration. 
	For any $k\geqslant 1$, let $\gra\in \dcGra^{(k)}$ and let 
	$f_{j} \in \F_{i_j} \g_{\mathcal{C}, A}$~, for any $1\leqslant j \leqslant k$~. 
	We have to show that $\gra(f_1, \ldots, f_k)\in \F_{i_1+\cdots +i_k} \g_{\mathcal{C}, A}$~, that is 
	$\gra(f_1, \ldots, f_k)|_{\mathscr{R}_{i_1+\cdots+i_k-1}\oC}
	=0$~. 
	For any $c\in \mathscr{R}_{i_1+\cdots+i_k-1}\oC$, the projection of $\widetilde{\Delta}(c)$  onto $\mathcal{G}^c\left(\oC\right)^{(k)}$ is a sum of graphs with vertices labelled by elements of $\oC$ such that at least one of them lives in $\mathscr{R}_{i_j-1}\oC$~,  for $1\leqslant j\leqslant k$, by the coassociativity of the graphs comonad: otherwise $\widetilde{\Delta}(c)$ would have one non-trivial term in $\mathcal{G}^c\left(\oC\right)^{(i_1+\cdots+i_k)}$~, which would contradict $c\in \mathscr{R}_{i_1+\cdots+i_k-1}\oC$~.   Thus we get $\gra(f_1, \ldots, f_k)(c)=0$~. 
\end{proof}
	
One can suppress the conilpotency assumption of \cref{prop:CompleteDcGra} at the cost of requiring completeness from the dg module $A$. This is a pertinent assumption to drop, already in the operadic context, in which the deformation complex encoding curved $L_\infty$-algebra structures on $A$
is provided by the convolution algebra $\g_{\mathsf{uCom^*},A}$ associated to the infinitesimal cooperad $\mathsf{uCom^*}$ of counital cocommutative coalgebras which fails to be conilpotent, see \cite[Section~4.3]{DotsenkoShadrinVallette16} and  \cite{calaque2021lie} for instance. 
This point is crucial for the twisting procedure described in \cite[Section~4]{CV25II}. We denote
by $\hom$ the filtered space of continuous maps and 
 by $\eend_A$ the complete endomorphism properad made up of continuous maps. 

\begin{proposition}\label{prop:nonconilpotent convolution}
For any filtered infinitesimal dg coproperad $\oC$ and any complete dg module $A$, the convolution algebra
	\[\mathfrak{g}_{\mathcal{C}, A}\coloneq \left(\hom_{\Sy}\left(\oC, \eend_A\right), \F, \partial, \{\gra\}_{\gra \in \dcGra}\right)~,\]
	equipped with the filtration induced by $A$,
	forms a complete dg Lie-graph algebra.
\end{proposition}

\begin{proof}
The formulas and the arguments \cref{prop:conv algebras are dsGra} still apply in the complete setting: the Lie-graph operations are well defined and they preserve the respective filtrations. 
\end{proof}

Let us conclude this section with a variant of the Lie-graph operad which takes into account the numbers of edges between vertices and 
which is complete. 

\begin{definition}[Directed multiple-edges graph] 
A \emph{directed multiple-edges graph} is a connected graph $\gra$, directed by a global flow from top to bottom, with at least one vertex, with no input, no output and an arbitrary number of edges between two vertices. 
We denote their set by $\dmGra$.
\end{definition}

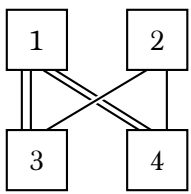
\begin{figure*}[h]
	\begin{tikzpicture}[scale=0.8]
		\draw[thick]	(0.8,2)--(2.4,1);	
		\draw[thick]	(0.6,2)--(2.2,1);	
		\draw[draw=white,double=black,double distance=2*\pgflinewidth,thick]	(0.66,1)--(2.33,2);		
		\draw[thick]	(0,0)--(1,0)--(1,1)--(0,1)--cycle;
		\draw[thick]	(2,0)--(3,0)--(3,1)--(2,1)--cycle;
		\draw[thick]	(0,2)--(1,2)--(1,3)--(0,3)--cycle;
		\draw[thick]	(2,2)--(3,2)--(3,3)--(2,3)--cycle;	
		\draw[thick]	(0.25,1)--(0.25,2);
		\draw[thick]	(0.4,1)--(0.4,2);
		\draw[thick]	(2.66,1)--(2.66,2);
		\draw (0.5,2.5) node {{$1$}} ; 
		\draw (2.5,2.5) node {{$2$}} ; 	
		\draw (0.5,0.5) node {{$3$}} ; 
		\draw (2.5,0.5) node {{$4$}} ; 	
	\end{tikzpicture}
	\caption{An example of a directed multiple-edges graph.}
	\label{Fig:2LevDCGraph}
\end{figure*}

On the $\k$-linear $\Sy$-module $\k \dmGra$ spanned by $\dmGra$, we consider 
the filtration defined by the number of edges: 
\[\mathscr{F}_k \k \dmGra \coloneq \bigoplus_{e\geqslant k} \k \dmGra_{[e]}~,\]
where $\dmGra_{[e]}$ stands for the set of directed multiple-edges graphs with $e$ edges. 
We consider partial composition products $\gra_1 \circ_i \gra_2$ defined in the same way as for the operad $\Liegra$ replacing Point (2) by: 
\begin{enumerate}
	\item[(2')] for any edge joining a vertex $j$ to the vertex $i$ in $\g_1$, considering the sum of all the possible 
	ways to join with one edge the vertex $j$ to one vertex of $\gra_2$.
\end{enumerate}

So this time, the number of edges is preserved. 

\begin{lemma}\label{lem:Lie-mgra is an operad}
	The data $\Liemgra\coloneq  (\k \dmGra, \mathscr{F}, \{\circ_i\}, \vcenter{\hbox{\begin{tikzpicture}[scale=0.3]
				\draw[thick]	(0,0)--(1,0)--(1,1)--(0,1)--cycle;
				\draw (0.5,0.5) node {{\tiny$1$}} ; 
	\end{tikzpicture}}} 
	)$ forms a filtered algebraic operad. 
\end{lemma}

\begin{proof}
The axioms are straightforward to check, like in the proof of \cref{lem:Lie-gra is an operad}. 
\end{proof}

\begin{definition}[Complete Lie-mgraph operad]\label{def:Lie-mgra}
The \emph{complete operad }$\widehat{\Liemgra}$ is defined as the completion of the operad $\Liemgra$: 
		\[\widehat{\Liemgra}\coloneq \left(\widehat{\k\dmGra}, \mathscr{F}, \{\circ_i\}, \vcenter{\hbox{\begin{tikzpicture}[scale=0.3]
					\draw[thick]	(0,0)--(1,0)--(1,1)--(0,1)--cycle;
					\draw (0.5,0.5) node {{\tiny$1$}} ; 
		\end{tikzpicture}}} 
		\right)~.\]
\end{definition}

Elements of $\widehat{\Liemgra}(n)$ are series of linear combination labelled by the number of edges 
of directed multiple-edges graphs with $n$ vertices. 
There is an injective morphism of operads 
\[\Liegra \rightarrowtail \widehat{\Liemgra}\]
 given by assigning a directed simple graph $\gra$ to the sum of all directed multiple-edges graphs with the same vertices as $\gra$ and an arbitrary positive number of edges between vertices $i$ and $j$, whenever there was an edge between $i$ and $j$ in $\gra$.

\begin{remark}
The complete operad $\widehat{\Liemgra}(n)$ is similar to the graphs operads considered in the study of the little discs operad, the configuration spaces of points, and the Grothendieck--Teichm\"uller Lie algebra  \cite{Kontsevich97, LambrechtsVolic14, Willwacher15, Idrissi19, CW16, DotsenkoShadrinVallette22} since its composition maps  preserve the number of edges. One can consider a version of it where internal edges receive a certain degree. 
\end{remark}

\section{Integration theory of Lie-graph algebras}\label{sec:IntDsGraAlg}

The Lie bracket of any complete Lie algebra $\g$, like the one associated the complete convolution algebra $\g_{\C, A}$, induces the following usual Hausdorff type group. Recall that the \emph{Baker--Campbell--Hausdorff formula} 
is  the element in the free associative algebra on two generators  $x$ and $y$, that is the algebra of formal power series on  $x$ and $y$, 
given by 
\[\BCH(x,y)\coloneq \ln\left(e^xe^y\right)~.\]
The celebrated theorem of Baker, Campbell, and Hausdorff, see \cite{BF12}, states that this formula can be written using only the Lie bracket $[a,b]=a\otimes b - b\otimes a$:
\[
\BCH(x,y)=x+y+\tfrac12[x,y]+\tfrac{1}{12}[x,[x,y]]+\tfrac{1}{12}[y,[x,y]]+\cdots~.\\
\]
In other words, the  Baker--Campbell--Hausdorff formula $\BCH(x,y)$ lives in the free complete Lie algebra on two generators. 

\begin{definition}[Gauge group]
	The \emph{gauge group} associated to a complete Lie algebra $(\g, \{\F_k \g\}_{k\geqslant 0}, [\ ,\, ])$ is the complete group
	defined by 
	\[\Gamma\coloneq\big( \F_1 \g_0,  \, \{\F_k \g_0\}_{k\geqslant 1},\,  \BCH, \,0\big)\ ,\]
\end{definition}

We refer the reader to \cite[Definition~3.8]{DotsenkoShadrinVallette22} for the notion of a complete group, which is a group equipped with a compatible and complete filtration. Since the $\BCH$ formula is quite involved, one would appreciate to be able to use another more amenable form for this gauge group.  
When the Lie bracket comes from a pre-Lie product, the BCH formula is actually given by a canonical closed formula based on 2-leveled rooted trees, see \cite{DotsenkoShadrinVallette16} and \cite[Chapter~3]{DotsenkoShadrinVallette22} in the complete case. 

\medskip

The purpose of this section is to simplify the $\BCH$-formula for Lie algebras arising from the particular structure of a \emph{Lie-graph algebra}. To this extent, we first introduce a deformation gauge group, which, in the case of the convolution algebra, is actually made up of the elements of the convolution monoid 
$\a_{\mathcal{C}, A}$ whose first component is the identity of $A$. Then, we coin a new exponential map, based on directed simple graphs, that provides us with an isomorphism between the gauge group and the deformation gauge group. Finally, we conclude this section with a closed graphical formula for their action on Maurer--Cartan elements.

\subsection{Deformation gauge group}\label{sec:gauge groups}

\begin{definition}[$k$-leveled directed simple graphs]\label{def:levGrap}
	For any $k\geqslant 1$, we consider the set $\multildcGra{k}_{\Sy}$ of 
	$k$\emph{-leveled directed simple graphs}, which are  
	directed simple graphs with the data of $k$ levels, along which sit the unlabelled vertices. The orientation of the graph should respect strictly the order of the levels: when one passes from an upper vertex to a lower vertex along one edge then one should go from an upper level to a strictly lower level. One vertex can sit on only one level and empty levels are allowed.	\end{definition}

\begin{figure}[h]
	$\vcenter{\hbox{\begin{tikzpicture}[scale=0.6]
				\draw[dotted] (-0.5, 0.5)--(3.5,0.5);
				\draw[dotted] (-0.5, 2.5)--(3.5,2.5);
				\draw[dotted] (-0.5, 4.5)--(3.5,4.5);		
				\draw[thick, fill=white]	(0,2)--(1,2)--(1,3)--(0,3)--cycle;
				\draw[thick, fill=white]	(2,2)--(3,2)--(3,3)--(2,3)--cycle;
				\draw[thick, fill=white]	(1,0)--(2,0)--(2,1)--(1,1)--cycle;
				\draw[thick, fill=white]	(1,4)--(2,4)--(2,5)--(1,5)--cycle;	
				\draw[thick]	(0.5,2)--(1.33,1);
				\draw[thick]	(2.5,2)--(1.66,1);
				\draw[thick]	(1.5,1)--(1.5,4);
	\end{tikzpicture}}}$
	\hspace{1cm}
	$\vcenter{\hbox{\begin{tikzpicture}[scale=0.6]
				\draw[dotted] (-0.5, 0.5)--(3.5,0.5);
				\draw[dotted] (-0.5, 2.5)--(3.5,2.5);
				\draw[dotted] (-0.5, 4.5)--(3.5,4.5);		
				\draw[thick, fill=white]	(0,4)--(1,4)--(1,5)--(0,5)--cycle;
				\draw[thick, fill=white]	(2,4)--(3,4)--(3,5)--(2,5)--cycle;
				\draw[thick, fill=white]	(1,0)--(2,0)--(2,1)--(1,1)--cycle;
				\draw[thick, fill=white]	(1,2)--(2,2)--(2,3)--(1,3)--cycle;	
				\draw[thick]	(0.5, 4)--(0.5,2)--(1.33,1);
				\draw[thick]	(2.5, 4)--(2.5,2)--(1.66,1);
				\draw[thick]	(1.5,1)--(1.5,2);
	\end{tikzpicture}}}$
	\hspace{1cm}
	$\vcenter{\hbox{\begin{tikzpicture}[scale=0.6]
				\draw[white] (0,5) node {};
				\draw[dotted] (-1, 0.5)--(4,0.5);
				\draw[dotted] (-1, 2.5)--(4,2.5);
				\draw[dotted] (-1, 4.5)--(4,4.5);	
				\draw[thick, fill=white]	(-0.5,2)--(0.5,2)--(0.5,3)--(-0.5,3)--cycle;
				\draw[thick, fill=white]	(2.5,2)--(3.5,2)--(3.5,3)--(2.5,3)--cycle;
				\draw[thick, fill=white]	(1,0)--(2,0)--(2,1)--(1,1)--cycle;
				\draw[thick, fill=white]	(1,2)--(2,2)--(2,3)--(1,3)--cycle;	
				\draw[thick]	(0,2)--(1.33,1);
				\draw[thick]	(3,2)--(1.66,1);
				\draw[thick]	(1.5,1)--(1.5,2);
	\end{tikzpicture}}}$
	\caption{Three different $3$-leveled graphs.}
	\label{Fig:leveledGra}
\end{figure}
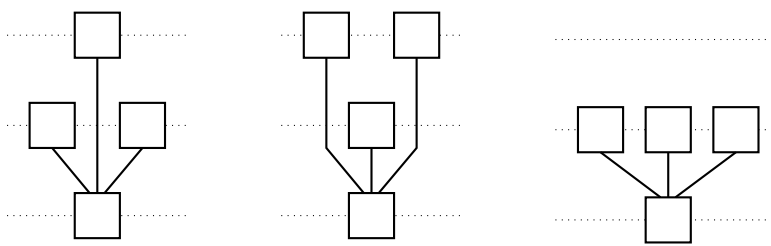

\begin{definition}[Automorphisms of $k$-leveled directed simple graphs]\label{def:AutGra}
	The \emph{automorphism group} of  a $k$-leveled directed simple graphs $\gra$ is the group of permutations of vertices preserving both the adjacency condition and the level structure. We denote it by $\Aut(\gra)$~. 
\end{definition}

\begin{example} 
	In the three cases depicted on \cref{Fig:leveledGra}, the  automorphism groups are respective isomorphic to $\Sy_2$, $\Sy_2$, and $\Sy_3$~. 
\end{example}

\begin{definition}[$\cc$-product]
Given any complete Lie-graph algebra $(\g, \F, \{\gra\}_{\gra\in \dcGra})$, we consider the set 
\[\1+\F_1\g_0\coloneq \{\1+x, \, x \in \F_1\g_0\}~,\]
where $\1$ is a formal symbol. 
We equip it with the following product 
\[
(\1+x)\circledcirc (\1+y)\coloneq \1 +\ \sum_{\gra \in {\ldcGra}_{\Sy}} {\textstyle \frac{1}{|\Aut(\gra)|}}\,  \gra(x,y)\ , 
\]
where $\gra(x,y)$ stands for the Lie-graph action of $\gra$ with bottom vertices labelled by $x$ and top vertices labelled by $y$, see \cref{Fig:GrpProd}.
\end{definition}

\begin{figure*}[h]
	\[(\1+x)\circledcirc (\1+y)\ =\ \1 \  + \
	\vcenter{\hbox{\begin{tikzpicture}[scale=0.4]
				\draw[thick]	(0,0)--(1,0)--(1,1)--(0,1)--cycle;
				\draw (0.5,0.5) node {{$x$}} ; 
	\end{tikzpicture}}}
	\ + \
	\vcenter{\hbox{\begin{tikzpicture}[scale=0.4]
				\draw[thick]	(0,0)--(1,0)--(1,1)--(0,1)--cycle;
				\draw (0.5,0.5) node {{$y$}} ; 
	\end{tikzpicture}}}
	\ + \
	\vcenter{\hbox{\begin{tikzpicture}[scale=0.4]
				\draw[thick]	(0,0)--(1,0)--(1,1)--(0,1)--cycle;
				\draw[thick]	(0,2)--(1,2)--(1,3)--(0,3)--cycle;
				\draw[thick]	(0.5,1)--(0.5,2);
				\draw (0.5,0.5) node {{$x$}} ; 
				\draw (0.5,2.5) node {{$y$}} ; 
	\end{tikzpicture}}}
	\ + \
	\scalebox{1.1}{$\frac{1}{2}$}\ 
	\vcenter{\hbox{\begin{tikzpicture}[scale=0.4]
				\draw[thick]	(0,0)--(1,0)--(1,1)--(0,1)--cycle;
				\draw[thick]	(2,0)--(3,0)--(3,1)--(2,1)--cycle;
				\draw[thick]	(1,2)--(2,2)--(2,3)--(1,3)--cycle;
				\draw[thick]	(0.5,1)--(1.33,2);
				\draw[thick]	(2.5,1)--(1.66,2);
				\draw (1.5,2.5) node {{$y$}} ; 
				\draw (0.5,0.5) node {{$x$}} ; 
				\draw (2.5,0.5) node {{$x$}} ; 	
	\end{tikzpicture}}}
	\ + \ 
	\scalebox{1.1}{$\frac{1}{2}$}\ 
	\vcenter{\hbox{\begin{tikzpicture}[scale=0.4]
				\draw[thick]	(0,2)--(1,2)--(1,3)--(0,3)--cycle;
				\draw[thick]	(2,2)--(3,2)--(3,3)--(2,3)--cycle;
				\draw[thick]	(1,0)--(2,0)--(2,1)--(1,1)--cycle;
				\draw[thick]	(0.5,2)--(1.33,1);
				\draw[thick]	(2.5,2)--(1.66,1);
				\draw (1.5,0.5) node {{$x$}} ; 
				\draw (0.5,2.5) node {{$y$}} ; 
				\draw (2.5,2.5) node {{$y$}} ; 	
	\end{tikzpicture}}}
	\ + \ 
	\scalebox{1.1}{$\frac{1}{4}$}\ 
	\vcenter{\hbox{\begin{tikzpicture}[scale=0.4]
				\draw[thick]	(0.66,2)--(2.33,1);	
				\draw[draw=white,double=black,double distance=2*\pgflinewidth,thick]	(0.66,1)--(2.33,2);		
				\draw[thick]	(0,0)--(1,0)--(1,1)--(0,1)--cycle;
				\draw[thick]	(2,0)--(3,0)--(3,1)--(2,1)--cycle;
				\draw[thick]	(0,2)--(1,2)--(1,3)--(0,3)--cycle;
				\draw[thick]	(2,2)--(3,2)--(3,3)--(2,3)--cycle;	
				\draw[thick]	(0.33,1)--(0.33,2);
				\draw[thick]	(2.66,1)--(2.66,2);
				\draw (0.5,2.5) node {{$y$}} ; 
				\draw (2.5,2.5) node {{$y$}} ; 	
				\draw (0.5,0.5) node {{$x$}} ; 
				\draw (2.5,0.5) node {{$x$}} ; 	
	\end{tikzpicture}}}
	\ + \ \cdots\]
	\caption{The first terms of the  product $\cc$\ .}
	\label{Fig:GrpProd}
\end{figure*}

\begin{theorem}\label{thm:dcGraGp}
	The above assignment defines a functor from complete Lie-graph algebras to complete groups: 
	\[ \begin{array}{ccc}
		\textrm{complete}\  \textrm{Lie-graph algebras} & \to & \textrm{complete groups}\\
		\left(\g,\{\F_k \g\}_{k\geqslant 0}, \{\gra\}_{\gra \in \dcGra}\right) &\mapsto & \left(\1+\F_1\g_0, \{\1+\F_k \g_0\}_{k\geqslant 1}, \circledcirc, \1\right)\ .
	\end{array} \]
\end{theorem}

\begin{proof}
	We will show that, for the free complete Lie-graph algebra on three generators $x,y, z$, the following identity holds 
	\begin{equation}\tag{$\ast$}\label{eq:assoc}
		((\1+x)\circledcirc (\1+y)) \circledcirc (\1+z) = \1 + \sum_{\gra \in \multildcGra{3}_{\Sy}} {\textstyle \frac{1}{|\Aut(\gra)|}} \, \gra(x,y,z)~,
	\end{equation}
	where $\gra(x,y,z)$ denotes the action of $\gra$ with all the vertices of  the top level labelled by $z$, all the vertices of the middle level labelled by $y$~, and all the vertices of the bottom level labelled by $x$~. 
	The symmetry of the argument gives the same formula for the other parenthesization which guarantees the associativity of $\circledcirc$~.
	
	\medskip
	
	Since the operadic structure on $\Liegra$ is given by the insertion of directed graphs, we claim, from the definition of the product $\circledcirc$ given in terms of 2-leveled graphs, that we can express its right iteration 
	using graphs with three levels, that is 
	\[((\1+x)\circledcirc (\1+y)) \circledcirc (\1+z)=\1+\displaystyle\sum_{\gra \in \multildcGra{3}_{\Sy}} \chi_\gra\ \gra(x,y,z)\]for some coefficients $\chi_\gra$ that we wish to compute. Indeed the left-hand side is made up of terms of the form 
	$(\cdots((\gra'\circ_r\gra_r)\circ_{r-1} \gra_{r-1}) \cdots )\circ_1 \gra_1$, with $\gra_i$ being $2$-leveled graphs in $x$ and $y$, and $\gra'$ being a $2$-leveled graph with top level vertices labelled by $z$ and bottom level vertices labelled by $1,\dots,r$. Such terms have an obvious 3-leveled structure. 
	
	\medskip
	
	In the other way round, let us fix a $3$-leveled graph $\gra \in \multildcGra{3}_{\Sy}$. We start by noticing that for $\gra$ to appear as a summand in $(\cdots((\gra'\circ_r\gra_r)\circ_{r-1} \gra_{r-1}) \cdots )\circ_1 \gra_1$, both the graph $\gra'$ and the set of subgraphs $\{\gra_1,\dots ,\gra_r\}$ (accounting for multiplicities) are automatically determined. The latter graphs are obtained, up to permutation, by cutting out  the top level of $\gra$ and the former graph $\gra'$ is obtained from $\gra$ by contracting each sub-graph $\gra_1,\dots ,\gra_r$ into a vertex  and by 
	deleting parallel edges, see \cref{Fig:Dec3LevGra} for an example. 
	
	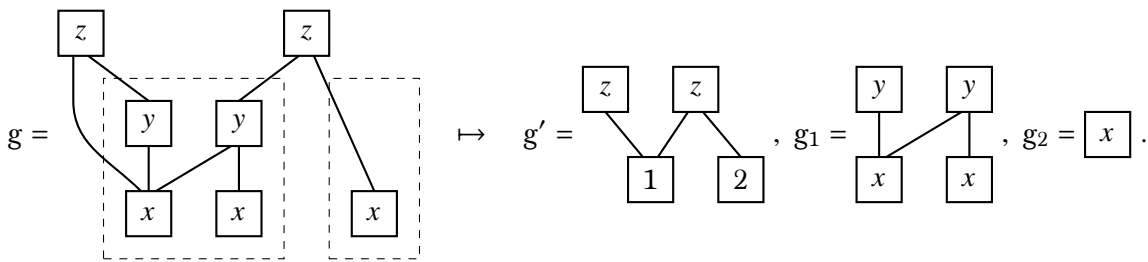
\begin{figure*}[h]
		\[
		\gra =
		\vcenter{\hbox{\begin{tikzpicture}[scale=0.6]
					\draw[dashed] (0.5, -0.5)--(4.5, -0.5)--(4.5,3.5)--(0.5, 3.5)--cycle;
					\draw[dashed] (5.5, -0.5)--(7.5, -0.5)--(7.5,3.5)--(5.5, 3.5)--cycle;	
					\draw[thick, fill=white]	(-0.5,4)--(0.5,4)--(0.5,5)--(-0.5,5)--cycle;
					\draw[thick, fill=white]	(4.5,4)--(5.5,4)--(5.5,5)--(4.5,5)--cycle;
					\draw[thick, fill=white]	(1,0)--(2,0)--(2,1)--(1,1)--cycle;
					\draw[thick, fill=white]	(3,0)--(4,0)--(4,1)--(3,1)--cycle;	
					\draw[thick, fill=white]	(6,0)--(7,0)--(7,1)--(6,1)--cycle;	
					\draw[thick, fill=white]	(1,2)--(2,2)--(2,3)--(1,3)--cycle;	
					\draw[thick, fill=white]	(3,2)--(4,2)--(4,3)--(3,3)--cycle;		
					\draw[thick]	(-0.17, 4) to[out=270,in=90] (-0.17,3) to[out=270,in=135] (1.33,1);		
					\draw[thick]	(0.16,4)--(1.5,3);
					\draw[thick]	(4.83,4)--(3.5,3);	
					\draw[thick]	(1.66,1)--(3.33,2);
					\draw[thick]	(1.5,1)--(1.5,2);
					\draw[thick]	(3.5,1)--(3.5,2);	
					\draw[thick]	(6.5,1)--(5.16,4);		
					\draw (1.5,0.5) node {{$x$}} ; 
					\draw (3.5,0.5) node {{$x$}} ; 	
					\draw (6.5,0.5) node {{$x$}} ; 		
					\draw (0,4.5) node {{$z$}} ; 
					\draw (5,4.5) node {{$z$}} ; 	
					\draw (1.5,2.5) node {{$y$}} ; 	
					\draw (3.5,2.5) node {{$y$}} ; 		
		\end{tikzpicture}}}
		\quad\mapsto\quad
		\gra' =
		\vcenter{\hbox{\begin{tikzpicture}[scale=0.6]
					\draw[thick, fill=white]	(3,0)--(4,0)--(4,1)--(3,1)--cycle;
					\draw[thick, fill=white]	(1,0)--(2,0)--(2,1)--(1,1)--cycle;
					\draw[thick, fill=white]	(0,2)--(1,2)--(1,3)--(0,3)--cycle;	
					\draw[thick, fill=white]	(2,2)--(3,2)--(3,3)--(2,3)--cycle;		
					\draw[thick]	(1.33,1)--(0.5,2);
					\draw[thick]	(3.5,1)--(2.66,2);		
					\draw[thick]	(1.66,1)--(2.33,2);			
					\draw (1.5,0.5) node {{$1$}} ; 		
					\draw (3.5,0.5) node {{$2$}} ; 		
					\draw (0.5,2.5) node {{$z$}} ; 		
					\draw (2.5,2.5) node {{$z$}} ; 			
		\end{tikzpicture}}}\ ,\
		\gra_1 =
		\vcenter{\hbox{\begin{tikzpicture}[scale=0.6]
					\draw[thick, fill=white]	(1,0)--(2,0)--(2,1)--(1,1)--cycle;
					\draw[thick, fill=white]	(3,0)--(4,0)--(4,1)--(3,1)--cycle;	
					\draw[thick, fill=white]	(1,2)--(2,2)--(2,3)--(1,3)--cycle;	
					\draw[thick, fill=white]	(3,2)--(4,2)--(4,3)--(3,3)--cycle;		
					\draw[thick]	(1.66,1)--(3.33,2);
					\draw[thick]	(1.5,1)--(1.5,2);
					\draw[thick]	(3.5,1)--(3.5,2);			
					\draw (1.5,0.5) node {{$x$}} ; 
					\draw (3.5,0.5) node {{$x$}} ; 		
					\draw (1.5,2.5) node {{$y$}} ; 	
					\draw (3.5,2.5) node {{$y$}} ; 		
		\end{tikzpicture}}}\ , \
		\gra_2=
		\vcenter{\hbox{\begin{tikzpicture}[scale=0.6]
					\draw[thick, fill=white]	(1,0)--(2,0)--(2,1)--(1,1)--cycle;
					\draw (1.5,0.5) node {{$x$}} ; 
		\end{tikzpicture}}}~.
		\]
		\caption{Example of a decomposition of 3-leveled graph.}
		\label{Fig:Dec3LevGra}
	\end{figure*}
	
	\medskip
	One can actually say a bit more about the subgraphs $\gra_i$ of $\gra$: they are actually more determined 
	when one considers the incident edges, called \emph{leaves} of $\gra_i$, connecting each of them to a top vertice labelled by $z$. Such graphs are  called  \emph{marked} graphs:  they are $2$-leveled directed simple graphs with at least one leaf, that is an edge going upward with no end vertex, and such that the leaves are colored according to the vertices they are attached to, see \cref{Fig:MarkedGra}. 
	Let us denote by $\Aut^\text{mark}$ the group of  marked automorphisms that 
	is, graph automorphisms that 
	respect the configuration of colored leaves.

	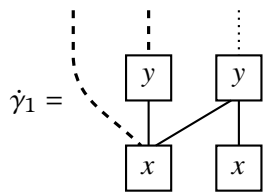
\begin{figure*}[h]
		\[
		\dot{\gamma}_1 =
		\vcenter{\hbox{\begin{tikzpicture}[scale=0.6]
					\draw[very thick, dashed]	(-0.17, 4) to[out=270,in=90] (-0.17,3) to[out=270,in=135] (1.33,1);		
					\draw[very thick, dashed]	(1.5,4)--(1.5,3);
					\draw[thick, dotted]	(3.5,4)--(3.5,3);	
					\draw[thick, fill=white]	(1,0)--(2,0)--(2,1)--(1,1)--cycle;
					\draw[thick, fill=white]	(3,0)--(4,0)--(4,1)--(3,1)--cycle;	
					\draw[thick, fill=white]	(1,2)--(2,2)--(2,3)--(1,3)--cycle;	
					\draw[thick, fill=white]	(3,2)--(4,2)--(4,3)--(3,3)--cycle;		
					\draw[thick]	(1.66,1)--(3.33,2);
					\draw[thick]	(1.5,1)--(1.5,2);
					\draw[thick]	(3.5,1)--(3.5,2);	
					\draw (1.5,0.5) node {{$x$}} ; 
					\draw (3.5,0.5) node {{$x$}} ; 		
					\draw (1.5,2.5) node {{$y$}} ; 	
					\draw (3.5,2.5) node {{$y$}} ; 		
		\end{tikzpicture}}}\]
		\caption{Example of a marked graph associated to the graph $\gra$~.}
		\label{Fig:MarkedGra}
	\end{figure*}

	\medskip
	
	Let us suppose that there are $l$ different types of marked graphs appearing among the $\gra_1,\dots ,\gra_r$ in $\gra$ and let us denote them by $\dot{\gamma}_m$, for $m=1,\dots, l$~. Suppose furthermore that each marked graph $\dot{\gamma}_m$ appears $j_m$ times in $\gra_1,\dots, \gra_r$, so that $j_1+\dots+j_l = r$~.
	One can regroup them even further as follows: consider the various configurations $\c_1, \ldots, \c_k$ of colors of the leaves with multiplicity of the marked graphs once one forgets the internal graph structure. For any $1\leqslant h\leqslant k$, we denote by $t_h$ the number of marked graphs appearing  among the $\gra_1,\dots ,\gra_r$ of colored leaves of configuration $\c_h$, so that $t_1+\cdots+t_k=r$ is a coarser partition of $r$~. 
	
	\medskip
	
	Any automorphism of the 3-leveled graph $\gra$ gives rise to an automorphism of the 2-leveled graph $\gra'$, but this latter one fails to distinguish the $t_1$ (resp. $t_2, \ldots, t_k$) marked graphs of colored leaves of type $\c_1$ (resp. $\c_2, \ldots, \c_k$). Such an automorphism of $\gra$ also permutes identical marked subgraphs and applies to each  of them a marked automorphism. In the other way round, this procedure fully characterises automorphisms of   3-leveled graphs, so we get 
	\[
	|\Aut(\gra)|= 
	\frac{|\Aut(\gra')|
		\prod_{m=1}^l j_m! \, \left|\Aut^\text{mark}\left(\dot{\gamma}_m\right)\right|^{j_m}}{t_1!\ldots t_k!}~. 
	\]
	
	Let us now compute the coefficient $\chi_\gamma$: there are 4 different coefficients.
	
	\begin{enumerate}[i)]
		\item The coefficient of $\gra'$, that is  
		\[\frac{1}{|\Aut(\gra')|}~.\]
		
		\item The coefficient of the subgraphs $\gra_i$ which is equal to 
		\[\prod_{i=1}^r \frac{1}{|\Aut(\gra_i)|}=\prod_{m=1}^l \frac{1}{|\Aut(\gamma_m)|^{j_m}}~,\]
		where the $\gamma_m$ are the underlying 2-leveled graphs of the marked graphs $\dot{\gamma}_m$ obtained by removing the leaves. 
		
		\item A multinomial coefficient corresponding to the places where the  marked graphs can be put. 
		For each $1\leqslant h \leqslant k$, we denote by $j_{m^h_1}, \ldots, j_{m^h_{n(h)}}$ the respective numbers of marked graphs that have a configuration of colored leaves of type $\c_h$, so that 
		$j_{m^h_1}+ \cdots+ j_{m^h_{n(h)}}=t_h$ and $\left\{m^1_1, \ldots, m^k_{n(k)}\right\}=\{1, \ldots, l\}$~. 
		On the $t_h$ bottom vertices of $\gra'$, there are 
		\[\frac{t_h!}{j_{m^h_1}! \ldots j_{m^h_{n(h)}}!}\] ways to place the various marked graphs of type $\c_h$~. In the end, this produces the coefficient 		
		\[
		\prod_{h=1}^k \frac{t_h!}{j_{m^h_1}!  \ldots j_{m^h_{n(h)}}!} = \frac{t_1!\dots t_k!}{j_1!\dots j_l!}~.
		\]		
		
		\item A \emph{reconnection} coefficient $\mathrm{R}\left(\dot{\gamma}_m\right)$ for each marked subgraph $\dot{\gamma}_m$, corresponding to different reconnections that yield the same final 3-leveled graph $\gamma$, see \cref{fig:same graph2} for an example. A straightforward application of the orbit-stabiliser theorem shows that 
		\[
		\mathrm{R}\left(\dot{\gamma}_m\right)=\frac{\left|\Aut(\gamma_m)\right|}{\left|\Aut^\text{mark}\left(\dot{\gamma}_m\right)\right|}~.
		\]
		
		\begin{figure}[h]
			\[\vcenter{\hbox{\begin{tikzpicture}[scale=0.6]
						\draw[thick, fill=white]	(2,4)--(3,4)--(3,5)--(2,5)--cycle;
						\draw[thick, fill=white]	(2,0)--(3,0)--(3,1)--(2,1)--cycle;		
						\draw[thick, fill=white]	(1,2)--(2,2)--(2,3)--(1,3)--cycle;	
						\draw[thick, fill=white]	(3,2)--(4,2)--(4,3)--(3,3)--cycle;		
						\draw[thick]	(2.33,1)--(1.5,2);
						\draw[thick]	(2.66,1)--(3.5,2);		
						\draw[thick]	(1.5,3)--(2.5,4);	
						\draw (2.5,0.5) node {{$x$}} ; 
						\draw (2.5,4.5) node {{$z$}} ; 	
						\draw (1.5,2.5) node {{$y$}} ; 	
						\draw (3.5,2.5) node {{$y$}} ; 		
			\end{tikzpicture}}}
			\quad = \quad  
			\vcenter{\hbox{\begin{tikzpicture}[scale=0.6]
						\draw[thick, fill=white]	(2,4)--(3,4)--(3,5)--(2,5)--cycle;
						\draw[thick, fill=white]	(2,0)--(3,0)--(3,1)--(2,1)--cycle;		
						\draw[thick, fill=white]	(1,2)--(2,2)--(2,3)--(1,3)--cycle;	
						\draw[thick, fill=white]	(3,2)--(4,2)--(4,3)--(3,3)--cycle;		
						\draw[thick]	(2.33,1)--(1.5,2);
						\draw[thick]	(2.66,1)--(3.5,2);		
						\draw[thick]	(3.5,3)--(2.5,4);	
						\draw (2.5,0.5) node {{$x$}} ; 
						\draw (2.5,4.5) node {{$z$}} ; 	
						\draw (1.5,2.5) node {{$y$}} ; 	
						\draw (3.5,2.5) node {{$y$}} ; 		
			\end{tikzpicture}}}
			\quad \text{appear in} \quad 
			\vcenter{\hbox{\begin{tikzpicture}[scale=0.6]
						\draw[thick, fill=white]	(2,4)--(3,4)--(3,5)--(2,5)--cycle;
						\draw[thick, fill=white]	(2,2)--(3,2)--(3,3)--(2,3)--cycle;		
						\draw[thick]	(2.5,3)--(2.5,4);	
						\draw (2.5,4.5) node {{$z$}} ; 	
						\draw (2.5,2.5) node {{$1$}} ; 		
			\end{tikzpicture}}}		
			\ \circ_1 \
			\vcenter{\hbox{\begin{tikzpicture}[scale=0.6]
						\draw[thick, fill=white]	(2,0)--(3,0)--(3,1)--(2,1)--cycle;		
						\draw[thick, fill=white]	(1,2)--(2,2)--(2,3)--(1,3)--cycle;	
						\draw[thick, fill=white]	(3,2)--(4,2)--(4,3)--(3,3)--cycle;		
						\draw[thick]	(2.33,1)--(1.5,2);
						\draw[thick]	(2.66,1)--(3.5,2);		
						\draw (2.5,0.5) node {{$x$}} ; 
						\draw (1.5,2.5) node {{$y$}} ; 	
						\draw (3.5,2.5) node {{$y$}} ; 		
			\end{tikzpicture}}}~. 
			\]
			\caption{These two graphs are identical but they appear twice in the $\circledcirc$ product.}
			\label{fig:same graph2}
		\end{figure}
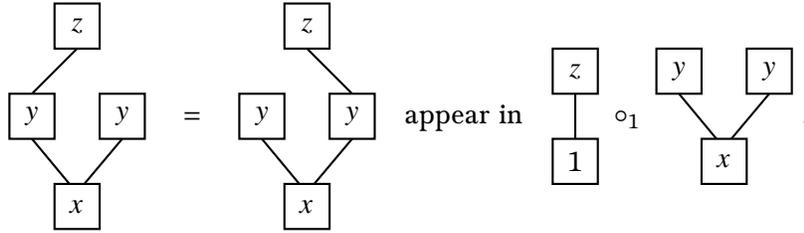
	\end{enumerate}
	
	Multiplying all the coefficients, we get 
	\[
	\chi_\gamma = 
	\underbrace{\frac{1}{|\Aut(\gra')|}}_\text{i)} \times 
	\underbrace{\prod_{m=1}^l \frac{1}{|\Aut(\gamma_m)|^{j_m}}}_\text{ii)}\times 
	\underbrace{\frac{t_1!\dots t_k!}{j_1!\dots j_l!}}_\text{iii)}\times 
	\underbrace{\prod_{m=1}^l  \frac{\left|\Aut(\gamma_m)\right|^{j_m}}{\left|\Aut^\text{mark}\left(\dot{\gamma}_m\right)\right|^{j_m}}}_\text{iv)}
	=\frac{1}{|\Aut(\gra)|}~.
	\]
	
	\medskip
	
	It remains to show that every element admits an inverse, which can be easily done  by an inductive argument on the number of vertices. We skip this part of the proof since we will provide the reader with an explicit formula for the inverse in the following proposition.
\end{proof}

In a way similar to \cref{def:levGrap}, we consider the set $\dcGra_{\Sy}$ of directed simple graphs with unlabelled vertices and no level structure this time. The group $\Aut(\gra)$ of automorphisms of such a graph $\gra$ is made up of permutations of vertices which respect the  adjacency condition. 

\begin{proposition}\label{prop:Inverse}
	The inverse of an element $\1+x$ in the complete group associated to a complete Lie-graph algebra is given by 
	\[(\1+x)^{-1}=\1+\displaystyle \sum_{\gra \in \dcGra_{\Sy}} \frac{(-1)^{|\gra|}}{|\mathrm{Aut}(\gra)|} \, \gra(x)~. \]
\end{proposition}

\begin{figure*}[h]
	\[(\1+x)^{-1}\ =\ \1 \  - \
	\vcenter{\hbox{\begin{tikzpicture}[scale=0.4]
				\draw[thick]	(0,0)--(1,0)--(1,1)--(0,1)--cycle;
				\draw (0.5,0.5) node {{$x$}} ; 
	\end{tikzpicture}}}
	\ + \ 
	\vcenter{\hbox{\begin{tikzpicture}[scale=0.4]
				\draw[thick]	(0,0)--(1,0)--(1,1)--(0,1)--cycle;
				\draw[thick]	(0,2)--(1,2)--(1,3)--(0,3)--cycle;
				\draw[thick]	(0.5,1)--(0.5,2);
				\draw (0.5,0.5) node {{$x$}} ; 
				\draw (0.5,2.5) node {{$x$}} ; 
	\end{tikzpicture}}}
	\ - \
	\scalebox{1.1}{$\frac{1}{2}$}\ 
	\vcenter{\hbox{\begin{tikzpicture}[scale=0.4]
				\draw[thick]	(0,0)--(1,0)--(1,1)--(0,1)--cycle;
				\draw[thick]	(2,0)--(3,0)--(3,1)--(2,1)--cycle;
				\draw[thick]	(1,2)--(2,2)--(2,3)--(1,3)--cycle;
				\draw[thick]	(0.5,1)--(1.33,2);
				\draw[thick]	(2.5,1)--(1.66,2);
				\draw (1.5,2.5) node {{$x$}} ; 
				\draw (0.5,0.5) node {{$x$}} ; 
				\draw (2.5,0.5) node {{$x$}} ; 	
	\end{tikzpicture}}}
	\ - \ 
	\scalebox{1.1}{$\frac{1}{2}$}\ 
	\vcenter{\hbox{\begin{tikzpicture}[scale=0.4]
				\draw[thick]	(0,2)--(1,2)--(1,3)--(0,3)--cycle;
				\draw[thick]	(2,2)--(3,2)--(3,3)--(2,3)--cycle;
				\draw[thick]	(1,0)--(2,0)--(2,1)--(1,1)--cycle;
				\draw[thick]	(0.5,2)--(1.33,1);
				\draw[thick]	(2.5,2)--(1.66,1);
				\draw (1.5,0.5) node {{$x$}} ; 
				\draw (0.5,2.5) node {{$x$}} ; 
				\draw (2.5,2.5) node {{$x$}} ; 	
	\end{tikzpicture}}}
	\ - \ 
	\vcenter{\hbox{\begin{tikzpicture}[scale=0.4]
				\draw[thick]	(0,0)--(1,0)--(1,1)--(0,1)--cycle;
				\draw[thick]	(0,2)--(1,2)--(1,3)--(0,3)--cycle;
				\draw[thick]	(0,4)--(1,4)--(1,5)--(0,5)--cycle;
				\draw[thick]	(0.5,1)--(0.5,2);
				\draw[thick]	(0.5,3)--(0.5,4);	
				\draw (0.5,0.5) node {{$x$}} ; 
				\draw (0.5,2.5) node {{$x$}} ; 
				\draw (0.5,4.5) node {{$x$}} ; 
	\end{tikzpicture}}}
	\ - \ 
	\vcenter{\hbox{\begin{tikzpicture}[scale=0.4]
				\draw[thick]	(0,0)--(1,0)--(1,1)--(0,1)--cycle;
				\draw[thick]	(0,2)--(1,2)--(1,3)--(0,3)--cycle;
				\draw[thick]	(0,4)--(1,4)--(1,5)--(0,5)--cycle;
				\draw[thick]	(0.33,1)--(0.33,2);
				\draw[thick]	(0.33,3)--(0.33,4);
				\draw[thick]	(0.66,1) to[out=60,in=270] (1.5,2.5) to[out=90,in=300] (0.66,4);	
				\draw (0.5,0.5) node {{$x$}} ; 
				\draw (0.5,2.5) node {{$x$}} ; 
				\draw (0.5,4.5) node {{$x$}} ; 
	\end{tikzpicture}}}
	\ + \
	\vcenter{\hbox{\begin{tikzpicture}[scale=0.4]
				\draw[thick]	(0,0)--(1,0)--(1,1)--(0,1)--cycle;
				\draw[thick]	(2,0)--(3,0)--(3,1)--(2,1)--cycle;
				\draw[thick]	(1,2)--(2,2)--(2,3)--(1,3)--cycle;
				\draw[thick]	(0,-2)--(1,-2)--(1,-1)--(0,-1)--cycle;	
				\draw[thick]	(0.5,1)--(1.33,2);
				\draw[thick]	(2.5,1)--(1.66,2);
				\draw[thick]	(0.5,0)--(0.5,-1);	
				\draw (0.5,-1.5) node {{$x$}} ; 
				\draw (1.5,2.5) node {{$x$}} ; 
				\draw (0.5,0.5) node {{$x$}} ; 
				\draw (2.5,0.5) node {{$x$}} ; 	
	\end{tikzpicture}}}
	\ + \ 
	\scalebox{1.1}{$\frac{1}{4}$}\ 
	\vcenter{\hbox{\begin{tikzpicture}[scale=0.4]
				\draw[thick]	(0.66,2)--(2.33,1);	
				\draw[draw=white,double=black,double distance=2*\pgflinewidth,thick]	(0.66,1)--(2.33,2);		
				\draw[thick]	(0,0)--(1,0)--(1,1)--(0,1)--cycle;
				\draw[thick]	(2,0)--(3,0)--(3,1)--(2,1)--cycle;
				\draw[thick]	(0,2)--(1,2)--(1,3)--(0,3)--cycle;
				\draw[thick]	(2,2)--(3,2)--(3,3)--(2,3)--cycle;	
				\draw[thick]	(0.33,1)--(0.33,2);
				\draw[thick]	(2.66,1)--(2.66,2);
				\draw (0.5,2.5) node {{$x$}} ; 
				\draw (2.5,2.5) node {{$x$}} ; 	
				\draw (0.5,0.5) node {{$x$}} ; 
				\draw (2.5,0.5) node {{$x$}} ; 	
	\end{tikzpicture}}}
	\ + \ \cdots\]
	\caption{The first terms of the inverse.}
	\label{Fig:Inv}
\end{figure*}

\begin{proof}We use the same method as in the previous proof: we work in the 
	free complete Lie-graph algebra on one generator $x$, where we show that 
	\begin{equation}\tag{$\circledast$}\label{eq:Inv}
		\left(\displaystyle \1+ \sum_{\gra \in \dcGra_{\Sy}} \frac{(-1)^{|\gra|}}{|\Aut(\gra)|} \, \gra(x)\right) \circledcirc (\mathbb{1}+x) = \mathbb 1~.
	\end{equation}
	
	Expanding the left-hand side of \cref{eq:Inv}, one can see that it is equal to $\1+\sum_{\gra \in \dcGra_{\Sy}} \chi_\gra\,  \gra(x)$, for some coefficients $\chi_\gra$ that we need to compute. 
	
	\medskip
	
	Let us fix a graph $\gra \in \dcGra_{\Sy}$ and let us denote by $T$ its non-empty set of top vertices, i.e. vertices with no incoming edges. 
	Notice that the left-hand side of \cref{eq:Inv} is computed by picking graphs in $x$ and adding top vertices, labeled again by $x$, along a 2-leveled graph. 
	In other words, for every subset $S\subset T$ of top vertices of $\gra$, the term $\gra(x)$ will appear as a summand of the form $(\cdots((\gra'\circ_r\gra_r)\circ_{r-1} \gra_{r-1}) \cdots )\circ_1 \gra_1$, where the $\gra_i$ are graphs in $\dcGra_{\Sy}$ obtained from $\gra$ after deleting the vertices in $S$ and where $\gra'\in \multildcGra{2}$ is the 2-leveled graph obtained by taking the top vertices to be the ones of $S$, by contracting each of the graphs $\gra_1,\dots,\gra_r$ into bottom vertices, and by merging parallel edges. 
	
	\medskip
	
	Let us denote by $\chi_S$ the coefficient associated to graph $\gra(x)$ appearing in this way. 
	The computation of this coefficient is  similar to the computation performed in the above proof of \cref{thm:dcGraGp}: one has just to consider \emph{marked graphs} non-necessarily leveled together with their canonical notion of automorphisms. This way, we get 
	\[|\chi_S|=\frac{1}{|\Aut(\gra)|}\ .\]
	This allows us to conclude the computation of the coefficient $\chi_\gra$:
	\[\chi_\gra=\sum_{S\subset T} \chi_S =\frac{1}{|\Aut(\gra)|} \sum_{S\subset T}  (-1)^{|\gra|-|S|}  =0~.\]
\end{proof}

\begin{definition}[Deformation gauge group]
	We call \emph{deformation gauge group} associated to any complete Lie-graph algebra, the complete group 
	\[
	\G\coloneq \big(\1+\F_1\g_0, \{\1+\F_k \g_0\}_{k\geqslant 1}, \circledcirc, \1\big)~.
	\]
\end{definition}

\begin{remark}
	This group generalises the ``group of formal flows'' of \cite{AgrachevGamkrelidze80} associated to a complete pre-Lie algebra and used in numerical analysis \cite{Brouder00, CHV10, LMK15}, renormalisation theory \cite{CEFM11, BHZ19}, and geometry \cite{Bandiera16}, for instance. 
	More precisely, the summand of the above formulae made up of rooted trees only gives the group associated to a complete pre-Lie algebra described in \cite{DotsenkoShadrinVallette16}. 
\end{remark}

In the case of the convolution algebra $\g_{\C,A}$, the deformation gauge group, denoted by $\G_{\C, A}$, 
actually sits inside the convolution monoid. 

\begin{proposition}\label{prop:DefGaugIna}
	The deformation gauge group of the convolution algebra $\g_{\C,A}$ is isomorphic to the following complete sub-group of the convolution monoid $\a_{\C, A}$:
	\begin{align*}
		\G_{\C,A} 
		\cong \left(\1+\Hom_{\Sy}\left(\oC, \End_A\right)_0, \left\{\1+\F_k \Hom_{\Sy}\left(\oC, \End_A\right)_0\right\}_{k\geqslant 1}, \circledcirc, \1\right) \subset \a_{\mathcal{C}, A}~. 
	\end{align*}
\end{proposition}

\begin{proof}
	The underlying space of the deformation gauge group is given by 
	\[\1+\F_1\left(\g_{\C,A}\right)_0=
	\1+\Hom_{\Sy}\left(\oC, \End_A \right)_0~.\] 
	Since the coproperad $\C$ is coaugmented, we have 
	\[\Hom_{\Sy}(\C, \End_A)\cong \Hom_{\Sy}(\I, \End_A)\oplus \Hom_{\Sy}\left(\oC, \End_A\right)~.\]
	Using the actual element $\1\in \a_{\C, A}$ defined by sending $\id\in \I$ to $\id_A$ and zero elsewhere, we get the underlying isomorphism. It remains to check that the group product given by \cref{thm:dcGraGp} on this example coincides with 
	the associative product of \cref{def:ConMono}: given $f_1, f_2 : \oC \to \End_A$, we indeed have 
	\begin{align*}
		\1 +\ \sum_{\gra \in {\ldcGra}} {\textstyle \frac{1}{|\Aut(\gra)|}}\,  \gra(f_1,f_2) \ :\  
		\C \xrightarrow{\Delta} 
		\C \boxtimes \C \xrightarrow{f_1\boxtimes f_2}  
		\End_A\boxtimes \End_A \xrightarrow{\gamma}
		\End_A
		~,
	\end{align*}
	by \cref{prop:conv algebras are dsGra}.
\end{proof}

\paragraph*{\bf A digression on the non-connected case}
	
Motivated by the study of props, see \cref{subsec:LiegraConv} and \cref{Fig:AlgStrDefTh}, one is led to try to generalise the results of this section to the non-connected setting.
Given a complete Lie-ncgraph algebra $\mathfrak g$, we equip $\1 + \F_1\mathfrak g_0$ with the analogous product 
	\[
	(\1+x)\circledcirc_\mathrm{nc} (\1+y)\coloneq \1 +\ \sum_{\gra \in {\ldcncGra_{\Sy}}} {\textstyle \frac{1}{|\Aut(\gra)|}}\,  \gra(x,y)\ , 
	\]
where  the sum runs over $2$-leveled non-necessarily connected graphs this time.
Unfortunately, this product cannot give rise to any meaningful ``deformation gauge group'' as it is not associative, see the following counterexample.
	
	\begin{cexample}
		On the free complete Lie-ncgraph algebra on three generators $x,y,z$, the  graph
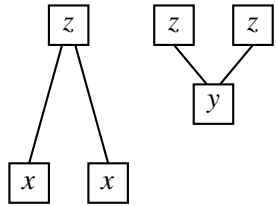
\[
\vcenter{\hbox{\begin{tikzpicture}[scale=0.35]
	\coordinate (A) at (0, 0);
	\draw[thick]	($(A)$)--($(A)+(0, 1.5)$)--($(A)+(1.5,1.5)$)--($(A)+(1.5, 0)$)--cycle;
	\draw ($(A)+(0.75,0.75)$) node {\scalebox{1}{$z$}} ; 	
	\coordinate (A) at (4, 0);
	\draw[thick]	($(A)$)--($(A)+(0, 1.5)$)--($(A)+(1.5,1.5)$)--($(A)+(1.5, 0)$)--cycle;
	\draw ($(A)+(0.75,0.75)$) node {\scalebox{1}{$z$}} ; 	
	\coordinate (A) at (7, 0);
	\draw[thick]	($(A)$)--($(A)+(0, 1.5)$)--($(A)+(1.5,1.5)$)--($(A)+(1.5, 0)$)--cycle;
	\draw ($(A)+(0.75,0.75)$) node {\scalebox{1}{$z$}} ; 	
	\coordinate (A) at (5.5, -3);
	\draw[thick]	($(A)$)--($(A)+(0, 1.5)$)--($(A)+(1.5,1.5)$)--($(A)+(1.5, 0)$)--cycle;
	\draw ($(A)+(0.75,0.75)$) node {\scalebox{1}{$y$}} ; 	
	\coordinate (A) at (1.5, -6);
	\draw[thick]	($(A)$)--($(A)+(0, 1.5)$)--($(A)+(1.5,1.5)$)--($(A)+(1.5, 0)$)--cycle;
	\draw ($(A)+(0.75,0.75)$) node {\scalebox{1}{$x$}} ; 	
	\coordinate (A) at (-1.5, -6);
	\draw[thick]	($(A)$)--($(A)+(0, 1.5)$)--($(A)+(1.5,1.5)$)--($(A)+(1.5, 0)$)--cycle;
	\draw ($(A)+(0.75,0.75)$) node {\scalebox{1}{$x$}} ; 	
	\draw[thick] (0.5, 0)--(-0.75,-4.5);		
	\draw[thick] (1, 0)--(2.25,-4.5);			
	\draw[thick] (4.75, 0)--(6,-1.5);			
	\draw[thick] (7.75, 0)--(6.5,-1.5);				
\end{tikzpicture}}}
\]
	appears with coefficient $\frac{5}{4}$ in $((\1+x)\circledcirc_\mathrm{nc} (\1+y))\circledcirc_\mathrm{nc} (\1+z)$ and 
		appears with coefficient $\frac{1}{2}$ in $(\1+x)\circledcirc_\mathrm{nc} ((\1+y)\circledcirc_\mathrm{nc} (\1+z))$.
	\end{cexample}
	
	On the other hand, there is a variation of associativity that is satisfied by $\circledcirc_\mathrm{nc}$ together with $\circledcirc$, as the following proposition shows.
	
	\begin{proposition}\label{prop:nc}
		Given a complete $\Liencgra$-algebra, we have 
		\[
		((\1+x)\circledcirc (\1+y))\circledcirc_\mathrm{nc} (\1+z)=(\1+x)\circledcirc_\mathrm{nc} ((\1+y)\circledcirc (\1+z))
		=\sum_{\gra \in \3ldcncGra_{\Sy}} {\textstyle \frac{1}{|\Aut(\gra)|}} \, \gra(x,y,z)~,
		\]
where  the sum runs over $3$-leveled non-necessarily connected graphs.
	\end{proposition}
	
	\begin{proof}
		Recall from \cref{prop:LiencgraDistLaw} that the operad $\Liencgra$ is given by a distributive law, so 
		$\Liencgra\cong \mathrm{Com}\circ \Liegra$~. 
		For $n\geqslant 1$, let $m_n$ be the generator of $\mathrm{Com}(n)$. We start by showing that 
		\begin{equation}\tag{$\circleddash$}\label{eq:nc}
			(\1+x)\circledcirc_\mathrm{nc} (\1+y) = \sum_{n\geqslant 1}\frac{1}{n!} m_n\left(\big((\1+x)\circledcirc (\1+y)\big)^{\otimes n}\right)~.
		\end{equation}
		Let us focus on the coefficient of a graph $\gra \in \dsncGra$. We write 
		\[\gra = \big(\gra^1_1 \sqcup \dots \sqcup \gra_{i_1}^1\big) \sqcup \dots \sqcup \big(\gra_1^k\sqcup \dots \sqcup \gra_{i_k}^k\big)~,\]
		
		where $\gra^j_1,\dots,\gra_{i_j}^j$ are isomorphic connected components, whose isomorphism class will by denoted by $\gra^j$. An automorphism of $\gra$ is given by automorphism of its connected components and permutations of isomorphic connected components, from which we deduce 
		\[
		| \Aut (\gra)| = \left|\Aut \big(\gra^1\big)\right|^{i_1} \dots \left|\Aut \big(\gra^k\big)\right|^{i_k} i_1! \dots i_k!~.
		\]
		The contributions of $\gra$ on the right-hand side of \cref{eq:nc} can only come when we take $m_{i_1+\dots + i_k}$ of any permutation of $i_1$ times the graph $\gra^1$, ..., and $i_k$ times the graph $\gra^k$. There are $\binom{i_1+\dots + i_k}{i_i,\dots,i_k}$ of such permutations and therefore the total contribution of the right-hand side is equal to 
		\[
		\frac{1}{n!}\frac{1}{\left|\Aut \big(\gra^1\big)\right|^{i_1} \dots \left|\Aut \big(\gra^k\big)\right|^{i_k}}\binom{i_1+\dots + i_k}{i_i,\dots,i_k}= \frac{1}{\left|\Aut \big(\gra^1\big)\right|^{i_1} \dots \left|\Aut \big(\gra^k\big)\right|^{i_k}i_i!\dots i_k! } = \frac{1}{| \Aut (\gra)|}~.
		\]
		
		The statement of the proposition now follows from the associativity of $\circledcirc$:
		
		\begin{align*}
			((\1+x)\circledcirc (\1+y))\circledcirc_\mathrm{nc} (\1+z)=  
			\sum_{n\geq 1}\frac{1}{n!} m_n\left( \big(((\1+x)\circledcirc (\1+y))\circledcirc (\1+z)\big)^{\otimes n}\right) =  \\ \notag
			=\sum_{n\geq 1}\frac{1}{n!} m_n\left( \big((\1+x)\circledcirc ((\1+y)\circledcirc (\1+z))\big)^{\otimes n}\right) = (\1+x)\circledcirc_\mathrm{nc} ((\1+y)\circledcirc (\1+z))~.
		\end{align*}

\medskip	

		As shown in the proof of Theorem \ref{thm:dcGraGp}, the product $(\1+x)\circledcirc (\1+y)\circledcirc (\1+z)$ is given by the sum over all $3$-leveled connected graphs, divided by the order of their automorphism groups. 
		Applying that result to the manipulation from the equation above, we obtain
		\[((\1+x)\circledcirc (\1+y))\circledcirc_\mathrm{nc} (\1+z)= 
		\sum_{\gra \in \3ldcncGra_{\Sy}} {\textstyle \frac{1}{|\Aut(\gra)|}} \, \gra(x,y,z)~.\]
	\end{proof}

\subsection{Graph exponential map} The following new kind of exponential map relates the above two notions of gauge groups. 

\begin{definition}[Total levelisation]
	A \emph{total levelisation} of a directed simple graph $\gra\in \dcGra_{\Sy}$  is the data of $|\gra|$-ordered levels each of which carrying only one vertex and which respects the adjacency condition. We denote the number of total levelisations of a graph by $\ell_\gra$~. 
\end{definition}

\begin{remark}
	In the case of rooted trees, this coefficient is called the \emph{Connes--Moscovici coefficient}, see \cite{Kreimer99, Brouder00}. 
\end{remark}

\begin{definition}[Graph exponential map]\label{def:exp}
	For any complete Lie-graph algebra $(\g, \F, \{\gra\}_{\gra\in \dcGra})$, the \emph{graph exponential map} is defined by 
	\[ \begin{array}{lccc}
		\exp\ : & \F_1 \g_0 & \to &\1 +\F_1 \g_0\\
		&x&\mapsto& \displaystyle\1 + \sum_{\gra \in \dcGra_{\Sy}} \frac{\ell_\gra}{|\gra|!}\,  \gra(x)\ .
	\end{array} \]
\end{definition}
\begin{figure*}[h]
	\[\exp(x)\ =\ \1 \  + \
	\vcenter{\hbox{\begin{tikzpicture}[scale=0.4]
				\draw[thick]	(0,0)--(1,0)--(1,1)--(0,1)--cycle;
				\draw (0.5,0.5) node {{$x$}} ; 
	\end{tikzpicture}}}
	\ + \
	\scalebox{1.1}{$\frac{1}{2}$}\ 
	\vcenter{\hbox{\begin{tikzpicture}[scale=0.4]
				\draw[thick]	(0,0)--(1,0)--(1,1)--(0,1)--cycle;
				\draw[thick]	(0,2)--(1,2)--(1,3)--(0,3)--cycle;
				\draw[thick]	(0.5,1)--(0.5,2);
				\draw (0.5,0.5) node {{$x$}} ; 
				\draw (0.5,2.5) node {{$x$}} ; 
	\end{tikzpicture}}}
	\ + \
	\scalebox{1.1}{$\frac{1}{6}$}\ 
	\vcenter{\hbox{\begin{tikzpicture}[scale=0.4]
				\draw[thick]	(0,0)--(1,0)--(1,1)--(0,1)--cycle;
				\draw[thick]	(2,0)--(3,0)--(3,1)--(2,1)--cycle;
				\draw[thick]	(1,2)--(2,2)--(2,3)--(1,3)--cycle;
				\draw[thick]	(0.5,1)--(1.33,2);
				\draw[thick]	(2.5,1)--(1.66,2);
				\draw (1.5,2.5) node {{$x$}} ; 
				\draw (0.5,0.5) node {{$x$}} ; 
				\draw (2.5,0.5) node {{$x$}} ; 	
	\end{tikzpicture}}}
	\ + \ 
	\scalebox{1.1}{$\frac{1}{6}$}\ 
	\vcenter{\hbox{\begin{tikzpicture}[scale=0.4]
				\draw[thick]	(0,2)--(1,2)--(1,3)--(0,3)--cycle;
				\draw[thick]	(2,2)--(3,2)--(3,3)--(2,3)--cycle;
				\draw[thick]	(1,0)--(2,0)--(2,1)--(1,1)--cycle;
				\draw[thick]	(0.5,2)--(1.33,1);
				\draw[thick]	(2.5,2)--(1.66,1);
				\draw (1.5,0.5) node {{$x$}} ; 
				\draw (0.5,2.5) node {{$x$}} ; 
				\draw (2.5,2.5) node {{$x$}} ; 	
	\end{tikzpicture}}}
	\ + \ 
	\scalebox{1.1}{$\frac{1}{6}$}\ 
	\vcenter{\hbox{\begin{tikzpicture}[scale=0.4]
				\draw[thick]	(0,0)--(1,0)--(1,1)--(0,1)--cycle;
				\draw[thick]	(0,2)--(1,2)--(1,3)--(0,3)--cycle;
				\draw[thick]	(0,4)--(1,4)--(1,5)--(0,5)--cycle;
				\draw[thick]	(0.5,1)--(0.5,2);
				\draw[thick]	(0.5,3)--(0.5,4);	
				\draw (0.5,0.5) node {{$x$}} ; 
				\draw (0.5,2.5) node {{$x$}} ; 
				\draw (0.5,4.5) node {{$x$}} ; 
	\end{tikzpicture}}}
	\ + \ 
	\scalebox{1.1}{$\frac{1}{6}$}\ 
	\vcenter{\hbox{\begin{tikzpicture}[scale=0.4]
				\draw[thick]	(0,0)--(1,0)--(1,1)--(0,1)--cycle;
				\draw[thick]	(0,2)--(1,2)--(1,3)--(0,3)--cycle;
				\draw[thick]	(0,4)--(1,4)--(1,5)--(0,5)--cycle;
				\draw[thick]	(0.33,1)--(0.33,2);
				\draw[thick]	(0.33,3)--(0.33,4);
				\draw[thick]	(0.66,1) to[out=60,in=270] (1.5,2.5) to[out=90,in=300] (0.66,4);	
				\draw (0.5,0.5) node {{$x$}} ; 
				\draw (0.5,2.5) node {{$x$}} ; 
				\draw (0.5,4.5) node {{$x$}} ; 
	\end{tikzpicture}}}
	\ + \
	\scalebox{1.1}{$\frac{1}{8}$}\ 
	\vcenter{\hbox{\begin{tikzpicture}[scale=0.4]
				\draw[thick]	(0,0)--(1,0)--(1,1)--(0,1)--cycle;
				\draw[thick]	(2,0)--(3,0)--(3,1)--(2,1)--cycle;
				\draw[thick]	(1,2)--(2,2)--(2,3)--(1,3)--cycle;
				\draw[thick]	(0,-2)--(1,-2)--(1,-1)--(0,-1)--cycle;	
				\draw[thick]	(0.5,1)--(1.33,2);
				\draw[thick]	(2.5,1)--(1.66,2);
				\draw[thick]	(0.5,0)--(0.5,-1);	
				\draw (0.5,-1.5) node {{$x$}} ; 
				\draw (1.5,2.5) node {{$x$}} ; 
				\draw (0.5,0.5) node {{$x$}} ; 
				\draw (2.5,0.5) node {{$x$}} ; 	
	\end{tikzpicture}}}
	\ + \ 
	\scalebox{1.1}{$\frac{1}{24}$}\ 
	\vcenter{\hbox{\begin{tikzpicture}[scale=0.4]
				\draw[thick]	(0.66,2)--(2.33,1);	
				\draw[draw=white,double=black,double distance=2*\pgflinewidth,thick]	(0.66,1)--(2.33,2);		
				\draw[thick]	(0,0)--(1,0)--(1,1)--(0,1)--cycle;
				\draw[thick]	(2,0)--(3,0)--(3,1)--(2,1)--cycle;
				\draw[thick]	(0,2)--(1,2)--(1,3)--(0,3)--cycle;
				\draw[thick]	(2,2)--(3,2)--(3,3)--(2,3)--cycle;	
				\draw[thick]	(0.33,1)--(0.33,2);
				\draw[thick]	(2.66,1)--(2.66,2);
				\draw (0.5,2.5) node {{$x$}} ; 
				\draw (2.5,2.5) node {{$x$}} ; 	
				\draw (0.5,0.5) node {{$x$}} ; 
				\draw (2.5,0.5) node {{$x$}} ; 	
	\end{tikzpicture}}}
	\ + \ \cdots\]
	\caption{The first terms of the graph exponential map.}
	\label{Fig:GraphExp}
\end{figure*}
So the first terms of the graph exponential are equal to $\1+x+x\star x+\cdots$\ . Notice that the other terms \emph{cannot} be obtained by the iteration of the Lie-admissible product and thus require the use of the Lie-graph algebra structure. This key observation is the central point of the present paper. 
A source of inspiration for the introduction of the  present exponential map is the work 
\cite{Chapoton02ArXiv} of F. Chapoton.

\begin{lemma}\label{lem:ExpLnBij}
	For any complete Lie-graph algebra, the graph exponential map is bijective, 
	\[\exp\ :  \F_1 \g_0 \xrightarrow{\cong}\1 +\F_1 \g_0\ . \]
\end{lemma}

\begin{proof}
	The arguments are classical and similar to the proof that a formal power series with unit constant term is invertible. For the sake of completeness, we adapt it here. 
	
	\medskip
	
	First, we prove that the graph exponential map is injective. Let $x$ and $y$ in $\F_1\g_0$ such that $\exp(x)=\exp(y)$~. 
	We will  prove, by induction on $k\geqslant 2$, that $x\equiv y \mod\F_k\g_0$, that is their projections onto 
	$\F_1\g_0/\F_k\g_0$ agree; this will imply that $x=y$ since $\g_0$ is complete. The image of the graph exponential map is equal to $\exp(x)=\1+\sum_{k\geqslant 1} x^{(k)}$, where $x^{(k)}\in \F_k\g_0$ is obtained by the action on $x$ of a finite sum of directed simple graphs with $k$ vertices. 
	For instance, $x^{(1)}=x$~. 
	The projection of 
	$\exp(x)-\1=x+\sum_{k\geqslant 2} x^{(k)}= \exp(y)-\1=y+\sum_{k\geqslant 2} y^{(k)}$ on $\F_1\g_0/\F_2\g_0$ shows that there exists $f^{(2)}\in \F_2 \g_0$ such that $x=y+f^{(2)}$~. Suppose now that there exists $f^{(k)}\in \F_k \g_0$ such that $x=y+f^{(k)}$~. The equation $\exp(x)=\exp(y)$ implies then 
	$x+y^{(2)}+\cdots+y^{(k)}\equiv y+y^{(2)}+\cdots+y^{(k)} \mod 
	\F_{k+1}\g_0$, which concludes the induction. 
	
	\medskip
	
	We prove now that the graph exponential map is surjective. Let $x\in \F_1\g_0$. We claim that there exists a convergent series $y=x+f^{(2)}+f^{(3)}+\cdots$, with $f^{(k)}\in \F_k\g_0$, for any $k\geqslant 2$, satisfying 
	$\exp(y)\equiv\exp\big(x+f^{(2)}+\cdots+f^{(k-1)}\big)\equiv \1+x \mod \F_k\g_0$. Such a series can be iteratively found as follows. For $k=2$, we have $\exp(x)\equiv\1+x \mod \F_2\g_0$. Let us now suppose that there exists $y=x+f^{(2)}+\cdots+f^{(k-1)}$, with $f^{(i)}\in \F_i\g_0$,  such that 
	$\exp(y)\equiv\exp\big(x+f^{(2)}+\cdots+f^{(j-1)}\big)\equiv \1+x \mod \F_j\g_0$, for any $2\leqslant j \leqslant k$.
	For any $f^{(k)}\in \F_k\g_0$, we have 
	$\exp\big(x+f^{(2)}+\cdots+f^{(k)}\big)\equiv \1+x+f^{(k)}+z^{(k)} \mod \F_{k+1}\g_0$, where $z^{(k)}\in \F_k\g_0$ is given by 
	the action of a finite sum of directed simple graphs with at most $k$ vertices  on $x, f^{(2)}, \ldots, f^{(k-1)}$~. It is thus enough to set $f^{(k)}\coloneq -z^{(k)}$ for the equation to hold modulo $\F_{k+1}\g_0$~. 
	
\end{proof}

\begin{definition}[Graph logarithm map]
	For any complete Lie-graph algebra, 
	the inverse of the graph exponential map is called the \emph{graph logarithm map}:
	\[\ln\ :  \1+\F_1 \g_0 \xrightarrow{\cong} \F_1 \g_0\ . \]
\end{definition}

\begin{figure*}[h]
	\[\ln(\1+x)\ =
	\vcenter{\hbox{\begin{tikzpicture}[scale=0.4]
				\draw[thick]	(0,0)--(1,0)--(1,1)--(0,1)--cycle;
				\draw (0.5,0.5) node {{$x$}} ; 
	\end{tikzpicture}}}
	\ - \
	\scalebox{1.1}{$\frac{1}{2}$}\ 
	\vcenter{\hbox{\begin{tikzpicture}[scale=0.4]
				\draw[thick]	(0,0)--(1,0)--(1,1)--(0,1)--cycle;
				\draw[thick]	(0,2)--(1,2)--(1,3)--(0,3)--cycle;
				\draw[thick]	(0.5,1)--(0.5,2);
				\draw (0.5,0.5) node {{$x$}} ; 
				\draw (0.5,2.5) node {{$x$}} ; 
	\end{tikzpicture}}}
	\ + \
	\scalebox{1.1}{$\frac{1}{12}$}\ 
	\vcenter{\hbox{\begin{tikzpicture}[scale=0.4]
				\draw[thick]	(0,0)--(1,0)--(1,1)--(0,1)--cycle;
				\draw[thick]	(2,0)--(3,0)--(3,1)--(2,1)--cycle;
				\draw[thick]	(1,2)--(2,2)--(2,3)--(1,3)--cycle;
				\draw[thick]	(0.5,1)--(1.33,2);
				\draw[thick]	(2.5,1)--(1.66,2);
				\draw (1.5,2.5) node {{$x$}} ; 
				\draw (0.5,0.5) node {{$x$}} ; 
				\draw (2.5,0.5) node {{$x$}} ; 	
	\end{tikzpicture}}}
	\ + \ 
	\scalebox{1.1}{$\frac{1}{12}$}\ 
	\vcenter{\hbox{\begin{tikzpicture}[scale=0.4]
				\draw[thick]	(0,2)--(1,2)--(1,3)--(0,3)--cycle;
				\draw[thick]	(2,2)--(3,2)--(3,3)--(2,3)--cycle;
				\draw[thick]	(1,0)--(2,0)--(2,1)--(1,1)--cycle;
				\draw[thick]	(0.5,2)--(1.33,1);
				\draw[thick]	(2.5,2)--(1.66,1);
				\draw (1.5,0.5) node {{$x$}} ; 
				\draw (0.5,2.5) node {{$x$}} ; 
				\draw (2.5,2.5) node {{$x$}} ; 	
	\end{tikzpicture}}}
	\ + \ 
	\scalebox{1.1}{$\frac{1}{3}$}\ 
	\vcenter{\hbox{\begin{tikzpicture}[scale=0.4]
				\draw[thick]	(0,0)--(1,0)--(1,1)--(0,1)--cycle;
				\draw[thick]	(0,2)--(1,2)--(1,3)--(0,3)--cycle;
				\draw[thick]	(0,4)--(1,4)--(1,5)--(0,5)--cycle;
				\draw[thick]	(0.5,1)--(0.5,2);
				\draw[thick]	(0.5,3)--(0.5,4);	
				\draw (0.5,0.5) node {{$x$}} ; 
				\draw (0.5,2.5) node {{$x$}} ; 
				\draw (0.5,4.5) node {{$x$}} ; 
	\end{tikzpicture}}}
	\ + \ 
	\scalebox{1.1}{$\frac{1}{3}$}\ 
	\vcenter{\hbox{\begin{tikzpicture}[scale=0.4]
				\draw[thick]	(0,0)--(1,0)--(1,1)--(0,1)--cycle;
				\draw[thick]	(0,2)--(1,2)--(1,3)--(0,3)--cycle;
				\draw[thick]	(0,4)--(1,4)--(1,5)--(0,5)--cycle;
				\draw[thick]	(0.33,1)--(0.33,2);
				\draw[thick]	(0.33,3)--(0.33,4);
				\draw[thick]	(0.66,1) to[out=60,in=270] (1.5,2.5) to[out=90,in=300] (0.66,4);	
				\draw (0.5,0.5) node {{$x$}} ; 
				\draw (0.5,2.5) node {{$x$}} ; 
				\draw (0.5,4.5) node {{$x$}} ; 
	\end{tikzpicture}}}
	\ + \ \cdots\]
	\caption{The first terms of the graph logarithm map.}
	\label{Fig:GraphLog}
\end{figure*}

\begin{remark}
	The present graph exponential map generalises the pre-Lie exponential map and 
	the present graph logarithm map generalises the Magnus expansion map, see \cite{AgrachevGamkrelidze80, Manchon11} for instance. 
	Like in the pre-Lie case \cite{EFM09}, it would be interesting to unravel a closed formula for the graph logarithm map.
\end{remark}

In Lie theory, the exponential map is a key notion, which recovers the local group structure from the infinitesimal Lie structure on on the tangent space; it is characterized by a certain differential equation. The present graph exponential map plays the same role for Lie-graph algebras; we describe its characterizing differential equation below. This Lie theoretical property satisfied by the  exponential graph map will play a major role in the proof of the forthcoming \cref{thm:ExpIsoMorph}. 

\begin{definition}[$\rhd$-product and $\lhd$-product]
		Given  a complete Lie-graph algebra $\g$ and two elements $x,y\in \F_1\mathfrak g_0$, the elements $(\1 + x)\rhd y$ and $x\lhd (\1 +y)$ of $\F_1\mathfrak g_0$ are defined respectively by
		
		\[
		(\1 + x)\rhd y \coloneq \vcenter{\hbox{\begin{tikzpicture}[scale=0.4]
					\draw[thick]	(0,0)--(1,0)--(1,1)--(0,1)--cycle;
					\draw (0.5,0.5) node {{$y$}} ; 
		\end{tikzpicture}}}
		\ + \
		\sum_{r\geqslant 1}\scalebox{1.1}{$\frac{1}{r!}$}\ 
		\vcenter{\hbox{\begin{tikzpicture}[scale=0.4]
					\draw[thick]	(0,0)--(1,0)--(1,1)--(0,1)--cycle;
					\draw[thick]	(1.5,0)--(2.5,0)--(2.5,1)--(1.5,1)--cycle;
					\draw[thick]	(4.5,0)--(5.5,0)--(5.5,1)--(4.5,1)--cycle;	
					\draw[thick]	(2,2)--(3,2)--(3,3)--(2,3)--cycle;
					\draw[thick]	(0.5,1)--(2.33,2);
					\draw[thick]	(2,1)--(2.5,2);
					\draw[thick]	(5,1)--(2.66,2);	
					\draw (2.5,2.5) node {{$y$}} ; 
					\draw (0.5,0.5) node {{$x$}} ; 
					\draw (2,0.5) node {{$x$}} ; 	
					\draw (5,0.5) node {{$x$}} ; 		
					\draw (3.58,0.4) node {{$\cdots$}} ; 		
		\end{tikzpicture}}} \quad \text{and}\quad 
		x\lhd (\1 + y) \coloneq \vcenter{\hbox{\begin{tikzpicture}[scale=0.4]
					\draw[thick]	(0,0)--(1,0)--(1,1)--(0,1)--cycle;
					\draw (0.5,0.5) node {{$x$}} ; 
		\end{tikzpicture}}}
		\ + \
		\sum_{r\geqslant 1}\scalebox{1.1}{$\frac{1}{r!}$}\ 
		\vcenter{\hbox{\begin{tikzpicture}[scale=0.4]	
	\coordinate (A) at (0, 2);
	\draw[thick]	($(A)$)--($(A)+(0, 1)$)--($(A)+(1,1)$)--($(A)+(1, 0)$)--cycle;	
	\draw ($(A)+(0.5,0.5)$) node {$y$}; 
	\coordinate (A) at (1.5, 2);
	\draw[thick]	($(A)$)--($(A)+(0, 1)$)--($(A)+(1,1)$)--($(A)+(1, 0)$)--cycle;
	\draw ($(A)+(0.5,0.5)$) node {$y$}; 
	\coordinate (A) at (4.5, 2);
	\draw[thick]	($(A)$)--($(A)+(0, 1)$)--($(A)+(1,1)$)--($(A)+(1, 0)$)--cycle;
	\draw ($(A)+(0.5,0.5)$) node {$y$}; 	
	\coordinate (A) at (2,0);
	\draw ($(A)+(0.5,0.5)$) node {$x$}; 	
	\draw[thick]	($(A)$)--($(A)+(0, 1)$)--($(A)+(1,1)$)--($(A)+(1, 0)$)--cycle;
	\draw ($(A)+(0.5,0.5)$) node {$x$}; 
	\draw[thick]	(0.5,2)--(2.33,1);
	\draw[thick]	(2,2)--(2.5,1);
	\draw[thick]	(5,2)--(2.66,1);	
	\draw (3.58,2.4) node {{$\cdots$}} ; 							
		\end{tikzpicture}}}~.		
		\]
\end{definition}

\begin{remark}
The product $(\1 + x)\rhd y$ is made up of the summands of $(\1 + x)\cc (\1 +y)$ with only one vertex labelled by $y$, that is its linear part in $y$. 
The product $x\lhd (\1 +y)$ is made up of the summands of $(\1 + x)\cc (\1 +y)$ with only one vertex labelled by $x$, that is its linear part in $x$. 
This explains the choice of notation, which also coincides with the convention of \cite[Section~3.3]{HLV19}. 
\end{remark}

	\begin{lemma}\label{lem:EDexp}
		For any complete Lie-graph algebra $\mathfrak g$, any $z \in \F_1 \mathfrak g_0$ and any $t \in \k$, we have
		\[
		\frac{d}{dt}\exp(tz)=\exp(tz)\rhd z = z \lhd \exp(tz)~ .
		\]		
	\end{lemma}
	
	\begin{proof}
		
		We will only show the first equality, as the second one follows by symmetry. Given a graph $\gra$ with $|\gra|$ vertices, we have $\gra(tz) = t^{|\gra|}\gra(z)$. It follows that
		\[\frac{d}{dt}\exp(tz)=\displaystyle \sum_{\gra \in {\dcGra}_{\Sy}} \frac{\ell_\gra}{(|\gra|-1)!} t^{|\gra|-1}\, \gra(z)\ .\]
		We will give a combinatorial proof that
		\begin{equation}\label{Eq:ddtexp}\tag{$\diamond$}
			\exp(tz)\rhd z=\displaystyle \sum_{\gra \in {\dcGra}_{\Sy}} \frac{\ell_\gra}{(|\gra|-1)!} t^{|\gra|-1}\, \gra(z)\ .
		\end{equation}
		Indeed, the left-hand side is equal to 
		$\sum_{\gra \in {\dcGra}_{\Sy}} \allowbreak\rho_\gra t^{|\gra|-1}\, \gra(z)$~,
		for some coefficients $\rho_\gra$ that we will now compute. 
		
		\medskip 
		
		Let $\gra\in {\dcGra}_{\Sy} $ be a directed simple graph 
		and consider one of its total levelisation, say $\gra_{\mathrm{lev}}$, and  let us see how it can appear on the left-hand side of \cref{Eq:ddtexp}, using arguments similar to the proof of \cref{thm:dcGraGp}. 
		The totally leveled graph $\gra_{\mathrm{lev}}$ admits a top vertex; cutting out its bottom edges attached to it produces $r$ totally leveled graphs $\gra_1, \ldots \gra_r$, with multiplicity, after forgetting the trivial levels without any vertex. These graphs are actually more determined when one considers at which vertices the edges coming from the top vertex is attached to. As above, we call \emph{marked} such totally leveled graphs. In the present case, let us say that the leveled graph $\gra_{\mathrm{lev}}$ induces $l$ different marked leveled graphs $\dot{\gamma}_m$, for $1\leqslant m\leqslant l$~. 
		Let us denote by $j_m$ the number of occurrences of the 
		marked leved graph $\dot{\gamma}_m$ among $\gra_1, \ldots \gra_r$, so that $j_1+\dots+j_l = r$~.
		One can regroup them as follows. Suppose that there are $k$ types of totally leveled graphs among $\gra_1, \ldots \gra_r$~; for any $1\leqslant h\leqslant k$, we denote by $t_h$ the number of marked leveled graphs appearing  among the $\gra_1,\dots ,\gra_r$ when one forget about the extra edges, so that $t_1+\cdots+t_k=r$ is a coarser partition of $r$~. 
		
		\medskip
		
		In the other way around, given such a data, let us see how many times one can get $\gra_{\mathrm{lev}}$ from contributions on the left-hand side. There are 5 different coefficients to take into account.
		
		\begin{enumerate}[i)]
			\item The overall coefficient $\frac{1}{r!}$~. 
			
			\item The coefficient of the totally leveled graphs $\gra_i$ which is equal to 
			$\prod_{i=1}^r \frac{1}{|\gra_i|!}$~.
			
			\item A multinomial coefficient corresponding to the places where the leveled graphs $\gra_1, \ldots \gra_r$ can be put: 
			$ \frac{r!}{t_1!\ldots t_k!}$~.

			\item A multinomial coefficient corresponding to the places where the marked leveled graphs can be put among the leveled graphs $\gra_1, \ldots, \gra_r$~. 
			For each $1\leqslant h \leqslant k$, we denote by $j_{m^h_1}, \ldots, j_{m^h_{n(h)}}$ the respective numbers of marked leveled graphs that give the same leveled graph after forgetting the extra incoming edge, so that 
			$j_{m^h_1}+ \cdots+ j_{m^h_{n(h)}}=t_h$ and $\left\{m^1_1, \ldots, m^k_{n(k)}\right\}=\{1, \ldots, l\}$~. This produces the coefficient 		
			\[
			\prod_{h=1}^k \frac{t_h!}{j_{m^h_1}!  \ldots j_{m^h_{n(h)}}!} = \frac{t_1!\dots t_k!}{j_1!\dots j_l!}~.
			\]		
			
			\item A shuffling coefficient corresponding to the various ways of obtaining the total levelization on $\gra_{\mathrm{lev}}$ from the ones of $\dot{\gamma}_1, \ldots, \dot{\gamma}_l$: 
			\[\frac{|\gra_1|!\ldots |\gra_r|!}{(|\gra|-1)!}j_1!\ldots, j_l!~.\]
		\end{enumerate}
		
		Multiplying all the contributions, we get $\frac{1}{(|\gra|-1)!}$ and thus 
		\[\rho_\gra=\frac{\ell_\gra}{(|\gra|-1)!}~.\]
	\end{proof}

\begin{theorem}\label{thm:ExpIsoMorph}
	For any complete Lie-graph algebra, the graph exponential and logarithm maps are  filtered isomorphisms, thus homeomorphisms, 
	between the gauge group and the deformation gauge group:
	\[\exp \ : \Gamma=\big( \F_1 \g_0,  \BCH, 0\big)
	\stackrel{\cong}{\leftrightarrows}
	\left( \1 + \F_1 \g_0, \circledcirc, \1\right)=\G\ : \ \ln
	\ .\]
\end{theorem}

\begin{proof}
	The graph exponential map and the graph logarithm map preserve the respective filtrations by construction. \cref{lem:ExpLnBij} showed that they are bijective. It  remains to show that they preserve the respective group structures; to this extent, we will show that 
	\begin{equation}\label{Eq:ExpBCH}\tag{$\circ$}
		\exp(\BCH(x,y))=\exp(x)\cc\exp(y)\ .
	\end{equation}
	Since we are working over a ground field of characteristic $0$, it is enough to prove this formula for the free complete Lie-graph algebra $\widehat{\Liegra}_{\mathbb{Q}}(x,y)$ on two elements $x,y$ over the field $\mathbb{Q}$ of rational numbers. 
	This free complete Lie-graph algebra admits for underlying space the product 
	$\prod_{\gra\in {\dcGra}} \mathbb{Q}\gra(x,y)$, where $\g(x,y)$ stands for all the possible ways to label the vertices of the simple graph $\g$ with $x$ and $y$. 
	The filtration $\F_k\widehat{\Liegra}_{\QQ}(x,y)$ is made of up products of directed simple graphs with at least $k$ vertices labelled by $x$ and $y$. Since this Lie-graph algebra is complete, it is enough to prove \cref{Eq:ExpBCH} modulo $\F_k$, for any $k \geqslant 1$~. Let us fix $k$ and work with the $\mathbb{R}$-extension of the $\mathbb{Q}$-vector space 
	$\widehat{\Liegra}_{\QQ}(x,y)/\F_k\widehat{\Liegra}_{\QQ}(x,y)$. This way, the data 
	\[\left(\1+\widehat{\Liegra}_{\mathbb{R}}(x,y)/\F_k\widehat{\Liegra}_{\mathbb{R}}(x,y), \cc, \1\right)\]
	becomes a finite dimensional Lie group so we can apply to it the classical methods of Lie groups as follows. 
	
	\medskip
	
	Notice now that its tangent space at the unit is $\widehat{\Liegra}_{\mathbb{R}}(x,y)/\F_k\widehat{\Liegra}_{\mathbb{R}}(x,y)$~. It is straightforward to see that the Lie bracket induced by $\cc$ is equal to  the skew-symmetrisation of the Lie-admissible product $\star$ given by the action of the 2-vertices graph. 
	In this real truncated version of the deformation gauge group, the differential of the left multiplication 
	$\Lambda_{\1+v} :  \1+w\mapsto (\1+v)\circledcirc (\1+w)$ is equal to 
	\[\mathrm{D}_{\1}\Lambda_{\1+v}(z)=
	\vcenter{\hbox{\begin{tikzpicture}[scale=0.4]
				\draw[thick]	(0,0)--(1,0)--(1,1)--(0,1)--cycle;
				\draw (0.5,0.5) node {{$z$}} ; 
	\end{tikzpicture}}}
	\ + \
	\sum_{r\geqslant 1}\scalebox{1.1}{$\frac{1}{r!}$}\ 
	\vcenter{\hbox{\begin{tikzpicture}[scale=0.4]
				\draw[thick]	(0,0)--(1,0)--(1,1)--(0,1)--cycle;
				\draw[thick]	(1.5,0)--(2.5,0)--(2.5,1)--(1.5,1)--cycle;
				\draw[thick]	(4.5,0)--(5.5,0)--(5.5,1)--(4.5,1)--cycle;	
				\draw[thick]	(2,2)--(3,2)--(3,3)--(2,3)--cycle;
				\draw[thick]	(0.5,1)--(2.33,2);
				\draw[thick]	(2,1)--(2.5,2);
				\draw[thick]	(5,1)--(2.66,2);	
				\draw (2.5,2.5) node {{$z$}} ; 
				\draw (0.5,0.5) node {{$v$}} ; 
				\draw (2,0.5) node {{$v$}} ; 	
				\draw (5,0.5) node {{$v$}} ; 		
				\draw (3.58,0.5) node {{$\cdots$}} ; 		
	\end{tikzpicture}}}
	=(\1+v)\rhd z~,
	\]
	which is the part of $(\1+v)\circledcirc (\1+z)$ made up of only one $z$ labelling the top level of the underlying 2-levels graphs. 
	
	\medskip

The (classical) exponential map of this Lie group  is the map integrating the flow induced by the vector field $\1+v\mapsto (\1+v)\rhd z$ and starting at the unit $\1$~. 
\cref{lem:EDexp} shows that the graph exponential map $\exp(tz)$ satisfies this differential equation, that is 
		$\frac{d}{dt}\varphi(t)=\varphi(t)\rhd z$~, with initial condition $\varphi(0)=\1$~.
	We conclude with the fact that the exponential map of a Lie group is (at least locally) a group morphism from the Hausdorff group defined by the BCH-formula to the Lie group. 
\end{proof}

\begin{remark}
	This proof shows that the graph exponential map is \emph{the} expected exponential map from Lie theory. 
\end{remark}

\subsection{Action on Maurer--Cartan elements}
In any dg Lie algebra $(\g, \d, [\, ,])$, one can use the \emph{differential trick} \cite[Section~1.2]{DotsenkoShadrinVallette22} to reduce  
involved computations to similar ones but without the differential map appearing. 
More precisely, this amounts to working in the one-dimensional abelian extension Lie algebra 
\[\g^+\coloneq(\k\delta\oplus \g, [\, ,])\] 
defined by $[\delta, x]\coloneq \d x$ and by $[\delta, \delta]\coloneq 0$~. 
Since the differential of $\g$ has been made internal in $\g^+$, the various formulas of $\g$ have a simpler form in $\g^+$. 
For instance, Maurer--Cartan elements of $\g$ are viewed as follows in $\g^+$: 
\[
\bar\alpha \in \g \mapsto  \alpha \coloneq \delta+\bar\alpha
\in  \g^+~. \]
In this way, the Maurer--Cartan equation in $\g$ 
is equivalent to the square-zero equation in $\g^+$:
\[[\alpha, \alpha]=0\ .\]
In the extended complete Lie algebra $\g^+$, the gauge group $\Gamma$ acts on the set of Maurer--Cartan elements $\MC(\g)$ via the shorter formula 
\[\lambda.\alpha=e^{\ad_\lambda}(\alpha)\ .\]
Going back to the original complete dg Lie $\g$, one recovers this way the classical but more intricate formula for the gauge group action:
\[
\lambda.\bar\alpha \coloneq e^{\ad_\lambda}(\bar\alpha)+\frac{\id-e^{\ad_\lambda}}{\ad_{\lambda}}(\d\lambda)~. 
\]
\medskip

In order to perform the same kind of differential trick in complete dg Lie-graph algebras, we cannot simply consider a one-dimensional abelian extension; the relevant extension is conceptually defined as follows. 

{We consider the category of \emph{internal dg Lie-graph algebras}, which is made up of dg Lie-graph algebras $(A, \d, \{\gra\}_{\gra \in \dcGra})$ equipped with a degree $-1$ element $\delta$ satisfying, for any $x\in A$:
	\[\d x=\ad_\delta(x)=\delta \star x - (-1)^{|x|}x \star \delta \quad \text{and} \quad \d \delta=0~,\] 
	where $\star$ denote the binary operation in $\g$ coming from the directed simple graph with two vertices.
	Notice that the first relation applied to $x=\delta$ gives $\d \delta = 2 \delta \star \delta$, so the second relation is equivalent to $\delta \star \delta=0$~. 
	The morphisms of internal dg Lie-graph algebras are the one of dg Lie-graph algebras preserving the respective extra degree $-1$ elements. Over the category of graded modules, the category of dg Lie-graph algebras is encoded by the operad $\mathrm{d}\Liegra$ with the differential included in it: 
	\[\mathrm{d}\Liegra\coloneq \frac{\Liegra\vee \mathcal{T}(\d)}{\left(\d \ \text{derivation}, \, \d^2 \right)}~,\]
	where $\d$ is a derivation for any operation in $\gra\in \Liegra(n)$, that is 
	$\d \circ_1 \gra - \sum_{i=1}^n \gra \circ_i \d$~. 
	Similarly, the category of internal dg Lie-graph algebras is encoded by the following operad $\mathrm{d}\Liegra^+$ in graded modules which is defined with an extra degree $-1$ and arity $0$ element:
	\[\mathrm{d}\Liegra^+\coloneq \frac{\Liegra\vee \mathcal{T}(\d, \delta)}{\left(\d \ \text{derivation},\,  \d^2~,\,  
		\d-
		\begin{tikzpicture}[scale=0.4, baseline=(n.base)]
			\node (n) at (0.5,1.3) {};
			\draw[thick]	(0,0)--(1,0)--(1,1)--(0,1)--cycle;
			\draw[thick]	(0,2)--(1,2)--(1,3)--(0,3)--cycle;
			\draw[thick]	(0.5,1)--(0.5,2);
			\draw (0.5,0.5) node {{$\delta$}} ; 
		\end{tikzpicture}
		+
		\begin{tikzpicture}[scale=0.4, baseline=(n.base)]
			\node (n) at (0.5,1.3) {};
			\draw[thick]	(0,0)--(1,0)--(1,1)--(0,1)--cycle;
			\draw[thick]	(0,2)--(1,2)--(1,3)--(0,3)--cycle;
			\draw[thick]	(0.5,1)--(0.5,2);
			\draw (0.5,2.5) node {{$\delta$}} ; 
		\end{tikzpicture}\ , \d(\delta)
		\right)}~.\]
	The  morphism of graded operads $\rho \colon \mathrm{d}\Liegra \to \mathrm{d}\Liegra^+$, preserving $\Liegra$ and $\d$, induces a pair of adjoint functors by \cite[Section~5.2.12]{LodayVallette12}: the right adjoint functor $\rho^*$  forgets the choice of  internal element and the left adjoint $\rho_!$ produces the differential extension for dg Lie-graph algebras.}

{
	\begin{definition}[Differential extension]
		The \emph{differential extension} of a complete dg Lie-graph algebra $\g$ is the complete internal dg Lie-graph algebra 
		$\g^+\coloneq \rho_!(\g)$~.
	\end{definition}
	
	By a slight abuse of notation, we will simply denote by $\g^+$, and not $\rho^*(\g^+)$, the complete dg Lie-graph algebra structure on the differential extension. 
	
	\begin{proposition}\label{prop:inclusion}
		The unit of the $(\rho_!, \rho^*)$-adjunction is an embedding of complete dg Lie-graph algebras: 
		$\g \rightarrowtail \g^+$~.
	\end{proposition}
	
	\begin{proof}
		We first claim that the morphism $\rho \colon \mathrm{d}\Liegra \rightarrowtail \mathrm{d}\Liegra^+$ of graded operads is injective. 
		We consider the increasing filtration on the operad $\mathrm{d}\Liegra^+$ given by the number of elements $\delta$ involved. 
		The associated graded operad  admits the following presentation 
		\[\mathrm{gr}\, \mathrm{d}\Liegra^+\cong \frac{\Liegra\vee \mathcal{T}(\d, \delta)}{\left(\d \ \text{derivation},\,  \d^2~,\,  
			\begin{tikzpicture}[scale=0.4, baseline=(n.base)]
				\node (n) at (0.5,1.3) {};
				\draw[thick]	(0,0)--(1,0)--(1,1)--(0,1)--cycle;
				\draw[thick]	(0,2)--(1,2)--(1,3)--(0,3)--cycle;
				\draw[thick]	(0.5,1)--(0.5,2);
				\draw (0.5,0.5) node {{$\delta$}} ; 
			\end{tikzpicture}
			-
			\begin{tikzpicture}[scale=0.4, baseline=(n.base)]
				\node (n) at (0.5,1.3) {};
				\draw[thick]	(0,0)--(1,0)--(1,1)--(0,1)--cycle;
				\draw[thick]	(0,2)--(1,2)--(1,3)--(0,3)--cycle;
				\draw[thick]	(0.5,1)--(0.5,2);
				\draw (0.5,2.5) node {{$\delta$}} ; 
			\end{tikzpicture}\ , \d(\delta)
			\right)}~.\]
		This shows that its part of weight $0$ in $\delta$ is isomorphic to 
		$\mathrm{gr}_0\, \mathrm{d}\Liegra^+\cong \mathrm{d}\Liegra$~.
		Since the underlying two $\Sy$-modules of the operads $\mathrm{gr}\, \mathrm{d}\Liegra^+$ and $\mathrm{d}\Liegra^+$ are isomorphic, this proves the claim. 
		
		\medskip 
		
		We recall then that the functor $\rho_!$ is given by the relative composite product 
		\[\rho_!(\g)\cong \mathrm{d}\Liegra^+ \,\widehat{\circ}_{\mathrm{d}\Liegra}\, \g~. \]
		Since we are working over a field $\k$ of characteristic $0$, every symmetry group algebra $\k[\Sy_n]$ is semisimple, so the underlying graded $\Sy$-module $\mathrm{d}\Liegra^+$ is isomorphic to $\mathrm{d}\Liegra \oplus \mathrm{M}$~. This implies that 
		$\g^+\cong \g \oplus V$~,
		which concludes the proof. 
	\end{proof}
			
	The above arguments apply as well to the Lie and pre-Lie cases. 
	The Lie case is more simple than the Lie-graph case, since we can make this relative composite product explicit in this context:  
	\[
	\mathrm{d}\Lie^+\, \widehat{\circ}_{\mathrm{d}\Lie}\, \g \cong \frac{\g\vee \widehat{\mathrm{Lie}}(\delta)}{\big([\delta, \delta], [\delta, x]=\d x \big)}
	\cong \k\delta \oplus \g\ ,
	\]
	which recovers the one-dimensional abelian extension mentioned above. (One can prove directly that this new complete internal dg Lie algebra satisfies the universal property of the left adjoint functor.) 
}

\begin{remark}
In contrast to the Lie case, it is difficult to give an expression for 
		$\g^+$ more explicit than  $\mathrm{d}\Liegra^+ \,\widehat{\circ}_{\mathrm{d}\Liegra}\, \g$~. In what comes next, we do not actually  need any explicit form nor basis for the differential extension, only the fact that the original dg Lie-graph algebra embeds in it. 
\end{remark}

\medskip

The following construction will allow us to make explicit the action of the deformation gauge group on Maurer--Cartan elements.

\begin{definition}[Bowtie graph]
	A \emph{bowtie graph} is a 3-leveled directed simple graphs with only one vertex on the middle level and possibly empty top and bottom levels. We denote the set of bowtie graphs by $\btGra$~.
\end{definition}

\begin{definition}[Bowtie element]
	For any elements $x,y \in\F_1\g$ and $\alpha\in\g$ in a complete Lie-graph algebra, we consider the associated \emph{bowtie element}
	of $\g$~:
	\[\pap{(\1+x)}{\alpha}{(\1+y)}\coloneq \sum_{\gra \in \btGra} {\textstyle \frac{1}{|\Aut(\gra)|}} \, \gra(x,\alpha,y)\ ,\]
	where the notation $\gra(x,\alpha,y)$ stands for the action in $\g$ of  the bowtie graph $\gra$ with 
	$x$ labelling the bottom vertices, $\alpha$ the middle vertex, and $y$ the top vertices, see \cref{Fig:Pap}.
	\begin{figure*}[h]
		\begin{align*}
			\pap{(\1+x)}{\alpha}{(\1+y)}\ = \ &
			\vcenter{\hbox{\begin{tikzpicture}[scale=0.4]
						\draw[thick]	(0,0)--(1,0)--(1,1)--(0,1)--cycle;
						\draw (0.5,0.5) node {{$\alpha$}} ; 
			\end{tikzpicture}}}
			\ + \
			\vcenter{\hbox{\begin{tikzpicture}[scale=0.4]
						\draw[thick]	(0,0)--(1,0)--(1,1)--(0,1)--cycle;
						\draw[thick]	(0,2)--(1,2)--(1,3)--(0,3)--cycle;
						\draw[thick]	(0.5,1)--(0.5,2);
						\draw (0.5,0.5) node {{$x$}} ; 
						\draw (0.5,2.5) node {{$\alpha$}} ; 
			\end{tikzpicture}}}
			\ + \
			\vcenter{\hbox{\begin{tikzpicture}[scale=0.4]
						\draw[thick]	(0,0)--(1,0)--(1,1)--(0,1)--cycle;
						\draw[thick]	(0,2)--(1,2)--(1,3)--(0,3)--cycle;
						\draw[thick]	(0.5,1)--(0.5,2);
						\draw (0.5,0.5) node {{$\alpha$}} ; 
						\draw (0.5,2.5) node {{$y$}} ; 
			\end{tikzpicture}}}
			\ + \
			\scalebox{1.1}{$\frac{1}{2}$}\ 
			\vcenter{\hbox{\begin{tikzpicture}[scale=0.4]
						\draw[thick]	(0,0)--(1,0)--(1,1)--(0,1)--cycle;
						\draw[thick]	(2,0)--(3,0)--(3,1)--(2,1)--cycle;
						\draw[thick]	(1,2)--(2,2)--(2,3)--(1,3)--cycle;
						\draw[thick]	(0.5,1)--(1.33,2);
						\draw[thick]	(2.5,1)--(1.66,2);
						\draw (1.5,2.5) node {{$\alpha$}} ; 
						\draw (0.5,0.5) node {{$x$}} ; 
						\draw (2.5,0.5) node {{$x$}} ; 	
			\end{tikzpicture}}}
			\ + \ 
			\scalebox{1.1}{$\frac{1}{2}$}\ 
			\vcenter{\hbox{\begin{tikzpicture}[scale=0.4]
						\draw[thick]	(0,2)--(1,2)--(1,3)--(0,3)--cycle;
						\draw[thick]	(2,2)--(3,2)--(3,3)--(2,3)--cycle;
						\draw[thick]	(1,0)--(2,0)--(2,1)--(1,1)--cycle;
						\draw[thick]	(0.5,2)--(1.33,1);
						\draw[thick]	(2.5,2)--(1.66,1);
						\draw (1.5,0.5) node {{$\alpha$}} ; 
						\draw (0.5,2.5) node {{$y$}} ; 
						\draw (2.5,2.5) node {{$y$}} ; 	
			\end{tikzpicture}}}
			\ + \ \\&
			\vcenter{\hbox{\begin{tikzpicture}[scale=0.4]
						\draw[thick]	(0,0)--(1,0)--(1,1)--(0,1)--cycle;
						\draw[thick]	(0,2)--(1,2)--(1,3)--(0,3)--cycle;
						\draw[thick]	(0,4)--(1,4)--(1,5)--(0,5)--cycle;
						\draw[thick]	(0.5,1)--(0.5,2);
						\draw[thick]	(0.5,3)--(0.5,4);	
						\draw (0.5,0.5) node {{$x$}} ; 
						\draw (0.5,2.5) node {{$\alpha$}} ; 
						\draw (0.5,4.5) node {{$y$}} ; 
			\end{tikzpicture}}}
			\ + \  
			\vcenter{\hbox{\begin{tikzpicture}[scale=0.4]
						\draw[thick]	(0,0)--(1,0)--(1,1)--(0,1)--cycle;
						\draw[thick]	(0,2)--(1,2)--(1,3)--(0,3)--cycle;
						\draw[thick]	(0,4)--(1,4)--(1,5)--(0,5)--cycle;
						\draw[thick]	(0.33,3)--(0.33,4);
						\draw[thick]	(0.66,1) to[out=60,in=270] (1.5,2.5) to[out=90,in=300] (0.66,4);	
						\draw (0.5,0.5) node {{$x$}} ; 
						\draw (0.5,2.5) node {{$\alpha$}} ; 
						\draw (0.5,4.5) node {{$y$}} ; 
			\end{tikzpicture}}}
			\ + \
			\vcenter{\hbox{\begin{tikzpicture}[scale=0.4]
						\draw[thick]	(0,0)--(1,0)--(1,1)--(0,1)--cycle;
						\draw[thick]	(0,2)--(1,2)--(1,3)--(0,3)--cycle;
						\draw[thick]	(0,4)--(1,4)--(1,5)--(0,5)--cycle;
						\draw[thick]	(0.33,1)--(0.33,2);
						\draw[thick]	(0.66,1) to[out=60,in=270] (1.5,2.5) to[out=90,in=300] (0.66,4);	
						\draw (0.5,0.5) node {{$x$}} ; 
						\draw (0.5,2.5) node {{$\alpha$}} ; 
						\draw (0.5,4.5) node {{$y$}} ; 
			\end{tikzpicture}}}
			\ + \  
			\vcenter{\hbox{\begin{tikzpicture}[scale=0.4]
						\draw[thick]	(0,0)--(1,0)--(1,1)--(0,1)--cycle;
						\draw[thick]	(0,2)--(1,2)--(1,3)--(0,3)--cycle;
						\draw[thick]	(0,4)--(1,4)--(1,5)--(0,5)--cycle;
						\draw[thick]	(0.33,1)--(0.33,2);
						\draw[thick]	(0.33,3)--(0.33,4);
						\draw[thick]	(0.66,1) to[out=60,in=270] (1.5,2.5) to[out=90,in=300] (0.66,4);	
						\draw (0.5,0.5) node {{$x$}} ; 
						\draw (0.5,2.5) node {{$\alpha$}} ; 
						\draw (0.5,4.5) node {{$y$}} ; 
			\end{tikzpicture}}}
			\ + \ 
			\scalebox{1.1}{$\frac{1}{2}$}\ 
			\vcenter{\hbox{\begin{tikzpicture}[scale=0.4]
						\draw[thick]	(0,0)--(1,0)--(1,1)--(0,1)--cycle;
						\draw[thick]	(2,0)--(3,0)--(3,1)--(2,1)--cycle;
						\draw[thick]	(1,2)--(2,2)--(2,3)--(1,3)--cycle;
						\draw[thick]	(1,4)--(2,4)--(2,5)--(1,5)--cycle;	
						\draw[thick]	(0.5,1)--(1.33,2);
						\draw[thick]	(2.5,1)--(1.66,2);
						\draw[thick]	(1.5,3)--(1.5,4);		
						\draw (1.5,2.5) node {{$\alpha$}} ; 
						\draw (0.5,0.5) node {{$x$}} ; 
						\draw (2.5,0.5) node {{$x$}} ; 	
						\draw (1.5,4.5) node {{$y$}} ; 	
			\end{tikzpicture}}}
			\ + \ 
			\vcenter{\hbox{\begin{tikzpicture}[scale=0.4]
						\draw[thick] (0.66,1)--(1.33,2);
						\draw[thick] (2.33,1)--(1.66,2);
						\draw[thick] (2.66,1) to (2.66,2.5) to[out=90,in=300]  (1.66,4);
						\draw[thick]	(0,0)--(1,0)--(1,1)--(0,1)--cycle;
						\draw[thick]	(2,0)--(3,0)--(3,1)--(2,1)--cycle;
						\draw[thick]	(1,2)--(2,2)--(2,3)--(1,3)--cycle;
						\draw[thick]	(1,4)--(2,4)--(2,5)--(1,5)--cycle;	
						\draw[thick]	(1.5,3)--(1.5,4);	
						\draw (1.5,4.5) node {{$y$}} ; 
						\draw (1.5,2.5) node {{$\alpha$}} ; 
						\draw (0.5,0.5) node {{$x$}} ; 
						\draw (2.5,0.5) node {{$x$}} ; 	
			\end{tikzpicture}}}
			\ + \ \cdots
		\end{align*}
		\caption{The first terms of the bowtie element 
			$(\1+x)\stackrel{\alpha}{\Join} (\1+y)$~.
		}
		\label{Fig:Pap}
	\end{figure*}
\end{definition}

\begin{theorem}\label{thm:Action}
	For any complete dg Lie-graph algebra $\g$, the deformation gauge group $\G$ acts on the set $\MC(\g)$ of Maurer--Cartan elements via the formula
	\[\exp(\lambda)\cdot \alpha\coloneq \lambda.\alpha =
	\pap{\exp(\lambda)}{\alpha}{\exp(-\lambda)} \ ,\]
	in the  differential extension $\g^+$, and via the formula 
\begin{align*}
\exp(\lambda)\cdot \ba \coloneq\lambda. \ba&=
	\pap{\exp(\lambda)}{\ba}{\exp(-\lambda)}
	-\left(\exp(\lambda); \d \big(\exp(\lambda)\big)\right)\cc \exp(-\lambda)
\\&=	
		\pap{\exp(\lambda)}{\ba}{\exp(-\lambda)}
	+\exp(\lambda)   \cc \left(\exp(-\lambda); \d \big(\exp(-\lambda)\big)\right)\ ,
\end{align*}

	in $\g$, where the notation $\left(x; y\right)$ means that we label all the vertices of the level by $x$ except for one which is labelled by $y$. 
\end{theorem}

\begin{proof}
	The deformation gauge group action comes from the gauge group action via the exponential map, which is a morphism by \cref{thm:ExpIsoMorph}, so it defines a genuine group action. 
	
	\medskip
	
	The first formula amounts to prove 
	\[e^{\ad_\lambda}(\alpha)=\pap{\exp(\lambda)}{\alpha}{\exp(-\lambda)}~.\]
	To this extent, we use the same method as in the proof of \cref{thm:ExpIsoMorph}: 
	it is enough to prove this for the free complete Lie-graph algebra 
	$\widehat{\Liegra}_{\mathbb{Q}}(\alpha, \lambda)$
	over $\mathbb{Q}$ on two generators of respective degrees $|\alpha|=-1$ and $|\lambda|=0$.
	Since this latter one is complete, it is enough to prove it for its truncated version 
	$\widehat{\Liegra}_{\QQ}(\alpha, \lambda)/\F_k\widehat{\Liegra}_{\QQ}(\alpha, \lambda)$~, 
	for any $k\geqslant 1$. For any fixed $k$, we work with its real extension 
	$\widehat{\Liegra}_{\RR}(\alpha, \lambda)/\F_k\widehat{\Liegra}_{\RR}(\alpha, \lambda)$, which is a finite-dimensional real graded Lie algebra. 
	
	\medskip
	
	For any finite dimensional real graded Lie algebra $\mathfrak{h}$, the element $e^{\ad_\lambda}(\alpha)$ is the unique solution at $t=1$ of the differential equation 
	$\frac{d}{dt}\psi(t)=\ad_\lambda(\psi(t))$, with initial condition $\psi(0)=\alpha$, in the algebraic variety 
	$\MC(\mathfrak{h})$~. To conclude, it is enough to show that 
	$\pap{\exp(t\lambda)}{\alpha}{\exp(-t\lambda)}$
	satisfies this differential equation, that is 
	\[
	\frac{d}{dt} \pap{\exp(t\lambda)}{\alpha}{\exp(-t\lambda)}=
	\lambda\star \left( \pap{\exp(t\lambda)}{\alpha}{\exp(-t\lambda)}\right) 
	- \left( \pap{\exp(t\lambda)}{\alpha}{\exp(-t\lambda)}\right)\star \lambda~.
	\]
	First, the bowtie element at $t=0$ is equal to $\pap{\exp(0)}{\alpha}{\exp(0)}=\alpha$. 
	Then, we recall, from the proof of 
	\cref{thm:ExpIsoMorph}, that 
	\[\frac{d}{dt}\exp(t\lambda) =\exp(t\lambda)\rhd \lambda = \lambda \lhd \exp(t\lambda)~,\]
	where the latter notion $\lhd$ is vertical mirror to the former notion $\rhd$~. 
	This implies  that 
	\begin{align*}
		\frac{d}{dt} \pap{\exp(t\lambda)}{\alpha}{\exp(-t\lambda)}=& 
		\pap{\big(\exp(t\lambda); \lambda\lhd  \exp(t\lambda) \big)}
		{\alpha}
		{\exp(-t\lambda)}-
		\\& 
		\pap
		{\exp(t\lambda)}
		{\alpha}
		{\big(\exp(-t\lambda); \exp(-t\lambda) \rhd \lambda \big)}\ .
	\end{align*}
	Finally we claim that 
	\begin{align}\label{proof:action}
		&\pap{\big(\exp(t\lambda); \lambda\lhd  \exp(t\lambda) \big)}
		{\alpha}
		{\exp(-t\lambda)}
		= 
		\lambda\star \left( \pap{\exp(t\lambda)}{\alpha}{\exp(-t\lambda)}\right)
		\qquad \text{and} \\ \notag
		&\pap
		{\exp(t\lambda)}
		{\alpha}
		{\big(\exp(-t\lambda); \exp(-t\lambda) \rhd \lambda \big)}
		= \left( \pap{\exp(t\lambda)}{\alpha}{\exp(-t\lambda)}\right)\star \lambda
		~,
	\end{align}
	
	which would conclude the proof. 
	\medskip
	
	Let us now show the first equation of \eqref{proof:action}; the proof of the other one being similar. The associativity relation (\cref{thm:dcGraGp}) of 
	the product $\cc$  gives 
	\[
	\big((\1+\lambda)\cc \exp(t\lambda)\big) \cc \big((\1+\alpha)\cc \exp(-t\lambda)\big)=
	(\1+\lambda)\cc \big(\exp(t\lambda)\big) \cc (\1+\alpha)\cc \exp(-t\lambda)\big)~, 
	\]
	which is given by the sum over the set of 4-leveled directed simple graphs. 
	The projection of this equation onto the sum over the directed simple graphs with only one vertex on the first and on the third level gives 
	
	\begin{align*}
		\pap{\big(\exp(t\lambda); \lambda\lhd  \exp(t\lambda) \big)}
		{\alpha}
		{\exp(-t\lambda)}
		&=\lambda \lhd  \big(\exp(t\lambda)\big) \cc  \exp(-t\lambda),\pap{\exp(t\lambda)}
		{\alpha}
		{\exp(-t\lambda)}
		\big)\\&=\lambda\star \left( \pap{\exp(t\lambda)}{\alpha}{\exp(-t\lambda)}\right)
		~, 
	\end{align*}
	since $\exp(t\lambda)\big) \cc  \exp(-t\lambda)=\1$~. 
	
	\medskip 
	
	In order to settle the second formula, it remains to show that 
	\[\delta-e^{\ad_\lambda}(\delta)=\delta-\pap{\exp(\lambda)}{\delta}{\exp(-\lambda)}=
	\frac{e^{\ad_\lambda}-\id}{\ad_{\lambda}}(\d\lambda)=\left(\exp(\lambda); \d \big(\exp(\lambda)\big)\right)\cc \exp(-\lambda)~,\]
	where the first equalities are already known to hold in $\g^+$. To this extent, we use the same kind of idea as above: we consider the projection of the associativity relation (\cref{thm:dcGraGp}) 
	\[
	(\1+\delta)\cc \big(\exp(\lambda) \cc (\1-\delta)\cc \exp(-\lambda)\big)
	=
	\big((\1+\delta)\cc \exp(\lambda) \cc (\1-\delta)\big)\cc \exp(-\lambda)~, 
	\]
	onto the summand made up of only one $\delta$~. The left-hand side gives 
	\[\delta-\pap{\exp(\lambda)}{\delta}{\exp(-\lambda)}\]
	and the right-hand side gives  
	\[\left(\exp(\lambda); \delta \lhd \exp{\lambda} - \exp(\lambda)\rhd \delta \right)\cc \exp(-\lambda)~.\]
	We claim that $\delta \lhd \exp{\lambda} - \exp(\lambda)\rhd = \delta \star \exp{\lambda} - \exp(\lambda)\star \delta=\d \big(\exp(\lambda)\big)$, which would conclude the proof. 
	
	\medskip
	
	{
		This relation comes from the fact that 
		\begin{align*}
			\vcenter{\hbox{\begin{tikzpicture}[scale=0.35]
						\coordinate (A) at (0, 0);
						\draw[thick]	($(A)$)--($(A)+(0, 1.5)$)--($(A)+(1.5,1.5)$)--($(A)+(1.5, 0)$)--cycle;
						\draw ($(A)+(0.75,0.75)$) node {\scalebox{1}{$x_1$}} ; 	
						\coordinate (A) at (3, 0);
						\draw[thick]	($(A)$)--($(A)+(0, 1.5)$)--($(A)+(1.5,1.5)$)--($(A)+(1.5, 0)$)--cycle;
						\draw ($(A)+(0.75,0.75)$) node {\scalebox{1}{$x_i$}} ; 	
						\coordinate (A) at (6, 0);
						\draw[thick]	($(A)$)--($(A)+(0, 1.5)$)--($(A)+(1.5,1.5)$)--($(A)+(1.5, 0)$)--cycle;
						\draw ($(A)+(0.75,0.75)$) node {\scalebox{1}{$x_n$}} ; 	
						\node at (2.3, 0.7) {\scalebox{0.8}{$\cdots$}} ; 	
						\node at (5.3, 0.7) {\scalebox{0.8}{$\cdots$}} ; 		
						\coordinate (A) at (3, -3);
						\draw[thick]	($(A)$)--($(A)+(0, 1.5)$)--($(A)+(1.5,1.5)$)--($(A)+(1.5, 0)$)--cycle;
						\draw ($(A)+(0.75,0.75)$) node {\scalebox{1}{$\delta$}} ; 	
						\draw[thick] (3.5, -1.5)--(0.75,0);	
						\draw[thick] (3.75, -1.5)--(3.75,0);
						\draw[thick] (4, -1.5)--(6.75,0);		
			\end{tikzpicture}}}
			=
			\vcenter{\hbox{\begin{tikzpicture}[scale=0.35]
						\coordinate (A) at (0, 0);
						\draw[thick]	($(A)$)--($(A)+(0, 1.5)$)--($(A)+(1.5,1.5)$)--($(A)+(1.5, 0)$)--cycle;
						\draw ($(A)+(0.75,0.75)$) node {\scalebox{1}{$x_1$}} ; 	
						\coordinate (A) at (3, 0);
						\draw[thick]	($(A)$)--($(A)+(0, 1.5)$)--($(A)+(1.5,1.5)$)--($(A)+(1.5, 0)$)--cycle;
						\draw ($(A)+(0.75,0.75)$) node {\scalebox{1}{$x_i$}} ; 	
						\coordinate (A) at (6, 0);
						\draw[thick]	($(A)$)--($(A)+(0, 1.5)$)--($(A)+(1.5,1.5)$)--($(A)+(1.5, 0)$)--cycle;
						\draw ($(A)+(0.75,0.75)$) node {\scalebox{1}{$x_n$}} ; 	
						\node at (2.3, 0.7) {\scalebox{0.8}{$\cdots$}} ; 	
						\node at (5.3, 0.7) {\scalebox{0.8}{$\cdots$}} ; 		
						\coordinate (A) at (3, 3);
						\draw[thick]	($(A)$)--($(A)+(0, 1.5)$)--($(A)+(1.5,1.5)$)--($(A)+(1.5, 0)$)--cycle;
						\draw ($(A)+(0.75,0.75)$) node {\scalebox{1}{$\delta$}} ; 		
						\draw[thick] (3.5, 3)--(0.75,1.5);
						\draw[thick] (3.75, 3)--(3.75,1.5);
						\draw[thick] (4, 3)--(6.75,1.5);		
			\end{tikzpicture}}}
		\end{align*}
		in $\g^+$, for any $n\geqslant 2$ and for any $x_1, \ldots, x_n\in \g$.
		In order to prove it, we consider the operad 
		\[\mathrm{d}\Liencgra^+\coloneq \frac{\Liencgra\vee \mathcal{T}(\d, \delta)}{\left(\d \ \text{derivation},\,  \d^2~,\,  
			\d-
			\begin{tikzpicture}[scale=0.4, baseline=(n.base)]
				\node (n) at (0.5,1.3) {};
				\draw[thick]	(0,0)--(1,0)--(1,1)--(0,1)--cycle;
				\draw[thick]	(0,2)--(1,2)--(1,3)--(0,3)--cycle;
				\draw[thick]	(0.5,1)--(0.5,2);
				\draw (0.5,0.5) node {{$\delta$}} ; 
			\end{tikzpicture}
			+
			\begin{tikzpicture}[scale=0.4, baseline=(n.base)]
				\node (n) at (0.5,1.3) {};
				\draw[thick]	(0,0)--(1,0)--(1,1)--(0,1)--cycle;
				\draw[thick]	(0,2)--(1,2)--(1,3)--(0,3)--cycle;
				\draw[thick]	(0.5,1)--(0.5,2);
				\draw (0.5,2.5) node {{$\delta$}} ; 
			\end{tikzpicture}\ , \d(\delta)
			\right)}\]
		encoding internal dg $\Liencgra$-algebras, together with the canonical morphism 
		$\eta \colon \mathrm{d}\Liegra^+ \hookrightarrow \mathrm{d}\Liencgra^+$ that we claim to be injective.
		The derivation relation endows the operad $\mathrm{d}\Liegra$ with a distributive law $\mathrm{D} \circ \Liegra \to \Liegra \circ \mathrm{D}$, where $\mathrm{D}\coloneq \mathcal{T}(\d)/\big(\d^2\big)$ stands for the algebra of dual numbers. This proves that its underlying graded $\Sy$-module is isomorphic to $\mathrm{d}\Liegra\cong \Liegra \circ \mathrm{D}$~. The same property holds for the operad $\mathrm{d}\Liencgra$, whose underlying graded $\Sy$-module is isomorphic to 
		\[\mathrm{d}\Liencgra\cong \mathrm{Com}\circ \Liegra \circ \mathrm{D}\cong  \Liegra \circ \mathrm{D}\oplus \overline{\mathrm{Com}}\circ \Liegra \circ \mathrm{D}~,\]
		where $\overline{\mathrm{Com}}\cong \mathrm{Com}/\mathrm{I}$~. This shows that the canonical morphism of operads 
		$\bar{\eta} \colon \mathrm{d}\Liegra \hookrightarrow \allowbreak \mathrm{d}\Liencgra$ is injective. It remains to show that its extension $\eta$ with a internal degree $-1$ arity $0$ element is still injective. To this extent, we consider the increasing filtration on $\mathrm{d}\Liencgra^+$ defined by the number of connected components or equivalently by the arity of the sub-operad $\mathrm{Com}$. 
		We claim that the underlying $\Sy$-module of the associated graded operad is isomorphic to 
		\[\mathrm{gr}\, \mathrm{d}\Liencgra^+\cong  \Com \circ \mathrm{d}\Liegra^+\cong 
		\mathrm{d}\Liegra^+\oplus \overline{\Com} \circ \mathrm{d}\Liegra^+~, \]
		which would conclude the injectivity of the morphism $\eta$ since the underlying $\Sy$-modules satisfy $\mathrm{d}\Liencgra^+\cong \mathrm{gr}\, \mathrm{d}\Liencgra^+$~. 
		The operadic ideal 
		\[\left(  
		\d-
		\begin{tikzpicture}[scale=0.4, baseline=(n.base)]
			\node (n) at (0.5,1.3) {};
			\draw[thick]	(0,0)--(1,0)--(1,1)--(0,1)--cycle;
			\draw[thick]	(0,2)--(1,2)--(1,3)--(0,3)--cycle;
			\draw[thick]	(0.5,1)--(0.5,2);
			\draw (0.5,0.5) node {{$\delta$}} ; 
		\end{tikzpicture}
		+
		\begin{tikzpicture}[scale=0.4, baseline=(n.base)]
			\node (n) at (0.5,1.3) {};
			\draw[thick]	(0,0)--(1,0)--(1,1)--(0,1)--cycle;
			\draw[thick]	(0,2)--(1,2)--(1,3)--(0,3)--cycle;
			\draw[thick]	(0.5,1)--(0.5,2);
			\draw (0.5,2.5) node {{$\delta$}} ; 
		\end{tikzpicture}\ , \d(\delta)
		\right)\]
		inside $\mathrm{Com}\circ \Liegra \circ \mathrm{D}\vee \mathcal{T}(\delta)$ is made up of two types of terms. 
		The first ones are given by non-necessarily connected graphs with at least one relation above inserted at one vertex. This preserves the connected components. 
		The second ones are given by non-necessarily connected graphs inserted at the only vertex of the first displayed relation above. 
		Since $\d$ is a derivation, it distributes over the connected components of the graph. The other term 
		\[ \begin{tikzpicture}[scale=0.4, baseline=(n.base)]
			\node (n) at (0.5,1.3) {};
			\draw[thick]	(0,0)--(1,0)--(1,1)--(0,1)--cycle;
			\draw[thick]	(0,2)--(1,2)--(1,3)--(0,3)--cycle;
			\draw[thick]	(0.5,1)--(0.5,2);
			\draw (0.5,0.5) node {{$\delta$}} ; 
		\end{tikzpicture}
		+
		\begin{tikzpicture}[scale=0.4, baseline=(n.base)]
			\node (n) at (0.5,1.3) {};
			\draw[thick]	(0,0)--(1,0)--(1,1)--(0,1)--cycle;
			\draw[thick]	(0,2)--(1,2)--(1,3)--(0,3)--cycle;
			\draw[thick]	(0.5,1)--(0.5,2);
			\draw (0.5,2.5) node {{$\delta$}} ; 
		\end{tikzpicture}\]
		fails to be a derivation in $\mathrm{d}\Liencgra^+$ because it might create terms where several connected components are linked to a $\delta$, but such terms are killed in $\mathrm{gr}\, \mathrm{d}\Liencgra^+$ by the definition of the filtration. 
		So this gives rise to  a derivation in 
		$\mathrm{gr}\, \mathrm{d}\Liencgra^+$ which distributes too over the connected components of the graph. 
		In the end, this shows that the above ideal preserves the connected components of the graphs in the operad $\mathrm{gr}\, \mathrm{d}\Liencgra^+$, which implies the isomorphism 
		$\mathrm{gr}\, \mathrm{d}\Liencgra^+\cong  \Com \circ \mathrm{d}\Liegra^+$~. 
		
		\medskip 
		
		Using the same arguments as above, namely \cref{prop:inclusion}, this gives rise to a pair of adjoint functors $\eta_!$ and $\eta^*$ whose unit 
		$\g^+ \hookrightarrow \eta^*\big( \eta_!\big(\g^+\big)\big)$ is an embedding of complete internal dg Lie-graph algebras. 
		We work by induction on $n\geqslant 2$ in the latter one. For $n=2$, we compute the following term in two ways: 
		\begin{eqnarray*}
			\d\left(
			\vcenter{\hbox{\begin{tikzpicture}[scale=0.35]
						\coordinate (A) at (0, 0);
						\draw[thick]	($(A)$)--($(A)+(0, 1.5)$)--($(A)+(1.5,1.5)$)--($(A)+(1.5, 0)$)--cycle;
						\draw ($(A)+(0.75,0.75)$) node {\scalebox{0.8}{$x_1$}} ; 	
						\coordinate (A) at (2, 0);
						\draw[thick]	($(A)$)--($(A)+(0, 1.5)$)--($(A)+(1.5,1.5)$)--($(A)+(1.5, 0)$)--cycle;
						\draw ($(A)+(0.75,0.75)$) node {\scalebox{0.8}{$x_2$}} ; 	
			\end{tikzpicture}}}
			\right)=
			\vcenter{\hbox{\begin{tikzpicture}[scale=0.35]
						\coordinate (A) at (0, 0);
						\draw[thick]	($(A)$)--($(A)+(0, 1.5)$)--($(A)+(1.5,1.5)$)--($(A)+(1.5, 0)$)--cycle;
						\draw ($(A)+(0.75,0.75)$) node {\scalebox{0.8}{$\d x_1$}} ; 	
						\coordinate (A) at (2, 0);
						\draw[thick]	($(A)$)--($(A)+(0, 1.5)$)--($(A)+(1.5,1.5)$)--($(A)+(1.5, 0)$)--cycle;
						\draw ($(A)+(0.75,0.75)$) node {\scalebox{0.8}{$x_2$}} ; 	
			\end{tikzpicture}}}
			+
			\vcenter{\hbox{\begin{tikzpicture}[scale=0.35]
						\coordinate (A) at (0, 0);
						\draw[thick]	($(A)$)--($(A)+(0, 1.5)$)--($(A)+(1.5,1.5)$)--($(A)+(1.5, 0)$)--cycle;
						\draw ($(A)+(0.75,0.75)$) node {\scalebox{0.8}{$x_1$}} ; 	
						\coordinate (A) at (2, 0);
						\draw[thick]	($(A)$)--($(A)+(0, 1.5)$)--($(A)+(1.5,1.5)$)--($(A)+(1.5, 0)$)--cycle;
						\draw ($(A)+(0.75,0.75)$) node {\scalebox{0.8}{$\d x_2$}} ; 	
			\end{tikzpicture}}}=
			\vcenter{\hbox{\begin{tikzpicture}[scale=0.5]
						\coordinate (A) at (0, 0);
						\draw[thick]	($(A)$)--($(A)+(0, 1)$)--($(A)+(1,1)$)--($(A)+(1,0)$)--cycle;
						\draw ($(A)+(0.5,0.5)$) node {\scalebox{0.8}{$x_1$}} ; 	
						\coordinate (A) at (1.5, -1);
						\draw[thick]	($(A)$)--($(A)+(0, 1)$)--($(A)+(1,1)$)--($(A)+(1,0)$)--cycle;
						\draw ($(A)+(0.5,0.5)$) node {\scalebox{0.8}{$x_2$}} ; 	
						\coordinate (A) at (0, -2);
						\draw[thick]	($(A)$)--($(A)+(0, 1)$)--($(A)+(1,1)$)--($(A)+(1,0)$)--cycle;
						\draw ($(A)+(0.5,0.5)$) node {\scalebox{0.8}{$\delta$}} ; 	
						\draw[thick] (0.5,0)--(0.5,-1);
			\end{tikzpicture}}}
			-\vcenter{\hbox{\begin{tikzpicture}[scale=0.5]
						\coordinate (A) at (0, 0);
						\draw[thick]	($(A)$)--($(A)+(0, 1)$)--($(A)+(1,1)$)--($(A)+(1,0)$)--cycle;
						\draw ($(A)+(0.5,0.5)$) node {\scalebox{0.8}{$\delta$}} ; 	
						\coordinate (A) at (1.5, -1);
						\draw[thick]	($(A)$)--($(A)+(0, 1)$)--($(A)+(1,1)$)--($(A)+(1,0)$)--cycle;
						\draw ($(A)+(0.5,0.5)$) node {\scalebox{0.8}{$x_2$}} ; 	
						\coordinate (A) at (0, -2);
						\draw[thick]	($(A)$)--($(A)+(0, 1)$)--($(A)+(1,1)$)--($(A)+(1,0)$)--cycle;
						\draw ($(A)+(0.5,0.5)$) node {\scalebox{0.8}{$x_1$}} ; 	
						\draw[thick] (0.5,0)--(0.5,-1);
			\end{tikzpicture}}}
			+
			\vcenter{\hbox{\begin{tikzpicture}[scale=0.5]
						\coordinate (A) at (0, -1);
						\draw[thick]	($(A)$)--($(A)+(0, 1)$)--($(A)+(1,1)$)--($(A)+(1,0)$)--cycle;
						\draw ($(A)+(0.5,0.5)$) node {\scalebox{0.8}{$x_1$}} ; 	
						\coordinate (A) at (1.5, 0);
						\draw[thick]	($(A)$)--($(A)+(0, 1)$)--($(A)+(1,1)$)--($(A)+(1,0)$)--cycle;
						\draw ($(A)+(0.5,0.5)$) node {\scalebox{0.8}{$x_2$}} ; 	
						\coordinate (A) at (1.5, -2);
						\draw[thick]	($(A)$)--($(A)+(0, 1)$)--($(A)+(1,1)$)--($(A)+(1,0)$)--cycle;
						\draw ($(A)+(0.5,0.5)$) node {\scalebox{0.8}{$\delta$}} ; 	
						\draw[thick] (2,0)--(2,-1);
			\end{tikzpicture}}}
			-\vcenter{\hbox{\begin{tikzpicture}[scale=0.5]
						\coordinate (A) at (1.5, 0);
						\draw[thick]	($(A)$)--($(A)+(0, 1)$)--($(A)+(1,1)$)--($(A)+(1,0)$)--cycle;
						\draw ($(A)+(0.5,0.5)$) node {\scalebox{0.8}{$\delta$}} ; 	
						\coordinate (A) at (1.5, -2);
						\draw[thick]	($(A)$)--($(A)+(0, 1)$)--($(A)+(1,1)$)--($(A)+(1,0)$)--cycle;
						\draw ($(A)+(0.5,0.5)$) node {\scalebox{0.8}{$x_2$}} ; 	
						\coordinate (A) at (0, -1);
						\draw[thick]	($(A)$)--($(A)+(0, 1)$)--($(A)+(1,1)$)--($(A)+(1,0)$)--cycle;
						\draw ($(A)+(0.5,0.5)$) node {\scalebox{0.8}{$x_1$}} ; 	
						\draw[thick] (2,0)--(2,-1);
			\end{tikzpicture}}}
		\end{eqnarray*}
		and 
		\begin{align*}
			\d\left(
			\vcenter{\hbox{\begin{tikzpicture}[scale=0.35]
						\coordinate (A) at (0, 0);
						\draw[thick]	($(A)$)--($(A)+(0, 1.5)$)--($(A)+(1.5,1.5)$)--($(A)+(1.5, 0)$)--cycle;
						\draw ($(A)+(0.75,0.75)$) node {\scalebox{0.8}{$x_1$}} ; 	
						\coordinate (A) at (2, 0);
						\draw[thick]	($(A)$)--($(A)+(0, 1.5)$)--($(A)+(1.5,1.5)$)--($(A)+(1.5, 0)$)--cycle;
						\draw ($(A)+(0.75,0.75)$) node {\scalebox{0.8}{$x_2$}} ; 	
			\end{tikzpicture}}}
			\right)&=\ 
			\vcenter{\hbox{\begin{tikzpicture}[scale=0.5]
						\coordinate (A) at (0, 0);
						\draw[thick]	($(A)$)--($(A)+(0, 1)$)--($(A)+(1,1)$)--($(A)+(1,0)$)--cycle;
						\draw ($(A)+(0.5,0.5)$) node {\scalebox{0.8}{$x_1$}} ; 	
						\coordinate (A) at (1.5, 0);
						\draw[thick]	($(A)$)--($(A)+(0, 1)$)--($(A)+(1,1)$)--($(A)+(1,0)$)--cycle;
						\draw ($(A)+(0.5,0.5)$) node {\scalebox{0.8}{$x_2$}} ; 	
						\coordinate (A) at (0.75, -2.5);
						\draw[thick]	($(A)$)--($(A)+(0, 1)$)--($(A)+(1,1)$)--($(A)+(1,0)$)--cycle;
						\draw ($(A)+(0.5,0.5)$) node {\scalebox{0.8}{$\delta$}} ; 	
						\draw[thick] (-0.5,-0.5)--(3,-0.5)--(3,1.5)--(-0.5,1.5)--cycle;
						\draw[thick] (1.25,-0.5)--(1.25,-1.5);
			\end{tikzpicture}}}
			\ -\ 
			\vcenter{\hbox{\begin{tikzpicture}[scale=0.5]
						\coordinate (A) at (0, 0);
						\draw[thick]	($(A)$)--($(A)+(0, 1)$)--($(A)+(1,1)$)--($(A)+(1,0)$)--cycle;
						\draw ($(A)+(0.5,0.5)$) node {\scalebox{0.8}{$x_1$}} ; 	
						\coordinate (A) at (1.5, 0);
						\draw[thick]	($(A)$)--($(A)+(0, 1)$)--($(A)+(1,1)$)--($(A)+(1,0)$)--cycle;
						\draw ($(A)+(0.5,0.5)$) node {\scalebox{0.8}{$x_2$}} ; 	
						\coordinate (A) at (0.75, 2.5);
						\draw[thick]	($(A)$)--($(A)+(0, 1)$)--($(A)+(1,1)$)--($(A)+(1,0)$)--cycle;
						\draw ($(A)+(0.5,0.5)$) node {\scalebox{0.8}{$\delta$}} ; 	
						\draw[thick] (-0.5,-0.5)--(3,-0.5)--(3,1.5)--(-0.5,1.5)--cycle;
						\draw[thick] (1.25,2.5)--(1.25,1.5);
			\end{tikzpicture}}}\\
			&=
			\vcenter{\hbox{\begin{tikzpicture}[scale=0.5]
						\coordinate (A) at (0, 0);
						\draw[thick]	($(A)$)--($(A)+(0, 1)$)--($(A)+(1,1)$)--($(A)+(1,0)$)--cycle;
						\draw ($(A)+(0.5,0.5)$) node {\scalebox{0.8}{$x_1$}} ; 	
						\coordinate (A) at (1.5, -1);
						\draw[thick]	($(A)$)--($(A)+(0, 1)$)--($(A)+(1,1)$)--($(A)+(1,0)$)--cycle;
						\draw ($(A)+(0.5,0.5)$) node {\scalebox{0.8}{$x_2$}} ; 	
						\coordinate (A) at (0, -2);
						\draw[thick]	($(A)$)--($(A)+(0, 1)$)--($(A)+(1,1)$)--($(A)+(1,0)$)--cycle;
						\draw ($(A)+(0.5,0.5)$) node {\scalebox{0.8}{$\delta$}} ; 	
						\draw[thick] (0.5,0)--(0.5,-1);
			\end{tikzpicture}}}
			+
			\vcenter{\hbox{\begin{tikzpicture}[scale=0.5]
						\coordinate (A) at (0, -1);
						\draw[thick]	($(A)$)--($(A)+(0, 1)$)--($(A)+(1,1)$)--($(A)+(1,0)$)--cycle;
						\draw ($(A)+(0.5,0.5)$) node {\scalebox{0.8}{$x_1$}} ; 	
						\coordinate (A) at (1.5, 0);
						\draw[thick]	($(A)$)--($(A)+(0, 1)$)--($(A)+(1,1)$)--($(A)+(1,0)$)--cycle;
						\draw ($(A)+(0.5,0.5)$) node {\scalebox{0.8}{$x_2$}} ; 	
						\coordinate (A) at (1.5, -2);
						\draw[thick]	($(A)$)--($(A)+(0, 1)$)--($(A)+(1,1)$)--($(A)+(1,0)$)--cycle;
						\draw ($(A)+(0.5,0.5)$) node {\scalebox{0.8}{$\delta$}} ; 	
						\draw[thick] (2,0)--(2,-1);
			\end{tikzpicture}}}
			+
			\vcenter{\hbox{\begin{tikzpicture}[scale=0.5]
						\coordinate (A) at (0, 0);
						\draw[thick]	($(A)$)--($(A)+(0, 1)$)--($(A)+(1,1)$)--($(A)+(1,0)$)--cycle;
						\draw ($(A)+(0.5,0.5)$) node {\scalebox{0.8}{$x_1$}} ; 	
						\coordinate (A) at (1.5, 0);
						\draw[thick]	($(A)$)--($(A)+(0, 1)$)--($(A)+(1,1)$)--($(A)+(1,0)$)--cycle;
						\draw ($(A)+(0.5,0.5)$) node {\scalebox{0.8}{$x_2$}} ; 	
						\coordinate (A) at (0.75, -2);
						\draw[thick]	($(A)$)--($(A)+(0, 1)$)--($(A)+(1,1)$)--($(A)+(1,0)$)--cycle;
						\draw ($(A)+(0.5,0.5)$) node {\scalebox{0.8}{$\delta$}} ; 	
						\draw[thick] (2,0)--(1.42,-1);
						\draw[thick] (0.5,0)--(1.08,-1);	
			\end{tikzpicture}}}
			-\vcenter{\hbox{\begin{tikzpicture}[scale=0.5]
						\coordinate (A) at (0, 0);
						\draw[thick]	($(A)$)--($(A)+(0, 1)$)--($(A)+(1,1)$)--($(A)+(1,0)$)--cycle;
						\draw ($(A)+(0.5,0.5)$) node {\scalebox{0.8}{$\delta$}} ; 	
						\coordinate (A) at (1.5, -1);
						\draw[thick]	($(A)$)--($(A)+(0, 1)$)--($(A)+(1,1)$)--($(A)+(1,0)$)--cycle;
						\draw ($(A)+(0.5,0.5)$) node {\scalebox{0.8}{$x_2$}} ; 	
						\coordinate (A) at (0, -2);
						\draw[thick]	($(A)$)--($(A)+(0, 1)$)--($(A)+(1,1)$)--($(A)+(1,0)$)--cycle;
						\draw ($(A)+(0.5,0.5)$) node {\scalebox{0.8}{$x_1$}} ; 	
						\draw[thick] (0.5,0)--(0.5,-1);
			\end{tikzpicture}}}
			-\vcenter{\hbox{\begin{tikzpicture}[scale=0.5]
						\coordinate (A) at (1.5, 0);
						\draw[thick]	($(A)$)--($(A)+(0, 1)$)--($(A)+(1,1)$)--($(A)+(1,0)$)--cycle;
						\draw ($(A)+(0.5,0.5)$) node {\scalebox{0.8}{$\delta$}} ; 	
						\coordinate (A) at (1.5, -2);
						\draw[thick]	($(A)$)--($(A)+(0, 1)$)--($(A)+(1,1)$)--($(A)+(1,0)$)--cycle;
						\draw ($(A)+(0.5,0.5)$) node {\scalebox{0.8}{$x_2$}} ; 	
						\coordinate (A) at (0, -1);
						\draw[thick]	($(A)$)--($(A)+(0, 1)$)--($(A)+(1,1)$)--($(A)+(1,0)$)--cycle;
						\draw ($(A)+(0.5,0.5)$) node {\scalebox{0.8}{$x_1$}} ; 	
						\draw[thick] (2,0)--(2,-1);
			\end{tikzpicture}}}
			-\vcenter{\hbox{\begin{tikzpicture}[scale=0.5]
						\coordinate (A) at (0, 0);
						\draw[thick]	($(A)$)--($(A)+(0, 1)$)--($(A)+(1,1)$)--($(A)+(1,0)$)--cycle;
						\draw ($(A)+(0.5,0.5)$) node {\scalebox{0.8}{$x_1$}} ; 	
						\coordinate (A) at (1.5, 0);
						\draw[thick]	($(A)$)--($(A)+(0, 1)$)--($(A)+(1,1)$)--($(A)+(1,0)$)--cycle;
						\draw ($(A)+(0.5,0.5)$) node {\scalebox{0.8}{$x_2$}} ; 	
						\coordinate (A) at (0.75, 2);
						\draw[thick]	($(A)$)--($(A)+(0, 1)$)--($(A)+(1,1)$)--($(A)+(1,0)$)--cycle;
						\draw ($(A)+(0.5,0.5)$) node {\scalebox{0.8}{$\delta$}} ; 	
						\draw[thick] (2,1)--(1.42,2);
						\draw[thick] (0.5,1)--(1.08,2);	
			\end{tikzpicture}}}~. 
		\end{align*}
		This shows 
		\[
		\vcenter{\hbox{\begin{tikzpicture}[scale=0.5]
					\coordinate (A) at (0, 0);
					\draw[thick]	($(A)$)--($(A)+(0, 1)$)--($(A)+(1,1)$)--($(A)+(1,0)$)--cycle;
					\draw ($(A)+(0.5,0.5)$) node {\scalebox{0.8}{$x_1$}} ; 	
					\coordinate (A) at (1.5, 0);
					\draw[thick]	($(A)$)--($(A)+(0, 1)$)--($(A)+(1,1)$)--($(A)+(1,0)$)--cycle;
					\draw ($(A)+(0.5,0.5)$) node {\scalebox{0.8}{$x_2$}} ; 	
					\coordinate (A) at (0.75, -2);
					\draw[thick]	($(A)$)--($(A)+(0, 1)$)--($(A)+(1,1)$)--($(A)+(1,0)$)--cycle;
					\draw ($(A)+(0.5,0.5)$) node {\scalebox{0.8}{$\delta$}} ; 	
					\draw[thick] (2,0)--(1.42,-1);
					\draw[thick] (0.5,0)--(1.08,-1);	
		\end{tikzpicture}}}
		=\vcenter{\hbox{\begin{tikzpicture}[scale=0.5]
					\coordinate (A) at (0, 0);
					\draw[thick]	($(A)$)--($(A)+(0, 1)$)--($(A)+(1,1)$)--($(A)+(1,0)$)--cycle;
					\draw ($(A)+(0.5,0.5)$) node {\scalebox{0.8}{$x_1$}} ; 	
					\coordinate (A) at (1.5, 0);
					\draw[thick]	($(A)$)--($(A)+(0, 1)$)--($(A)+(1,1)$)--($(A)+(1,0)$)--cycle;
					\draw ($(A)+(0.5,0.5)$) node {\scalebox{0.8}{$x_2$}} ; 	
					\coordinate (A) at (0.75, 2);
					\draw[thick]	($(A)$)--($(A)+(0, 1)$)--($(A)+(1,1)$)--($(A)+(1,0)$)--cycle;
					\draw ($(A)+(0.5,0.5)$) node {\scalebox{0.8}{$\delta$}} ; 	
					\draw[thick] (2,1)--(1.42,2);
					\draw[thick] (0.5,1)--(1.08,2);	
		\end{tikzpicture}}}~. 
		\]
		Suppose now that the result holds up to $n-1$ and let us prove it for $n$~. We apply the same idea as follows:
		\begin{align*}
			\d\left(
			\vcenter{\hbox{\begin{tikzpicture}[scale=0.35]
						\coordinate (A) at (0, 0);
						\draw[thick]	($(A)$)--($(A)+(0, 1.5)$)--($(A)+(1.5,1.5)$)--($(A)+(1.5, 0)$)--cycle;
						\draw ($(A)+(0.75,0.75)$) node {\scalebox{0.8}{$x_1$}} ; 	
						\coordinate (A) at (3, 0);
						\draw[thick]	($(A)$)--($(A)+(0, 1.5)$)--($(A)+(1.5,1.5)$)--($(A)+(1.5, 0)$)--cycle;
						\draw ($(A)+(0.75,0.75)$) node {\scalebox{0.8}{$x_i$}} ; 	
						\coordinate (A) at (6, 0);
						\draw[thick]	($(A)$)--($(A)+(0, 1.5)$)--($(A)+(1.5,1.5)$)--($(A)+(1.5, 0)$)--cycle;
						\draw ($(A)+(0.75,0.75)$) node {\scalebox{0.8}{$x_n$}} ; 	
						\node at (2.3, 0.7) {\scalebox{0.8}{$\cdots$}} ; 	
						\node at (5.3, 0.7) {\scalebox{0.8}{$\cdots$}} ; 		
			\end{tikzpicture}}}
			\right)&=
			\sum_{i=1}^n
			\vcenter{\hbox{\begin{tikzpicture}[scale=0.35]
						\coordinate (A) at (0, 0);
						\draw[thick]	($(A)$)--($(A)+(0, 1.5)$)--($(A)+(1.5,1.5)$)--($(A)+(1.5, 0)$)--cycle;
						\draw ($(A)+(0.75,0.75)$) node {\scalebox{0.8}{$x_1$}} ; 	
						\coordinate (A) at (3, 0);
						\draw[thick]	($(A)$)--($(A)+(0, 1.5)$)--($(A)+(1.5,1.5)$)--($(A)+(1.5, 0)$)--cycle;
						\draw ($(A)+(0.75,0.75)$) node {\scalebox{0.8}{$\d x_i$}} ; 	
						\coordinate (A) at (6, 0);
						\draw[thick]	($(A)$)--($(A)+(0, 1.5)$)--($(A)+(1.5,1.5)$)--($(A)+(1.5, 0)$)--cycle;
						\draw ($(A)+(0.75,0.75)$) node {\scalebox{0.8}{$x_n$}} ; 	
						\node at (2.3, 0.7) {\scalebox{0.8}{$\cdots$}} ; 	
						\node at (5.3, 0.7) {\scalebox{0.8}{$\cdots$}} ; 		
			\end{tikzpicture}}}\\
			&=
			\sum_{i=1}^n
			\vcenter{\hbox{\begin{tikzpicture}[scale=0.35]
						\coordinate (A) at (0, 0);
						\draw[thick]	($(A)$)--($(A)+(0, 1.5)$)--($(A)+(1.5,1.5)$)--($(A)+(1.5, 0)$)--cycle;
						\draw ($(A)+(0.75,0.75)$) node {\scalebox{0.8}{$x_1$}} ; 	
						\coordinate (A) at (3, 1.5);
						\draw[thick]	($(A)$)--($(A)+(0, 1.5)$)--($(A)+(1.5,1.5)$)--($(A)+(1.5, 0)$)--cycle;
						\draw ($(A)+(0.75,0.75)$) node {\scalebox{0.8}{$x_i$}} ; 	
						\coordinate (A) at (3, -1.5);
						\draw[thick]	($(A)$)--($(A)+(0, 1.5)$)--($(A)+(1.5,1.5)$)--($(A)+(1.5, 0)$)--cycle;
						\draw ($(A)+(0.75,0.75)$) node {\scalebox{0.8}{$\delta$}} ; 	
						\coordinate (A) at (6, 0);
						\draw[thick]	($(A)$)--($(A)+(0, 1.5)$)--($(A)+(1.5,1.5)$)--($(A)+(1.5, 0)$)--cycle;
						\draw ($(A)+(0.75,0.75)$) node {\scalebox{0.8}{$x_n$}} ; 	
						\node at (2.3, 0.7) {\scalebox{0.8}{$\cdots$}} ; 	
						\node at (5.3, 0.7) {\scalebox{0.8}{$\cdots$}} ; 
						\draw[thick] (3.75,0)--(3.75,1.5);		
			\end{tikzpicture}}}
			-
			\sum_{i=1}^n
			\vcenter{\hbox{\begin{tikzpicture}[scale=0.35]
						\coordinate (A) at (0, 0);
						\draw[thick]	($(A)$)--($(A)+(0, 1.5)$)--($(A)+(1.5,1.5)$)--($(A)+(1.5, 0)$)--cycle;
						\draw ($(A)+(0.75,0.75)$) node {\scalebox{0.8}{$x_1$}} ; 	
						\coordinate (A) at (3, 1.5);
						\draw[thick]	($(A)$)--($(A)+(0, 1.5)$)--($(A)+(1.5,1.5)$)--($(A)+(1.5, 0)$)--cycle;
						\draw ($(A)+(0.75,0.75)$) node {\scalebox{0.8}{$\delta$}} ; 	
						\coordinate (A) at (3, -1.5);
						\draw[thick]	($(A)$)--($(A)+(0, 1.5)$)--($(A)+(1.5,1.5)$)--($(A)+(1.5, 0)$)--cycle;
						\draw ($(A)+(0.75,0.75)$) node {\scalebox{0.8}{$x_i$}} ; 	
						\coordinate (A) at (6, 0);
						\draw[thick]	($(A)$)--($(A)+(0, 1.5)$)--($(A)+(1.5,1.5)$)--($(A)+(1.5, 0)$)--cycle;
						\draw ($(A)+(0.75,0.75)$) node {\scalebox{0.8}{$x_n$}} ; 	
						\node at (2.3, 0.7) {\scalebox{0.8}{$\cdots$}} ; 	
						\node at (5.3, 0.7) {\scalebox{0.8}{$\cdots$}} ; 
						\draw[thick] (3.75,0)--(3.75,1.5);		
			\end{tikzpicture}}}
		\end{align*}
		and
		\begin{align*}
			\d\left(
			\vcenter{\hbox{\begin{tikzpicture}[scale=0.35]
						\coordinate (A) at (0, 0);
						\draw[thick]	($(A)$)--($(A)+(0, 1.5)$)--($(A)+(1.5,1.5)$)--($(A)+(1.5, 0)$)--cycle;
						\draw ($(A)+(0.75,0.75)$) node {\scalebox{0.8}{$x_1$}} ; 	
						\coordinate (A) at (3, 0);
						\draw[thick]	($(A)$)--($(A)+(0, 1.5)$)--($(A)+(1.5,1.5)$)--($(A)+(1.5, 0)$)--cycle;
						\draw ($(A)+(0.75,0.75)$) node {\scalebox{0.8}{$x_i$}} ; 	
						\coordinate (A) at (6, 0);
						\draw[thick]	($(A)$)--($(A)+(0, 1.5)$)--($(A)+(1.5,1.5)$)--($(A)+(1.5, 0)$)--cycle;
						\draw ($(A)+(0.75,0.75)$) node {\scalebox{0.8}{$x_n$}} ; 	
						\node at (2.3, 0.7) {\scalebox{0.8}{$\cdots$}} ; 	
						\node at (5.3, 0.7) {\scalebox{0.8}{$\cdots$}} ; 		
			\end{tikzpicture}}}
			\right)\ = \ 
			\vcenter{\hbox{\begin{tikzpicture}[scale=0.35]
						\coordinate (A) at (0, 0);
						\draw[thick]	($(A)$)--($(A)+(0, 1.5)$)--($(A)+(1.5,1.5)$)--($(A)+(1.5, 0)$)--cycle;
						\draw ($(A)+(0.75,0.75)$) node {\scalebox{0.8}{$x_1$}} ; 	
						\coordinate (A) at (3, 0);
						\draw[thick]	($(A)$)--($(A)+(0, 1.5)$)--($(A)+(1.5,1.5)$)--($(A)+(1.5, 0)$)--cycle;
						\draw ($(A)+(0.75,0.75)$) node {\scalebox{0.8}{$x_i$}} ; 	
						\coordinate (A) at (6, 0);
						\draw[thick]	($(A)$)--($(A)+(0, 1.5)$)--($(A)+(1.5,1.5)$)--($(A)+(1.5, 0)$)--cycle;
						\draw ($(A)+(0.75,0.75)$) node {\scalebox{0.8}{$x_n$}} ; 	
						\node at (2.3, 0.7) {\scalebox{0.8}{$\cdots$}} ; 	
						\node at (5.3, 0.7) {\scalebox{0.8}{$\cdots$}} ; 		
						\coordinate (A) at (3, -3.5);
						\draw[thick]	($(A)$)--($(A)+(0, 1.5)$)--($(A)+(1.5,1.5)$)--($(A)+(1.5, 0)$)--cycle;
						\draw ($(A)+(0.75,0.75)$) node {\scalebox{0.8}{$\delta$}} ; 	
						\draw[thick] (-0.5,-0.5)--(8,-0.5)--(8,2)--(-0.5,2)--cycle;	
						\draw[thick] (3.75, -2)--(3.75,-0.5);
			\end{tikzpicture}}}
			\ -\ 
			\vcenter{\hbox{\begin{tikzpicture}[scale=0.35]
						\coordinate (A) at (0, 0);
						\draw[thick]	($(A)$)--($(A)+(0, 1.5)$)--($(A)+(1.5,1.5)$)--($(A)+(1.5, 0)$)--cycle;
						\draw ($(A)+(0.75,0.75)$) node {\scalebox{0.8}{$x_1$}} ; 	
						\coordinate (A) at (3, 0);
						\draw[thick]	($(A)$)--($(A)+(0, 1.5)$)--($(A)+(1.5,1.5)$)--($(A)+(1.5, 0)$)--cycle;
						\draw ($(A)+(0.75,0.75)$) node {\scalebox{0.8}{$x_i$}} ; 	
						\coordinate (A) at (6, 0);
						\draw[thick]	($(A)$)--($(A)+(0, 1.5)$)--($(A)+(1.5,1.5)$)--($(A)+(1.5, 0)$)--cycle;
						\draw ($(A)+(0.75,0.75)$) node {\scalebox{0.8}{$x_n$}} ; 	
						\node at (2.3, 0.7) {\scalebox{0.8}{$\cdots$}} ; 	
						\node at (5.3, 0.7) {\scalebox{0.8}{$\cdots$}} ; 		
						\coordinate (A) at (3, 3.5);
						\draw[thick]	($(A)$)--($(A)+(0, 1.5)$)--($(A)+(1.5,1.5)$)--($(A)+(1.5, 0)$)--cycle;
						\draw ($(A)+(0.75,0.75)$) node {\scalebox{0.8}{$\delta$}} ; 	
						\draw[thick] (-0.5,-0.5)--(8,-0.5)--(8,2)--(-0.5,2)--cycle;	
						\draw[thick] (3.75, 3.5)--(3.75,2);
			\end{tikzpicture}}} \ \ .
		\end{align*}
		The development of these last two terms produce graphs with one edge, which coincide with the ones on the line just above. By the induction hypothesis, the alternate sum of all the graphs with $k$ edges for $2\leqslant k \leqslant n-1$ vanish. It remains only 
		\begin{align*}
			\vcenter{\hbox{\begin{tikzpicture}[scale=0.35]
						\coordinate (A) at (0, 0);
						\draw[thick]	($(A)$)--($(A)+(0, 1.5)$)--($(A)+(1.5,1.5)$)--($(A)+(1.5, 0)$)--cycle;
						\draw ($(A)+(0.75,0.75)$) node {\scalebox{1}{$x_1$}} ; 	
						\coordinate (A) at (3, 0);
						\draw[thick]	($(A)$)--($(A)+(0, 1.5)$)--($(A)+(1.5,1.5)$)--($(A)+(1.5, 0)$)--cycle;
						\draw ($(A)+(0.75,0.75)$) node {\scalebox{1}{$x_i$}} ; 	
						\coordinate (A) at (6, 0);
						\draw[thick]	($(A)$)--($(A)+(0, 1.5)$)--($(A)+(1.5,1.5)$)--($(A)+(1.5, 0)$)--cycle;
						\draw ($(A)+(0.75,0.75)$) node {\scalebox{1}{$x_n$}} ; 	
						\node at (2.3, 0.7) {\scalebox{0.8}{$\cdots$}} ; 	
						\node at (5.3, 0.7) {\scalebox{0.8}{$\cdots$}} ; 		
						\coordinate (A) at (3, -3);
						\draw[thick]	($(A)$)--($(A)+(0, 1.5)$)--($(A)+(1.5,1.5)$)--($(A)+(1.5, 0)$)--cycle;
						\draw ($(A)+(0.75,0.75)$) node {\scalebox{1}{$\delta$}} ; 	
						\draw[thick] (3.5, -1.5)--(0.75,0);	
						\draw[thick] (3.75, -1.5)--(3.75,0);
						\draw[thick] (4, -1.5)--(6.75,0);		
			\end{tikzpicture}}}
			=
			\vcenter{\hbox{\begin{tikzpicture}[scale=0.35]
						\coordinate (A) at (0, 0);
						\draw[thick]	($(A)$)--($(A)+(0, 1.5)$)--($(A)+(1.5,1.5)$)--($(A)+(1.5, 0)$)--cycle;
						\draw ($(A)+(0.75,0.75)$) node {\scalebox{1}{$x_1$}} ; 	
						\coordinate (A) at (3, 0);
						\draw[thick]	($(A)$)--($(A)+(0, 1.5)$)--($(A)+(1.5,1.5)$)--($(A)+(1.5, 0)$)--cycle;
						\draw ($(A)+(0.75,0.75)$) node {\scalebox{1}{$x_i$}} ; 	
						\coordinate (A) at (6, 0);
						\draw[thick]	($(A)$)--($(A)+(0, 1.5)$)--($(A)+(1.5,1.5)$)--($(A)+(1.5, 0)$)--cycle;
						\draw ($(A)+(0.75,0.75)$) node {\scalebox{1}{$x_n$}} ; 	
						\node at (2.3, 0.7) {\scalebox{0.8}{$\cdots$}} ; 	
						\node at (5.3, 0.7) {\scalebox{0.8}{$\cdots$}} ; 		
						\coordinate (A) at (3, 3);
						\draw[thick]	($(A)$)--($(A)+(0, 1.5)$)--($(A)+(1.5,1.5)$)--($(A)+(1.5, 0)$)--cycle;
						\draw ($(A)+(0.75,0.75)$) node {\scalebox{1}{$\delta$}} ; 		
						\draw[thick] (3.5, 3)--(0.75,1.5);
						\draw[thick] (3.75, 3)--(3.75,1.5);
						\draw[thick] (4, 3)--(6.75,1.5);		
			\end{tikzpicture}}}~.
		\end{align*}
	}
\medskip 
	
Finally, the symmetric formula 
\begin{align*}
\exp(\lambda)\cdot \ba=
		\pap{\exp(\lambda)}{\ba}{\exp(-\lambda)}
	+\exp(\lambda)   \cc \left(\exp(-\lambda); \d \big(\exp(-\lambda)\big)\right)
\end{align*}
can either been proved by the same arguments applied from the beginning in a symmetrical way or as follows. 
We apply the differential $\d$ to the element 
\[\exp(\lambda)   \cc \exp(-\lambda)=\1~, \]
which gives 
\[\left(\exp(\lambda); \d \big(\exp(\lambda)\big)\right)\cc \exp(-\lambda)+\exp(\lambda)   \cc \left(\exp(-\lambda); \d \big(\exp(-\lambda)\big)\right)=0~.\]
\end{proof}

\begin{remark}\leavevmode
	If we restrict the first formula of \cref{thm:Action} to rooted trees, we recover precisely the formula $(\exp(\lambda)\star \alpha) \circledcirc \exp(-\lambda)$ obtained in \cite{DotsenkoShadrinVallette16} for pre-Lie algebras.
\end{remark}

\appendix

\section{Proof of Proposition \ref{prop:nofinitepresentation}}\label{AppA}

\cref{prop:nofinitepresentation} states that the operad $\Liegra$ is not finitely generated. 
Recall that a presentation by generators and relations of an operad $\P$ is an isomorphism of the form $\P\cong\mathcal T(E)/(R)$~. The operad $\P$ is said to be \emph{finitely generated} if the generating space $E$ can be chosen to be of total finite dimension. 
Since we are dealing with the operad $\Liegra$ that is trivial in arities $0$ and $1$ and is arity-wise finite dimensional, the only potential source of infinite dimensionality phenomena is that $E(n)$ stays non-trivial as $n$ goes to infinity. Concretely, we will use the following two general propositions.

\begin{proposition}\label{prop:truncated means f.g.}
	Let $\P$ be an operad such that $\dim \P(n)<+\infty$~, for all $n$.  Suppose that $\P$ can be generated by an 
	$\Sy$-module $E$ for which there exists $N\in \NN$ such that $E(n)=0$~, for $n> N$~. Then, the operad $\P$ is finitely generated.
\end{proposition}

\begin{proof}
	We can take as a new $\Sy$-module of generators $F$ defined by  $F(n) \coloneq \P(n)$,  for $n\leqslant N$, and  $F(n)\coloneq0$ otherwise. 
	Using the isomorphism $\P\cong \mathcal T(E)/(R)$, we consider the canonical map of $\Sy$-modules $g \colon F \to 
	\mathcal{T}(E)/(R)$ which is an isomorphism in arities $n\leqslant N$~. The morphism of $\Sy$-modules defined by 
	\[f(n)\ \colon E(n) \to \mathcal{T}(E)(n) \twoheadrightarrow \mathcal{T}(E)/(R)(n) \xrightarrow{g(n)^{-1}}F(n) \]
	for $n\leqslant N$, and $f(n)\coloneq 0$ otherwise, induces a morphism of operads 
	$\mathcal{T}(f)\ \colon \mathcal{T}(E) \to \mathcal{T}(F)$ which satisfies the commutative diagram 
\begin{center}
\begin{tikzcd}[column sep=large]
\mathcal{T}(E) \arrow[r,"\mathcal{T}(f)"]  \arrow[d, two heads]& \mathcal{T}(F)  \arrow[dl]\\
\mathcal{T}(E)/(R) &  F~. \arrow[l, "g"] \arrow[u]
\end{tikzcd}
\end{center}
	This shows that the morphism of operads $\mathcal{T}(F) \twoheadrightarrow \P$ is surjective. 
\end{proof}

\begin{proposition}\label{prop:get rid of arities 0 and 1}
	Let $\P=\mathcal T(E)/(R)$ be a finitely generated presentation of an operad $\P$ such that $\P(0)=0$ and $\P(1)=\k I$~. Then, there exists a finitely generated presentation $\P\cong\mathcal T(F)/(S)$ with $F(0)=F(1)=0$~.
\end{proposition}
\begin{proof}
The arguments are the same as above: we consider the generating $\Sy$-module $F$ defined 
 to be the truncation of $E$ to arities $n \geqslant 2$. 
One readily checks that the composite 
	\[
	\mathcal T (E') \to \mathcal T(E) \twoheadrightarrow {\mathcal T(E)}/{(R)}\]
	is surjective. 
\end{proof}

Thanks to the previous propositions, it suffices to show the following variant of Proposition \ref{prop:nofinitepresentation}.

\begin{proposition}\label{prop:bounded}
Suppose that there exists a surjective morphism of operads $\mathcal T(E) \twoheadrightarrow \Liegra$, with $E$ an  $\Sy$-module 
concentrated in arities $\geqslant 2$. Then, there exist arbitrarily large $n$ such that $E(n)\ne 0$~.
\end{proposition}

\begin{proof}[Proof of {\cref{prop:nofinitepresentation}}]
Recall the operad $\Liegra$ satisfies $\Liegra(0)=0$, $\Liegra(1)=\k I$. Furthermore  $\dim \Liegra(n)\leq |\text{oriented simple graphs on }n \text{ vertices}| = 3^{\binom{n}{2}}<+\infty$~, for all $n\geqslant 1$.
Suppose that the operad $\Liegra$ is finitely generated. By \cref{prop:get rid of arities 0 and 1}, we can suppose that the 
generating $\Sy$-module $E$ of total finite dimension 
 is trivial in arities $0$ and $1$. Then \cref{prop:bounded} ensures that 
there exist arbitrarily large $n$ such that $E(n)\ne 0$~, which violates the initial finiteness assumption.
\end{proof}

We start by estimating a lower bound for the dimension of $\Liegra(n)$, which is to say, the cardinality of $\dcGra(n)$. 

\begin{lemma}\label{lem:dim of LieGra}
There exists a constant $c>1$ such that $|\dcGra(n)| \geqslant c^{n^2}$, for all sufficiently large $n$.
\end{lemma}

\begin{proof}
	Consider the linear graph 
	
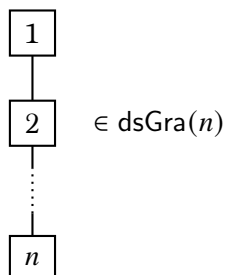
\begin{equation*}
\begin{tikzpicture}[scale=0.6,baseline=1.4cm]
	\draw[thick]	(0,-1)--(1,-1)--(1,0)--(0,0)--cycle;
	\draw[thick]	(0,2)--(1,2)--(1,3)--(0,3)--cycle;
	\draw[thick]	(0,4)--(1,4)--(1,5)--(0,5)--cycle;
	\draw[thick,dotted]	(0.5,0.5)--(0.5,1.5);
	\draw[thick]	(0.5,2)--(0.5,1.5);	
	\draw[thick]	(0.5,0.5)--(0.5,0);		
	\draw[thick]	(0.5,3)--(0.5,4);	
	\draw (0.5,-0.5) node {{$n$}} ; 
	\draw (0.5,2.5) node {{$2$}} ; 
	\draw (0.5,4.5) node {{$1$}} ; 
\end{tikzpicture}\quad \in \dcGra(n)
	\end{equation*}
and notice that we can add any single edge from the vertex $i$ to $j$, for $i\leqslant j-2$, and produce a valid graph in $\dcGra(n)$. Since there are $\binom{n}{2} - n+1$ such possible edges, we have 
\[|\dcGra(n)| > 2^{\binom{n}{2} - n+1}~,\] 
which is eventually bigger than $c^{n^2}$, for $c < \sqrt{2}$~.	
\end{proof}

\begin{proof}[Proof of Proposition \ref{prop:bounded}]
	The strategy of the proof is a dimension argument. Assume by contradiction that there exists a generating space $E$ and $N> 0$ such that $E(n)=0$, for $n> N$. 
	By \cref{prop:truncated means f.g.}, we can further assume that $E(n) =\Liegra(n)$, for any $n\leqslant N$. 
	We will show that for $n$ large enough $\dim \mathcal T(E)(n)$ will be smaller than $\dim \Liegra(n)$, thus establishing the result. 
	
\medskip	

Since $E(0)=0$, a basis of the free operad on $E$ is parameterized by shuffle trees: concretely,
	\[
	\mathcal T(E)(n) \cong \bigoplus_{t\in \mathsf{ShTr}(n)} E^{\otimes |t|},
	\]
where $\mathsf{ShTr}(n)$ denotes the set of shuffle trees with $n$ leaves and where $|t|$ is the number of  the vertices of $t$.
	 Let $b \coloneq \max\{\dim \Liegra(n)\ |\ n\leqslant  N\}$ and such maximum is attained in arity $k$. Denoting 
	 ${\mathsf{ShTr}^{\leqslant N}(n)}$ the set of shuffle trees in which all vertices have arity between $2$ and $N$, we deduce
\begin{equation}\label{eq:dim T(E)}
		\dim 	\mathcal T(E)(n) \leqslant \dim \bigoplus_{\mathsf{ShTr}^{\leqslant N}(n)} \Liegra(k)^{\otimes n-1} \leqslant 
		\left|\mathsf{ShTr}^{\leqslant N}(n)\right|\times  b^{n-1}
\end{equation}
from $E(n)=0$, for $n> N$.   

\medskip

To show that this value is eventually smaller than the one from Lemma \ref{lem:dim of LieGra}, we need to find an upper bound for the cardinal of ${\mathsf{ShTr}^{\leqslant N}(n)}$. Recall that a shuffle tree is a rooted planar tree together with a permutation of $\{1,\dots,n\}$ satisfying some condition.
The number of rooted planar trees with $n$ leaves 
and with vertices of arity greater or equal to 2
is given by the Schr\"oder--Hipparchus numbers $s_n$, which satisfy the following recurrence formula 
\[
s_n=\frac{1}{n}\left((6n-9)s_{n-1}-(n-3)s_{n-2}\right)~,
\]
see \cite{Stanleyplanartrees}. By induction, we get that $s_n\leqslant 6^n$, for all $n$. 
The permutations indexing a shuffle tree are a certain subset of $\mathbb S_n$, so we the  following upper bound 
\[		
\dim 	\mathcal T(E)(n) \leqslant 6^n \times n! \times b^{n-1}~. 
\]

 Stirling's formula provides us with the following asymptotical estimation $n!\sim \sqrt{2\pi n} \left(\frac{n}{e}\right)^n$.
It follows that 
\begin{equation}
	\dim 	\mathcal T(E)(n) \leqslant  6^n \times  \sqrt{2\pi n} \left(\frac{n}{e}\right)^n\times  b^{n} \leqslant (an)^{n+\frac{1}{2}}~, 
\end{equation}
for large enough n and  for some constant $a>1$.
Taking logarithms, we see that the estimate $c^{n^2}$ from Lemma \ref{lem:dim of LieGra} eventually overtakes  $(an)^{n+\frac{1}{2}}$ which gives us the desired contradiction.
\end{proof}

\begin{remark}
	The proof above by dimension counting boils down to the fact that while trees grow at most like $n^n$, the present operad of graphs grows at least like $c^{n^2}$. The same kind of argument shows that other types of graph operads appearing in the literature cannot be finitely generated.
\end{remark}

\bibliographystyle{alpha}
\bibliography{bib}

\end{document}